%% file: markov.tex
\documentclass[11pt]{article} 
\input{math}
\input{common}

\input{symbols}

\usepackage[style=numeric-comp,sorting=none,giveninits=true]{biblatex}
\usepackage[margin=1in]{geometry}
\usepackage{stmaryrd}

\usepackage{changepage}

\DeclareFieldFormat{sentencecase}{\MakeSentenceCase{#1}}

\usepackage{authblk}
\usepackage{orcidlink}
\author[1]{Mark Fornace\,\orcidlink{0000-0002-5829-5839}}
\author[1,2]{Michael Lindsey\,\orcidlink{0000-0002-4508-9454}}
\affil[1]{Lawrence Berkeley National Laboratory}
\affil[2]{University of California, Berkeley}

\title{An approximation theory for Markov chain compression}

\begin{document}



\maketitle

\begin{abstract}
We develop a framework for the compression of reversible Markov chains with rigorous error control. Given a subset of selected states, we construct reduced dynamics that can be lifted to an approximation of the full dynamics, and we prove simple spectral and nuclear norm bounds on the recovery error in terms of a suitably interpreted Nystr\"{o}m approximation error. We introduce two compression schemes: a projective compression based on committor functions and a structure-preserving compression defined in terms of an induced Markov chain over the selected states. The Nystr\"{o}m error appearing in our bounds can be controlled using recent results on column subset selection by nuclear maximization. Numerical experiments validate our theory and demonstrate the scalability of our approach.
\end{abstract}

\section{Introduction}
\setcounter{page}{1}

Markov processes \cite{Aldous1995-reversible,Rogers2000-diffusions,Kemeny1976-markov} play a fundamental role across a broad range scientific application areas \cite{Frenkel2023-understanding,Szekely2014-stochastic,Durr2011-lattice}. From a computational point of view, a key difficulty in understanding Markov processes is imposed by the size of the state space, which may be discrete and combinatorially large or continuous and high-dimensional. However, even over such large state spaces, in many scientific contexts the important dynamics over relevant timescales are dominated by only a few slow modes, while the vast majority of modes relax over extremely short timescales. To take advantage of this perspective, Markov state models (MSMs) \cite{Pande2010-everything,Husic2018-markov} have been advanced as a framework for distilling Markovian dynamics into an effective Markov process over only a few ``macrostates'' \cite{Chung2000-discrete,Aldous1995-reversible}. These may be interpretable as a coarse-graining of the original large state space, potentially corresponding to basins of attraction for metastable equilibria. Although the state space is reduced, the MSM can still capture the dynamics over physically relevant timescales \cite{Pande2010-everything}.

Widely-used practical approaches to such Markovian coarse-graining \cite{Pande2010-everything,Park2006-validation,Junghare2023-markov,Junghare2023-markov,Chodera2014-markov,Pande2010-everything,Shukla2015-markov,Schwantes2013-improvements,Husic2018-markov,Park2006-validation,McGibbon2015-variational,Chatterjee2015-uncertainty,Thiede2019-galerkin} generally rely on clustering and exhaustive repeated simulation in order to sample the complete space of interest. It is a significant problem to design algorithms for this task that are automatic and computationally efficient, while at the same time supporting rigorous error control.

We tackle this problem for reversible Markov chains by formulating a notion of reconstruction error relative to the exact Markovian dynamics $P(t) := e^{-Lt}$, where $L$ is a positive semidefinite ``Laplacian'' obtained from the rate matrix, via reversibility, by diagonal similarity transformation. In order to formulate our theory, we focus on the setting of discrete state spaces, though we comment in \Cref{s:contributions} on the applicability of our theory to 
continuous state spaces.  To the best of our knowledge, our algorithms and accompanying error analysis offer the first approximation theory for Markov process compression capable of rigorously quantifying  the reconstruction error for $P(t)$, measured in both the spectral norm and the nuclear or trace norm.

Interestingly, our approximation theory quantifies this reconstruction error in terms of the \Nystrom{} approximation error of $L^{-1}$, suitably interpreted in the sense of \cite{Fornace2024-column} to account for the null vector of $L$ due to the stationary distribution. This connects our approximation theory with the theory of the column subset selection problem (CSSP) \cite{Deshpande2006-adaptive,Belabbas2009-spectral,Gu2004-strong,Steinerberger2024-randomly,Civril2009-selecting,Williams2000-using,Derezinski2021-determinantal,Guruswami2012-optimal,Cortinovis2024-adaptive,Chaturantabut2010-nonlinear,Mahoney2009-cur}, which typically targets the nuclear norm error for \Nystrom{} approximation \cite{Fornace2024-column,Chen2022-randomly,Zhang2025-accurate}. In particular, the recently introduced nuclear maximization algorithm for Laplacian reduction \cite{Fornace2024-column} is known to perform nearly as well as the \emph{combinatorially optimal} subset selection, which is striking since CSSP is NP-complete in general \cite{Shitov2021-column}. The theory of this paper, together with the nuclear maximization approach to CSSP, then offers practical algorithms for Markov chain compression, enjoying almost-linear scaling \cite{Fornace2024-column} with respect to the number of edges in the Markov chain Laplacian.

\subsection{Contributions} \label{s:contributions}
We outline our technical contributions by highlighting our main theorems and their interpretations.

First in \Cref{s:compression}, we consider the setting of a general $n \times n$ positive definite matrix $L$ without any Laplacian structure. For a selected subset $\I \subset [n] := \{1,\ldots, n\}$ of indices, we also define the compression $P_{\I} (t) = V e^{-(V^\t L V) t} V^\t$ 
where $\Eps (\I)$ is the error of the \Nystrom{} approximation of the kernel $K := L^{-1}$ furnished by the selected subset $\I$, measured in the spectral and nuclear norms, and $V = \mathrm{orth}(K_{:,\I})$. Then we can bound the compression error as follows:

\begin{restatable}[Error bound for compression of dissipative dynamics \ProofLink{th:dissipative}]{theorem}{ThDissipative} \label{th:dissipative}
    For positive definite $L$, nonempty $\I$, and $t>0$:
    \begin{eqn}
        \Norms{\SymP(t) - \OrthP_{\I}(t)} \leq \frac{3 \sqrt{3}}{2 \pi} \frac{\Eps(\I)}{t}.
    \end{eqn}
\end{restatable}

\noindent Moreover, Theorem \hyperlink{link.Theorem1A}{1A} and Theorem \hyperlink{link.Theorem1B}{1B} extend this theorem to more general choices of $V$ not obtained by \Nystrom{} approximation, as well as to the case of singular $L$.

The result of \Cref{th:dissipative} can be intuitively understood
as follows. For simplicity, we will directly address the intuition
of the spectral norm bound, though the nuclear norm bound can be understood
similarly.
First, note that since $\Vert \SymP(t)\Vert_{2},\Vert \OrthP_{\I}(t)\Vert_{2}\leq1$,
we automatically have a trivial bound $\Vert \SymP(t)-\OrthP_{\I}(t)\Vert_{2}\leq2$
which supersedes the bound of \Cref{th:dissipative} in the
$t\ra0$ limit. Thus the theorem has no content in this limit. 
This is not an artifact of the proof: in fact the approximation error is simply not controlled by $\Es=\Es(\I)$ in the regime of very small $t$, when fast modes not captured in the approximation are yet to relax to equilibrium.

The error is only guaranteed to become small once $t\gg\Es$,
though $\Es$ may indeed be quite small if our \Nystrom{} approximation
is successful. We can think of $\Es$ as the slowest timescale \emph{not} captured in the \Nystrom{} approximation of $K$,
noting that slow timescales in the dynamics (\ref{eq:dynamics}) correspond to large eigenvalues in $K$, which are prioritized by algorithms for \Nystrom{} approximation. We remind the reader that for many systems of interest,
the number of slow timescales, which typically correspond to rare
transition events among metastable equilibria, is quite small compared
to the number of fast timescales.

In summary, the reduced matrix $\OrthP_{\I}(t)$ accurately reproduces the
dynamics on all but the fastest timescales, according to which there
may be a short transient regime in which $\OrthP_{\I}(t)$ disagrees with
$\SymP(t)$.
Finally, observe that for very large times $t\gg \Vert K \Vert_2$, both $\SymP(t)$ and $\OrthP_{\I}(t)$ become exponentially small, so our bound
is once again superseded in this regime. 
However, this regime kicks
in precisely once all the interesting dynamics have already been completed.
We should think of $\Vert K \Vert_2$ (the slowest timescale) as
being arbitrarily large, and we should avoid using $ \Vert K \Vert_2 $
in our analysis.

In \Cref{s:projective}, we move to the setting of Markov dynamics, where $L$ is a positive semidefinite Laplacian in a suitable sense that we define, associated with an arbitrary irreducible reversible Markov chain over the state space $[n]$. In this setting we introduce a suitable generalization of the notions of \Nystrom{} approximation error $\Eps (\I)$, as well as a suitable generalization of $P_{\I} (t)$ as defined above, to accommodate the nontrivial null space of $L$. However, the Markovian structure of the problem can be used to write alternatively $P_{\I} (t) = C e^{-(C^+ L C) t} C^+$, where $C$ is a suitably defined ``committor matrix'' (cf. \Cref{def:killed-committor}) associated to the selected states $\I$. In this context, we refer to $P_{\I} (t)$ as the \emph{projective compression} of the exact dynamics $P(t)$. In terms of these objects, we can phrase an approximation theorem which looks almost identical to \Cref{th:dissipative}:

\begin{restatable}[Error bound for projective Markov chain compression \ProofLink{th:ortho-markov}]{theorem}{ThOrthoMarkov} \label{th:ortho-markov}
    For a Laplacian $L$, nonempty $\I$, and $t > 0$: 
    \begin{eqn} \label{eq:ortho-markov-norm-bound}
    \Norms{\SymP(t)-\OrthP_{\I}(t)} \leq \frac{3\sqrt{3}}{2\pi}\,\frac{\Eps (\I)}{t}.
    \end{eqn}
\end{restatable}

In \Cref{s:structure-preserving}, we consider an alternative compression $\ObliqueP_{\I}(t)$ of the Markovian dynamics $P(t)$ that is \emph{structure-preserving} in the sense that (1) it can be interpreted in terms of a Markov process over $\I$, governed by a Laplacian $\Lr$, and (2) $\ObliqueP_{\I}(t)$ is entrywise nonnegative, corresponding via similarity transformation to a transition probability matrix that is row stochastic. We refer to the Markov process over $\I$ as the \emph{induced chain}, for which we offer several probabilistic interpretations and constructions. Ultimately, the structure-preserving compression admits the simple formula $\ObliqueP_{\I}(t) = \Co \expt{\Lr} \Co^\t$.

Although the formula for $\ObliqueP_{\I}(t)$ appears similar to that of $P_{\I} (t)$, the analysis of its error is significantly more delicate and much more deeply rooted in probabilistic perspectives. In particular, for our analysis we must construct a Markov chain over an augmented state space, based on marking the original Markov process with a label according to the last state visited among those in $\I$. We call this the \emph{marked chain} by analogy to the marked point process concept in the literature \cite{Jacobsen2005-point,Aldous1995-reversible}. It offers a useful perspective both on the construction of our structure-preserving compression as well as its error analysis. This analysis furnishes the following theorem:

\begin{restatable}[Error bound for structure-preserving Markov chain compression \ProofLink{th:sp-nuclear-bound}]{theorem}{ThSpNuclearBound} \label{th:sp-nuclear-bound}
    For a Laplacian $L$, nonempty $\I$, and $t>0$:
    \begin{eqn} \label{eq:thm3bound}
        \nnorm{\ObliqueP_{\I}(t) - \SymP(t)} \leq \lrp{\frac{3 \sqrt{3}}{2\pi} + \Ia \frac{2}{\pi}} \frac{\En (\I) }{t}.
    \end{eqn}
\end{restatable}

\booltrue{STATING}

\Cref{th:sp-nuclear-bound} successfully bounds the nuclear norm approximation error in terms of $\ve_* (\I)$, which can in turn be controlled \emph{a priori} via the aforementioned results on column selection via nuclear maximization \cite{Fornace2024-column}. However, our path to proving \Cref{th:sp-nuclear-bound}  also yields Theorem \hyperlink{Theorem3*}{3*}, which expresses both the spectral and nuclear norm approximation errors in terms of an intermediate quantity that can be calculated \emph{a posteriori} for any chosen subset $\I$. We observe in practice that the bounds furnished by Theorem \hyperlink{Theorem3*}{3*} are approximately tight and moreover quantitatively similar to those of \Cref{th:ortho-markov}.

In \Cref{s:algs}, we explain how the nuclear maximization algorithm can be used to compute a subset $\I$ with guaranteed bounds on $\ve_{*} (\I)$. Via the aforementioned  \Cref{th:dissipative,th:ortho-markov,th:sp-nuclear-bound}, we can in turn use these bounds to control the error in our compressions of $P(t)$. In particular, Theorem \hyperlink{Theorem4A}{4A}, restated from \cite{Fornace2024-column}, shows that the subset $\I$ furnished by nuclear maximization is nearly optimal. Moreover, Theorem \hyperlink{Theorem4B}{4B}, which we prove in this work as a slight extension of \cite{Fornace2024-column}, offers an \emph{a priori} bound on $\ve_{*} (\I)$ in terms of the spectrum of $K = L^+$. 

Overall, our analysis involves a combination of techniques from complex analysis and linear algebra, as well as probabilistic techniques. In addition to their role in the analysis itself, our constructions and characterizations of the induced and marked chains may be of independent interest, as they convey novel intuitions about the construction of Markov state models.

In future work, we plan to extend this work to continuous state spaces, for example, via an initial basis projection step, leveraging Theorem \hyperlink{link.Theorem1B}{1B} to reduce the problem to a discrete and possibly large but enumerable state space. Further reduction via the \Nystrom{} techniques and analysis of this paper may then be employed.

\subsection{Related work}

We are motivated by the widespread interest in Markov state models (MSMs) in computational chemistry applications. Here we review some of the existing methods and available analysis. First, \cite{Sarich2010-approximation} states a bound of the approximation error of an MSM within the reduced subspace in the long-time limit, while later work by some of the same authors analyzes the approximation of an MSM's slowest eigenvalues \cite{Djurdjevac2012-estimating}.
Other authors \cite{Hummer2015-optimal,Kells2020-correlation} constructed reduced kinetic models based on mean first passage times within the original dynamics, the latter yielding a variational formulation based on the Kemeny constant of the Markov chain.
Additionally, \cite{Schutte2011-markov} and \cite{Berezhkovskii2019-committors} include a number of results related to the induced chain and marked chain that we define in this work, although the approximation strategies are different. 
In \cite{Yao2013-hierarchical}, a hierarchical MSM construction based on spectral clustering is introduced.
Other researchers have focused on efficient algorithms for Markov chain ``lumping'' (e.g., \cite{Valmari2010-simple,Derisavi2003-optimal}) on discrete spaces.
These methods seek to partition the state space into clusters such that the outgoing transition probabilities are approximately constant on each cluster.

We also review some related computational methods for characterizing the long-timescale properties of Markov chains without explicitly constructing MSMs. 
Galerkin projection methods \cite{Thiede2019-galerkin} use a basis set to form a low-rank approximation, not necessarily interpretable as a Markov state model. 
The variational approach to conformational dynamics (VAC) \cite{Webber2021-error,Lorpaiboon2020-integrated} estimates the slowest eigenmodes of a Markov transition operator, constructed by measurement of an autocorrelation matrix at a fixed lag time $\tau$. Up to estimation errors due to stochastic simulation, VAC is equivalent to the Rayleigh-Ritz method.
\cite{Webber2021-error} analyzes the errors due to basis projection and stochastic simulation, with a focus on the large $\tau$ regime. 
Meanwhile, milestoning methods \cite{Faradjian2004-computing,Vanden-Eijnden2008-assumptions,Elber2021-modeling} seek to accelerate or analyze trajectory sampling by recording transition times and probabilities between selected conformations.
Relevant theoretical efforts (e.g., \cite{Schutte2011-markov,Lin2018-mathematical}) involve committor analysis and marking constructions similar to those considered in this work.

To the best of our knowledge, no previous work bounds the recovery error for the full dynamics $P(t) = e^{-Lt}$. Notably, our bounds do not depend any notion of spectral gap. Recall that MSMs are often based on the intuition of metastable macrostates, for which the timescales of mixing within a macrostate are much faster than those of mixing between macrostates.
One of the advantages of our theory is that the approximation error bounds can be stated simply through the lens of low-rank approximation, without involving more complicated formulations of metastability or scale separation. 
Indeed, note that augmentation of the column subset can only decrease the \Nystrom{} error, regardless of how poorly conditioned the columns may be.

Perhaps the work most closely related to our perspective concerns the approximation of a sum of exponential decays by a single exponential function \cite{Keilson1979-rarity,Brown1983-approximating}.
Such results can be used to show that the exit time distribution of a Markov chain starting in an approximately metastable set is approximately exponential  \cite{Aldous1992-inequalities,Aldous1995-reversible}.

\subsection{Outline}

Here is an outline of the paper. Refer above to \Cref{s:contributions} for further context.

\noindent In \textbf{\Cref{s:compression}}, we consider the case of a general positive definite $L$.

\begin{adjustwidth}{2em}{0em}
\begin{description}
    \item [\Cref{s:compnystrom}:] We introduce the notion of compression via \Nystrom{} approximation.
    \item [\Cref{s:contour}:] We prove \Cref{th:dissipative} on the compression error, offering a sketch in the main text that illustrates the contour integration argument that underlies the proofs of all of the main theorems.
    \item [\Cref{s:non-nystrom}:] We prove extensions (Theorem \hyperlink{link.Theorem1A}{1A} and Theorem \hyperlink{link.Theorem1B}{1B}), generalizing to non-\Nystrom{} approximation and allowing $L$ to be singular.
\end{description}
\end{adjustwidth}

\noindent In \textbf{\Cref{s:projective}}, we consider the \emph{projective} compression of $P(t)$ where $L$ is obtained from a Markov process.
\begin{adjustwidth}{2em}{0em}
\begin{description}
    \item [\Cref{s:symmetrization}:] We explain how a positive semidefinite $L$, which we call a Laplacian in a sense that we define, can be obtained from a reversible Markov chain.
    \item [\Cref{s:killedchain}:] We define the ``killed'' Markov chain, which allows us to view $L$ as a limit of strictly positive definite killed Laplacians $L_\Kill$ as $\gamma \ra 0$.
    \item [\Cref{s:committor}:] We define the committor functions induced by the selection $\I$ and explain how they arise from the limit of \Nystrom{} approximations in the $\gamma \ra 0$ limit.
    \item [\Cref{s:projapproxbounds}:] We prove \Cref{th:ortho-markov} on the projective compression error.
\end{description}
\end{adjustwidth}

\noindent In \textbf{\Cref{s:structure-preserving}}, we consider the \emph{structure-preserving} compression of the Markov process.
\begin{adjustwidth}{2em}{0em}
\begin{description}
    \item [\Cref{s:induced}:] We define the induced chain over the selected states $\I$ and offer several probabilistic interpretations.
    \item [\Cref{s:lifting}:] We explain how the induced chain dynamics are lifted to a structure-preserving approximation of $P(t)$ on the full state space.
    \item [\Cref{s:marked}:] We construct the marked chain, prove several of its key properties, and show how the induced chain can be constructed as a projection of the marked chain.
    \item [\Cref{s:oblique-approximation}:] We use the marked chain projection to prove our approximation results (\Cref{th:sp-nuclear-bound} and Theorem \hyperlink{Theorem3*}{3*}) for the structure-preserving approximation.
\end{description}
\end{adjustwidth}

\noindent In \textbf{\Cref{s:algs}}, we show how the nuclear norm \Nystrom{} error $\ve_{*} (\I)$ can be controlled by choosing the subset $\I$ by nuclear maximization \cite{Fornace2024-column}. Our results are summarized in Theorem \hyperlink{Theorem4A}{4A} and \hyperlink{Theorem4B}{4B}.

\noindent In \textbf{\Cref{s:experiments}}, we present numerical experiments validating our approximation theory on sample problems from nucleic acid secondary structure modeling \cite{Fornace2024-column,Fornace2020-unified,Schaeffer2015-stochastic,Flamm2000-rna} and webgraph analysis \cite{Rozemberczki2021-multi-scale}.

\subsection{Notation}\label{sec:notation}

Throughout the paper, $L$ is a symmetric positive semidefinite $n \times n$ operator and $K \Eq L^+$.
In \Cref{s:compression}, we assume $L$ to be positive definite (so that $K=L^\Mo$). In \Cref{s:projective,s:structure-preserving}, where $L$ can be viewed as a graph Laplacian associated to an arbitrary Markov chain, we allow $L$ to have a zero eigenvalue corresponding to the Markov chain stationary distribution.
In the Markov chain analysis, we consider a killed operator $\La \Eq L+\Kill \Id$ and its inverse $\Ka \Eq \La^\Mo$ for $\gamma > 0$. Much of our analysis will consider the $\gamma \ra 0$ limit, but note carefully that $K_\gamma$ does not converge to $K$ in this limit.

Throughout this work, we consider only stationary and irreducible continuous-time Markov chains.
Within this scope, we consider the approximation of a reversible Markov chain $\FullChain$ over a state space $\Indices \Eq \lrb{1, \dots, n}$. Depending on the context, we identify $\FullChain$ with a stochastic process $\Process{X}$ initialized with distribution $X_0 \sim \pi$ where $\pi$ is the stationary distribution, or alternatively with the rate matrix $R$, which can be viewed as an equivalent specification of this process.
For a general set $\mathcal{S} \subset \Indices$, we also define the stopping time
\begin{eqn} \label{eq:stoppingtime}
\tau_{\,\mathcal{S}} := \min \{ t : t \geq 0, \, X_t \in \mathcal{S} \}.
\end{eqn}
Finally, we define $h \Eq \pi^{1/2}$, the entrywise square root of $\pi$. 

We use $\FullP(t)$ for the probability evolution operator of $\FullChain$, i.e., $\FullP_{i,j}(t)$ is the probability that the chain starting in state $i$ at time 0 occupies state $j$ at time $t$.
We first consider a \emph{projective compression} $\NonSymOrthP_{\I}(t)$ of the dynamics based on \Nystrom{} approximation with selected indices $\I \subset \Indices$, so that $\NonSymOrthP_{\I}(t) \approx \FullP(t)$. We generally assume $\I$ to be non-empty with a non-empty set complement $\Ic \Eq \Indices \setminus \I$.

In order to conserve a probabilistic interpretation and positivity constraints in our compression, we also introduce a \emph{structure-preserving compression} of the dynamics, denoted $\NonSymObliqueP_{\I}(t)$, which is also meant to satisfy $\NonSymObliqueP_{\I}(t) \approx \FullP(t)$.
In this case, the dynamics may be written in terms of an induced chain $\InducedChain$, a reversible Markov chain corresponding to a stochastic process $\Process{\XI}$ governed by rate matrix $\Rr$, possessing stationary distribution $\pr$, and initialized with distribution $\XI_0 \sim \pr$.
All hat quantities ($\Lr$, $\Kr$, $\hr$) may be interpreted analogously to their non-hat analogs, but in a reduced space.
Our analysis will also leverage the concept of a marked chain $\MarkedChain$, a non-reversible Markov chain associated to a stochastic process $\Process{\XM}$ governed by rate matrix $\Rex$, possessing stationary distribution $\pe$, and initialized with distribution $\XM_0 \sim \pe$.
All circled quantities ($\Le$, $\Ke$, $\he$) can be interpreted similarly to their non-hat antecedents, but in an augmented space.

In order to work with symmetric matrices, our analysis mostly works with symmetrized operators $\SymP_{\I}(t)$, $\OrthP_{\I}(t)$, $\ObliqueP_{\I}(t)$ formed from $\FullP(t)$, $\NonSymOrthP_{\I}(t)$, $\NonSymObliqueP_{\I}(t)$ via the similarity transformation $\Diag(h) \cdots \DiagI(h)$. 
More generally, the tilde indicates ``asymmetric'' objects defined in terms of the untransformed representation, while counterparts without tildes are their symmetrized equivalents.
There is no difference in the particular case of $\pi \propto \One$.
(Generally we use $\Id$, $\One$, and $\Zero$ to denote the identity matrix, the column vector of ones, and the zero column vector. When a matrix or vector shape is not implied by context, we specify, e.g., $\Id_{\I,\I}$ for an identity matrix of shape $\Abs{\I} \times \Abs{\I}$, and $\One_{\I}$ for a column of $\Abs{\I}$ ones.)


Throughout we will consider bounds involving both the operator spectral norm $\Vert \, \cdot \, \Vert_2$ (i.e., the largest singular value) and the operator nuclear or trace norm $\Vert \, \cdot \, \Vert_*$ (i.e., the sum of the singular values). In order to make simultaneous statements about both norms, we adopt notation such as
$\Vert A \Vert_{\{2,*\}} \leq \cdots $ to indicate, e.g., simultaneous upper bounds. Analogous notation for other quantities will be clear from context.
We use $\norm{\,\cdot\,}_\pi$ to denote the 2-norm of a vector with respect to $\pi$, i.e. $\norm{x}_\pi \Eq \sqrt{x^\t \Diag(\pi) x}$.

\section{Compression of linear dissipative dynamics \label{s:compression}}

Consider any real positive definite $n\times n$ matrix $L$.
In understanding the dissipative dynamics 
\begin{equation}
    \frac{\odif}{\odif t} \phi(t) = -L \phi(t),\label{eq:dynamics}
\end{equation}
 it is natural to consider the solution matrix 
\begin{equation}
\SymP(t)=e^{-Lt}.\label{eq:Pt}
\end{equation}
We will be interested in the low-rank compression of such dynamics.
In this section, we will assume that $L$ is invertible and work with the definition $K:=L^{-1}$. 
The scenario in which $L$ has a null space corresponding to the stationary distribution of a Markov process will be considered in \Cref{s:projective,s:structure-preserving}.

\subsection{Compression via \Nystrom{} approximation}\label{s:compnystrom}

Our approach for compressing $\SymP(t)$ is inspired by a recently introduced
column selection approach \cite{Fornace2024-column} for kernel matrices which directly
seeks a subset of indices $\I \subset \{1,\ldots,n\}$ such
that the corresponding \Nystrom{} approximation (e.g., \cite{Drineas2005-nystrom})
\begin{equation}
    K \approx \Kn \label{eq:nystrom}
\end{equation}
 achieves minimal error in the nuclear (trace) norm. The developments of this section can be generalized straightforwardly in terms of a generalized \Nystrom{} approximation of the form $K \approx (KX) (X^\t K X)^{-1} (KX)^\t$, where $X$ is an arbitrary rectangular matrix with linearly independent columns. However, for thematic consistency with later sections we restrict our attention to the setting of true \Nystrom{} approximation.

In this work, we define two notions of error for the \Nystrom{} approximation 
\begin{equation}
    \Es(\I):= \snorm{K-\Kn}, \quad
    \En(\I):= \nnorm{K-\Kn}, \label{eq:errors}
\end{equation}
 where $\norm{\,\cdot\,}_2$ denotes the spectral norm (i.e., the largest singular value) and $\nnorm{\,\cdot\,}$ the nuclear norm (i.e., the sum of the singular values).
Although the algorithm of \cite{Fornace2024-column} targets the trace error directly,
it still achieves favorable spectral norm error in practice. Moreover,
we would like to state error bounds for our approximation of $\SymP(t)$
in both norms, which will require both norm estimates in (\ref{eq:errors}).

Given $\I$, we define our compressed solution matrix
\begin{equation}
\OrthP_{\I}(t):=Ve^{-(V^{\top}LV) t}V^{\top},\label{eq:lowrankP}
\end{equation}
in terms of a matrix $V$ with orthonormal columns (whose dependence on $\I$ is frequently omitted for visual clarity) defined by 
\begin{equation}
V = V[\I] := \mathrm{orth}(K_{:,\I})=K_{:,\I} (K^{2})_{\I,\I}^\Mh.\label{eq:V}
\end{equation}
 Although the symmetric orthogonalization is considered in the last expression, in fact any orthogonalization can be used without altering the solution matrix $\OrthP_{\I}$ defined in (\ref{eq:lowrankP}).

 It is possible to verify by direct calculation (using \Cref{lem:refactor}) that time integration of the compressed dynamics recovers the \Nystrom{} approximation:
 \begin{eqn} \label{eq:integrated}
 \int_0^\infty P_{\I} (t) \odif t = \Kn,
 \end{eqn}
 while the original dynamics more straightforwardly satisfy $\int_0^\infty P(t) \odif t = K$.

\booltrue{STATING}

\subsection{Approximation error bounds via contour integration \label{s:contour}}

Remarkably, extremely simple bounds on the discrepancy between $\OrthP_{\I}$
and $\SymP$ can be stated in terms of the \Nystrom{} errors (\ref{eq:errors}),
as follows:

\label{th:dissipative:section}
\ThDissipative*

We turn to a sketch of the proof of the theorem. The full proof is
given in \Cref{si:thm}. Before giving the sketch we will
state three useful lemmas.

\label{lem:contour0:section}
\begin{restatable}[Contour integral identity \ProofLink{lem:contour0}]{lemma}{LemContour} \label{lem:contour0}
    Suppose that $A \in \R^{k \times k}$ has eigenvalues in the strict right half-plane and $U \in \R^{n\times k}$ has orthonormal columns. Then 
    \begin{eqn}
        U e^{-A} U^\t = \frac{1}{2\pi i} \oint_{\Contour} e^{-1/z} \,  (z \Id - U A^{-1} U^\t )^{-1} \, \odif z
    \end{eqn}
    for any contour $\Contour$ in the strict right half-plane enclosing the eigenvalues of $A^{-1}$.
\end{restatable}

\label{lem:refactor:section}
\begin{restatable}[Refactoring identity \ProofLink{lem:refactor}]{lemma}{LemRefactor} \label{lem:refactor}
Given symmetric positive definite $K$ and $V = \Orth(K_{:,\I})$, the following identity holds: 
\begin{eqn}
    \label{eq:refactor}
    K_{:,\I} \Ki K_{\I,:} =V(V^{\t} K^\Mo V)^\Mo V^{\t}.
\end{eqn}
\end{restatable}

\label{lem:polebound:section}
\begin{restatable}[Resolvent bound \ProofLink{lem:polebound}]{lemma}{LemPolebound} \label{lem:polebound}
Let $A$ be a real symmetric matrix and $z\in\mathbb{C}$
with $\im(z)\ne0$. Then 
\begin{eqn}
    \snorm{(z\Id-A)^\Mo} \leq \Abs{\im(z)}^\Mo .
\end{eqn}
\end{restatable}

Now we are equipped to give the sketch of \Cref{th:dissipative}.
\begin{proof}
[Proof sketch for \Cref{th:dissipative}.] Fix $t>0$, and
consider a connected contour $\mathcal{C}$ in the strict right half-plane that encloses
the spectrum of $L^{-1} / t$. Note that consequently it also encloses the spectrum of  $(V^{\top} L V)^{-1} / t$. We will discuss further details of the choice
of contour in the full proof. (To enclose the spectrum it must thread between $1/ ( t \Vert L\Vert_2 )$
and the origin, but for technical reasons we will consider a limit
in which the contour touches the origin.)

\begin{figure}
    \centering
    \includegraphics[width=3.3in]{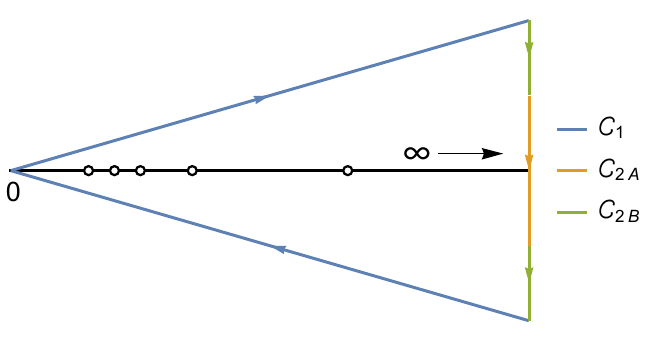}
    \vspace{-0.1in}
    \caption{
    Illustration of our limiting contour in the complex plane.
    Eigenvalues of $K / t$ are depicted as circles on the positive real axis.
    The contributions of $\Contour_{2A}$ and $\Contour_{2B}$ asymptotically vanish.
    \label{fig:contour}
    }
\end{figure}

We will apply Cauchy's integral formula: 
\begin{equation}
e^{-s^{-1}}=\frac{1}{2\pi i}\oint_{\Contour} \frac{e^{-1/z}}{z-s}\,\odif z.\label{eq:cauchy}
\end{equation}
 Since the contour encloses the spectrum of $(L t)^{-1}$, by substituting
$s\leftarrow (L t)^{-1} $ we obtain \cite{Higham_2008} the following exact expression for $\SymP(t)=e^{-Lt}$:
\begin{eqn} \label{eq:int1}
    \SymP(t)=\frac{1}{2\pi i}\oint_{\Contour} e^{-1/z} \,  \FullRes(z) \,\odif z,
\end{eqn}
where we define the resolvent function:
\begin{eqn}
    \label{eq:full-resolvent}
    \FullRes(z) \Eq (z \Id - K / t)^\Mo
\end{eqn}
and recall that $K = L^{-1}$.

Meanwhile, applying \Cref{lem:contour0} with $A = (V^\t L V) t$ and $U=V$ implies that 
\begin{eqn} \label{eq:int2}
    \OrthP_{\I}(t) = \frac{1}{2\pi i} \oint_{\Contour} e^{-1/z} \,  \OrthRes(z) \, \odif z,
\end{eqn}
where we define:
\begin{eqn} \label{eq:orth-resolvent}
    \OrthRes(z) \Eq \left( z \Id - V ( V^\top K^{-1} V )^{-1} V^\top / t \right) ^\Mo.
\end{eqn}
Note that by \Cref{lem:refactor}, we can
alternatively write 
\begin{eqn} \label{eq:orth-resolvent2}
\OrthRes(z) = \left( z \Id - (\Kn)/t \right) ^\Mo.
\end{eqn}

Now by subtracting (\ref{eq:int1}) and (\ref{eq:int2}) we
deduce: 
\begin{eqn}
    \Norms{\SymP(t)-\OrthP_{\I}(t)} \leq \frac{1}{2\pi}\oint_{\Contour} \Abs{e^{-1/z}}\, \Norms{\FullRes(z)-\OrthRes(z)} \, \Abs{\odif z}.\label{eq:intbound}
\end{eqn}

Then by factoring 
\begin{eqn}
    \FullRes(z)-\OrthRes(z) = \FullRes(z) \left[\OrthRes(z)^\Mo - \FullRes(z)^\Mo\right] \OrthRes(z),
\end{eqn}
 we can bound 
\begin{eqn}
    \Norms{\FullRes(z)-\OrthRes(z)} \leq \snorm{\FullRes(z)} \, \snorm{\OrthRes(z)} \, \Norms{\OrthRes(z)^\Mo-\FullRes(z)^\Mo},\label{eq:intbound2}
\end{eqn}
where in the ``$\, * \, $" case we have used the general inequality $\Vert A B \Vert_* \leq \Vert A \Vert_2 \, \Vert B \Vert_*$, which follows from H\"{o}lder's inequality for the Schatten norms. 
Now by Lemma \ref{lem:polebound}, we have that $\snorm{\FullRes(z)}, \snorm{\OrthRes(z)} \leq \Abs{\im(z)}^\Mo$,
and by \Cref{eq:full-resolvent} and \Cref{eq:orth-resolvent2}, we also have that  
\begin{eqn} 
    \Norms{\OrthRes(z)^\Mo-\FullRes(z)^\Mo} = \frac{1}{t} \Norms{K - \Kn} = \frac{\Eps}{t},
\end{eqn}
hence
\begin{eqn}\label{eq:resdiffbound0}
    \Norms{\FullRes(z)-\OrthRes(z)} \leq \frac{\Eps}{t  \,  \vert\im{(z)}\vert^2 }.
\end{eqn}


Finally, (\ref{eq:intbound}) and (\ref{eq:resdiffbound0}) imply that
\begin{eqn}
    \Norms{\SymP(t)-\OrthP_{\I}(t)} \leq \frac{\Eps}{2\pi t}\oint_{\mathcal{C}}\,\Abs{e^{-1/z}} \,\Abs{\im(z)}^{-2}\, \Abs{\odif z}.
\end{eqn}
 Now the contour integrand is independent of the problem, and by making
a suitable limiting choice of contour (\Cref{fig:contour}, which encloses all of $\mathbb{R}^{+}$
in the limit), we can achieve the numerical value of $\frac{3\sqrt{3}}{2\pi}$. 
(Note that the singularity of $\Abs{\im(z)}^{-2}$ on the real
axis is cancelled by the zero of $\Abs{e^{-1/z}}$ provided that
our limiting contour passes through the real axis at the origin.)
\end{proof}

\subsection{Extension to non-\Nystrom{} approximation}\label{s:non-nystrom}

In the following result, we extend the low-rank approximation bound of \Cref{th:dissipative} to cover constructions in which $V$ is an arbitrary matrix with orthonormal columns, not necessarily constructed as $\Orth(K_{:,\I})$. It recovers \Cref{th:dissipative} as a special case.

\label{th:non-nystrom:section}
\begin{restatable}[Non-\Nystrom{} compression error]{Theorem 1A}{ThNonNystrom}\label{th:non-nystrom}
    \hypertarget{link.Theorem1A}{} Let $L$ be positive definite, let $V$ have orthonormal columns, and define $\NoninterpolativeP(t) \Eq V \expt{(V^\t L V)} V^\t$. Then 
    \begin{eqn}
        \Norms{ \SymP(t) - \NoninterpolativeP(t)} \leq \frac{3 \sqrt{3}}{2 \pi} \frac{\Nus}{t},
    \end{eqn}
    where $\Nus \Eq \Norms{K - V (V^\t L V)^{-1} V^\t}$.
\end{restatable}
\begin{proof}
    The proof is exactly the same as that of \Cref{th:dissipative}, except that we carry forward the expression \Cref{eq:orth-resolvent}, rather than substituting \Cref{eq:orth-resolvent2}.
\end{proof}

In fact, we can extend Theorem \hyperlink{link.Theorem1A}{1A} further to the setting where $L$ is only positive semidefinite. Crucially, the columns of $V$ must contain the null space of $L$ in their span.

\label{th:non-nystrom-sing:section}
\begin{restatable}[Singular non-\Nystrom{} compression error \ProofLink{th:non-nystrom-sing}]{Theorem 1B}{ThNonNystromSing} \label{th:non-nystrom-sing}
    \hypertarget{link.Theorem1B}{} Let $L$ be positive semidefinite, let $V$ have orthonormal columns which contain the null space of $L$ in their span, and define 
    $\NoninterpolativeP(t) \Eq V \expt{(V^\t L V)} V^\t$. Then 
    \begin{eqn}
        \Norms{ \SymP(t) - \NoninterpolativeP(t)} \leq \frac{3 \sqrt{3}}{2 \pi} \frac{\Nus}{t},
    \end{eqn}
    where $\Nus \Eq \Norms{K - V (V^\t L V)^{+} V^\t}$.
\end{restatable}

It is interesting to contrast Theorem \hyperlink{link.Theorem1B}{1B} with the developments of the following sections, which focus specifically on the scenario where $L$ is obtained from a reversible Markov chain and possesses a one-dimensional null space due to the stationary distribution. Since  $L^{-1}$ does not exist, it is not trivial to define a notion $\Eps$ of \Nystrom{} error in this setting. We will do so by defining $L_\gamma := L + \gamma \Id$ and considering the limiting \Nystrom{} approximation of $K_\gamma := L_\gamma^{-1}$ as $\gamma \ra 0$. However, note with caution that the limiting \Nystrom{} error $\Eps$ obtained as $\gamma \ra 0$ does \emph{not} coincide with $\Nus$ furnished by Theorem \hyperlink{link.Theorem1B}{1B} with $V := \lim_{\gamma \ra 0} \mathrm{orth}( [K_\gamma]_{:,\I} )$. Importantly, it is $\ve_*$ that can be controlled directly via column selection \cite{Fornace2024-column}. We still highlight Theorem \hyperlink{link.Theorem1B}{1B} as a potentially useful tool for alternative constructions of reduced Markov models that are not tied to the \Nystrom{} framework. Indeed, the partitioning approaches \cite{Valmari2010-simple,Derisavi2003-optimal} used in some constructions of MSMs could be used to produce $V$ with orthonormal columns.

\section{Projective compression of Markov chains \label{s:projective}}

At this stage we turn our attention to reversible continuous time Markov chains.
We assume these chains to be time-homogeneous and irreducible, possessing a single stationary mode.
As such, the developments of \Cref{s:compression} must be extended to deal with this stationary mode.
The approach that we pursue is based on a physically motivated limit.


A continuous-time Markov chain $\FullChain$ on a discrete space is defined by a rate matrix $R$, which satisfies $R \One = \Zero$, $R_{i,j} \geq 0$ for all $i \neq j$. The rate matrix induces the following evolution of the probability density
\begin{eqn}
    \frac{\odif }{\odif t} \UnSymPhi(t) = R^\t \UnSymPhi(t). \label{eq:markov-ode}
\end{eqn}
Following the conventions of \Cref{sec:notation}, we use tildes in the notation to indicate objects prior to the introduction of a symmetrizing similarity transformation. Suitably transformed quantities in which the tildes are absent will be introduced below in \Cref{s:symmetrization}.

Now under the assumption that $\FullChain$ is irreducible, these dynamics have a unique, entrywise positive stationary distribution $\pi$, satisfying $\One^\t \pi = 1$, which any initial probability distribution approaches in the infinite time limit.
Relatedly, $R$ has non-positive eigenvalues, and its only zero eigenvalue corresponds to the left eigenvector $\pi$, satisfying that $R^\t \pi = \Zero$.
Under our assumption of reversibility (or detailed balance), we have that $\Diag(\pi) R = R^\t \Diag(\pi)$. 
We will always consider irreducible and reversible Markov chains in this work.

Solution of \cref{eq:markov-ode} immediately yields the probability evolution:
\begin{eqn}
    \UnSymPhi(t) = \FullP^\t(t)  \, \UnSymPhi(0) \text{ where } \FullP(t) \Eq e^{R t}.
\end{eqn}
The entry $\FullP_{i,j}(t)$ can be interpreted as the probability that the Markov chain 
starting in state $i$ at time 0 will be in state $j$ at time $t$.
The total probability $\One^\t \UnSymPhi(t)$ is constant.

\subsection{Symmetrization}\label{s:symmetrization}

Following much of the past literature (e.g., \cite{Aldous1995-reversible,Chung2000-discrete}), we work within a symmetrized framework, which is often convenient in linear algebraic derivations for reversible chains.
Given a reversible Markov chain with rate matrix $R$ and stationary distribution $\pi$, we consider 
\begin{eqn}
    \label{eq:L-h-def}
    L \Eq \Minus \Diag(h) R \DiagI(h), \qquad h \Eq \pi^\Oh.
\end{eqn}
$L$ is then symmetric (by reversibility) and positive semidefinite, with the same eigenvalues as $-R$.
Since $\pi$ is a probability density, $h$ is entrywise positive with $\norm{h}_2 = 1$.
If we define $\phi(t) = \DiagI(h) \UnSymPhi(t)$, then we may translate \cref{eq:markov-ode} as
\begin{eqn}
    \frac{\odif}{\odif t} \phi(t) = \Minus L \phi(t), \label{eq:markov-sym-ode}
\end{eqn}
bringing our perspective in alignment with \cref{eq:dynamics}, except for the fact that we must now accommodate a stationary mode in $L$.
After this transformation, conservation of probability implies that $h^\t \SymPhi(t)$ is constant.

We pause to comment on the connection to graph Laplacians. A graph Laplacian in the traditional sense is a symmetric positive semidefinite matrix $\bar{L}$ with a single zero eigenvalue corresponding to the constant function, i.e., $\bar{L} \One = \Zero$ and satisfying $L_{i,j} \leq 0$ for all $i \neq j$.
This corresponds to the special case of our framework in which $h = \One / \sqrt{n}$.
Conversely, any $L$ satisfying \cref{eq:L-h-def} furnishes a graph Laplacian $\bar{L} = \Diag(h) L \Diag(h)$, and any reversible rate matrix $R$ corresponds to a graph Laplacian $\bar{L} = \Minus \Diag(\pi) R$.
In the rest of this paper, we shall just use \emph{Laplacian} to refer to any operator $L$ satisfying \cref{eq:L-h-def} in terms of a irreducible, reversible rate matrix, and $h$ will mean the zero eigenvector of $L$.

Observe that 
\begin{eqn}
    \FullP(t) = e^{R t} = \DiagI(h) \expt{L} \Diag(h).
\end{eqn}
As such, we define: 
\begin{eqn}
    \SymP(t) \Eq \expt{L} = \Diag(h) \ert{R} \DiagI(h).
\end{eqn}
Since this matrix exponential is symmetric, we can begin to leverage the approach of \Cref{s:compression}.

\subsection{Invocation of the killed chain}\label{s:killedchain}
In \Cref{s:compression} we demonstrated how \Nystrom{} approximation of $K = L^\Mo$ could be translated into effective $1/t$ error bounds.
However, for Laplacians $L$ of the sort we examine here, $L^\Mo$ does not exist, and the stationary mode of irreducible Markov chains is tricky to handle in approximation analysis.

On the other hand, it is natural to consider \Nystrom{} approximation of a regularized dynamics obtained by killing $\FullChain$ randomly at an exponentially distributed time with mean $1/\Kill$.
    For killing rate $\Kill > 0$, we define the killed Laplacian and its inverse:
    \begin{eqn}
        \label{eq:killed-L-K}
        \La \Eq L + \Kill \Id, \qquad
        \Ka \Eq \La^\Mo.
    \end{eqn}
Then $h$ is an eigenvector of $\La$ with eigenvalue $\Kill$ and of $\Ka$ with eigenvalue $\Kill^\Mo$, and $\La$ and $\Ka$ are both symmetric positive definite.
Given a finite time $t$, the probability evolution of the killed chain may be made arbitrarily close to that of the unmodified chain by taking the limit $\Kill \ra 0$.
Therefore, we consider the \Nystrom{} approximation of $\Ka = (L + \Kill \Id)^\Mo$ in the $\Kill \ra 0$ limit.
While the inverse does not exist in this limit, the normalized quantity $\Va \Eq \Orth((\Ka)_{:,\I})$ converges suitably.

In summary, we consider a projective approach to Markov chain reduction obtained as the limiting result of the procedure of Section~\ref{s:compression} as $\gamma \ra 0$.
    To wit, given a Laplacian $L$ and nonempty subset $\I$, we define the projective compression of the dynamics killed at rate $\Kill$ as:
\begin{eqn}
    \label{eq:killed-projective-compression}
    \OrthP_{\I,\Kill}(t) &\Eq \Va \Expt{\Va^\t \La \Va} \Va^\t \text{ where } \Va \Eq \Orth((\Ka)_{:,\I}), 
\end{eqn}
and the projective compression of the dynamics as:
\begin{eqn} \label{eq:projective-compression}
    \OrthP_{\I}(t) &\Eq \lim_{\Kill \ra 0} \OrthP_{\I,\Kill}(t).
\end{eqn}

\subsection{Projective compression in terms of the committor}\label{s:committor}

While \cref{eq:projective-compression} does not express $\OrthP_{\I}(t)$ in closed form, in this section we show how this can be achieved using the committor functions of the Markov chain, defined as follows:

\begin{definition}[Committor matrix]
\label{def:sym-operators}
Given non-empty $\I$, the committor $\Cu$ is the $n \times \Ia$ matrix where: 
\begin{eqn} \label{eq:committor-prob-def}
    \Cu_{i,j} \Eq \Prob(X_{\tau_{\I}} = j \mid X_0 = i), 
\end{eqn}
i.e., $\Cu_{i,j}$ is the probability that the Markov chain started at $i$ will reach $j$ first among all indices in $\I$. (Recall the notation \Cref{eq:stoppingtime} used for hitting times.)
We also define:
\begin{eqn}
    \label{eq:stationary-chp}
    \pr \Eq \Cu^\t \pi, \qquad \hr \Eq \pr^\Oh, \qquad
    \Co \Eq \Diag(h) \Cu \DiagI(\hr).
\end{eqn}
\end{definition}
\noindent As we will explain later in \Cref{s:structure-preserving}, $\pr$ can be interpreted as the stationary distribution of an induced Markov chain on the smaller subset $\I$ of states. For now, we simply adopt this definition for ease of notation in defining the transformed committor $C$ following~\cref{eq:stationary-chp}.

It follows immediately from the irreducibility of the Markov chain that $\Cu \One = \One$ in general.
As a result, $\One^\t \pr = 1$ and $\norm{\hr}_2 = 1$, consistent with the aforementioned probabilistic interpretation.
Observe moreover that $\Cu$ and $\Co$ are entrywise non-negative, while $\pr$ and $\hr$ are entrywise positive.




We may now use \cref{eq:stationary-chp} to provide a closed-form expression for $\OrthP_{\I}(t)$:
\label{th:projective-compression:section}
\begin{restatable}[Projective compression in terms of committor  \ProofLink{th:projective-compression}]{proposition}{ThProjectiveCompression} \label{th:projective-compression}
    The projectively compressed dynamics \cref{eq:projective-compression} can be expressed 
    \begin{eqn}
        \OrthP_{\I}(t) = \Co \Expt{\Co^+ L \Co} \Co^+ = V \Expt{V^\t L V} V^\t  \text{ where } V \Eq \Orth(\Co).
        \label{eq:orth-prop}
    \end{eqn}
\end{restatable}
\noindent The proof of this proposition makes use of \Cref{th:killed-committor}, stated below in terms of the suitable committor quantities for the killed chain. Note that in \cref{eq:orth-prop}, any orthogonalization of $\Co$ works equivalently, but for concreteness we will later just consider the symmetric orthogonalization $V = \Co (\Co^\t \Co)^\Mh$.

Before proceeding, we comment that it is possible to prove closed-form expressions for $\Cu$ (and $\Co$) in closed form using the definition of $K \Eq L^+$, the Moore-Penrose pseudoinverse of $L$.
Note that in this setting, $K$ coincides (modulo diagonal scalings) with the \emph{fundamental matrix} of the Markov chain, which appears commonly in the wider literature (e.g., \cite{Aldous1995-reversible,Kemeny1976-markov,Chung2000-discrete}).
While previous work has provided an expression for $\Cu$ in the special case $\Ia = 2$ \cite{Aldous1995-reversible}, we offer a closed-form expression in the general case of non-empty $\I$ in \Cref{th:committor}, which is novel to the best of our knowledge. An alternative derivation of \Cref{th:committor}, via the limit of Markov chains killed at rate $\gamma \ra 0$, is outlined in the following section.

\subsection{Approximation error bounds}\label{s:projapproxbounds}

To properly treat the stationary mode associated with the Markov chain, the above discussion suggests treating the chain killed at rate $\Kill$, performing requisite analysis, and taking the limit $\Kill \ra 0$ at the conclusion.
This will be our general strategy in extending the approximation approach of \Cref{s:compression} to handle the stationary distribution of the Markov chain.
We begin by designing committor-related quantities for the killed chain as follows.
\begin{definition}[Committor-related quantities for the killed chain] \label{def:killed-committor}
    In terms of $\Ka = \La^\Mo$ \Cref{eq:killed-L-K}, define 
    \begin{eqn}
        \label{eq:killed-committor}
        \Cu_\Kill \Eq \Diag^\Mo(h) (\Ka)_{:,\I} (\Ka)_{\I,\I}^\Mo \Diag(\hi)
    \end{eqn}
    as well as the downstream quantities 
\begin{eqn}
    \label{eq:killed-chp}
    \pa \Eq \Cu_\Kill^\t \pi, \qquad \ha \Eq \pa^\Oh, \qquad
    \Ca \Eq \Diag(h) \Cu_\Kill \DiagI(\ha).
\end{eqn}
\end{definition}

\noindent Note that though we choose to define it here via an explicit formula, $\Cu_\Kill$ admits a probabilistic interpretation analogous to the specification of $\Cu$ in \Cref{def:sym-operators}, by consideration of a Markov chain that is killed at rate $\gamma > 0$. This interpretation is useful to us later, and we prove it in \Cref{th:killed-committor-prob}.

Then we show that all killed quantities converge to their anticipated limits as $\Kill \ra 0$. Note that for the committor, the proof of this fact offers an alternative derivation of the formula for $C$ in terms of the fundamental matrix (\Cref{th:committor}).

\label{th:killed-committor:section}
\begin{restatable}[Convergence of committor quantities   \ProofLink{th:killed-committor}]{lemma}{ThKilledCommittor} \label{th:killed-committor}
    Given \Cref{def:sym-operators} and \Cref{def:killed-committor},
    in the $\Kill \ra 0$ limit:
    \begin{eqn}
        \Cu_\Kill \ra \Cu, \qquad
        \Ca \ra \Co, \qquad
        \pa \ra \pr, \qquad
        \ha \ra \hr.
    \end{eqn}
\end{restatable}
\noindent \Cref{th:killed-committor} highlights how, in addition to its probabilistic motivation, $C_\gamma$ may be thought of as a convenient normalization of $(\Ka)_{:,\I}$ in the $\Kill \ra 0$ limit.

Next we define the \Nystrom{} approximation error norms for the original and killed chains:

\begin{definition}[\Nystrom{} errors for reversible Markov chains]
    \label{def:markov-epsilon}
    For the killed chain with $\gamma > 0$, define the approximation norms:
    \begin{eqn}
        \label{eq:killed-epsilon}
        \Eps^{\Kill}(\I) \Eq \Norms{\Ka - (\Ka)_{:,\I} (\Ka)_{\I,\I}^\Mo (\Ka)_{\I,:}},
    \end{eqn}
    and extend these definitions to the original Markov chain by taking the limit $\gamma \ra 0$: 
    \begin{eqn}
        \label{def:recurrent-epsilon}
        \Eps(\I) \Eq \lim_{\Kill \ra 0} \Eps^{\Kill}(\I).
    \end{eqn}
\end{definition}

In \Cref{th:markov-norm-limits} we show that the limit \Cref{eq:killed-epsilon} does in fact exist and can be written in closed form in terms of the fundamental matrix $K$. Note that it is precisely the limiting quantity $\En (\I) $ that is controlled \emph{a priori} by choosing $\I$ using the nuclear maximization approach of \cite{Fornace2024-column}, as we shall elaborate further in \Cref{s:algs} below.

At this point, we can apply \Cref{th:dissipative} directly to obtain an  approximation error bound for all $\Kill > 0$. By taking the limit of the compressed dynamics and the error bound as $\gamma \ra 0$, we obtain the following:

\label{th:ortho-markov:section}
\ThOrthoMarkov*
\noindent \Cref{s:experiments} shows the empirical utility of these bounds through multiple case studies.








\subsection{Relating error bounds to autocorrelation}\label{s:untransformed}

Finally, we show how norm bounds on the approximation of $\SymP (t)$ may be interpreted in terms of the original, untransformed Markov chain dynamics $\FullP (t)$. 
Autocorrelation analysis (e.g., \cite{Park2017-fundamentals,Frenkel2023-understanding}) offers an elegant perspective.

Given dynamics $\FullP$, the autocorrelation of a function on the state space, viewed as a vector $f \in \R^n$, is defined by 
\begin{eqn} \label{eq:acorr}
A_{t}(f, \FullP) \Eq \Ex \lrs{f_{X_0} f_{X_t} \mid X_0 \sim \pi, \,  X_t \sim \FullP_{X_0,\,:\,}(t)}.
\end{eqn}
More generally, the autocorrelation of vector-valued function, viewed as a matrix $F \in \R^{n \times m}$ for arbitrary $m$, is defined by 
\begin{eqn} \label{eq:acorr2}
A_{t}(F, \FullP) \Eq \Ex \lrs{F_{X_0,\,:\,}^\t \, F_{X_t,\,:\,} \mid X_0 \sim \pi, \, X_t \sim \FullP_{X_0,\,:\,}(t)}.
\end{eqn}
Note that the output is an $m \times m$ matrix.

A more proper definition of autocorrelation should assume that the functions have zero mean with respect to the stationary distribution $\pi$, but the following proposition, which relates our symmetric approximation error to autocorrelation discrepancies, is invariant to such a shift.

\label{th:autocorrelation:section}
\begin{restatable}[Approximation of autocorrelation \ProofLink{th:autocorrelation}]{proposition}{ThAutocorrelation} \label{th:autocorrelation}
Given dynamics $\FullQ(t)$ which approximate $\FullP(t)$ with $\SymQ(t) \Eq \Diag(h) \FullQ(t) \DiagI(h)$:
\begin{eqn}
    \label{eq:autocorrelation}
    \snorm{\SymP(t) - \SymQ(t)} &= \sup_{f \in \R^n} \lrb{ \Abs{A_t(f, \FullQ) - A_t(f, \FullP)} \;:\; \norm{f}_\pi = 1}, \\
    \nnorm{\SymP(t) - \SymQ(t)} &= \sup_{m \leq n, \, F \in \R^{n\times m} } \lrb{ \nnorm{A_t(F, \FullQ) - A_t(F, \FullP)} \;:\; F^\t \Diag(\pi) F = \Id}.
\end{eqn}
\end{restatable}




\section{Structure-preserving compression of Markov chains \label{s:structure-preserving}}

In \Cref{s:projective}, we demonstrated a projective compression approach to Markov chain approximation. 
In summary (cf. \Cref{th:projective-compression}), our approximation to the dynamics was defined by the following formula:
\begin{eqn}
    \label{eq:ortho-p}
    \OrthP_{\I}(t) \Eq \Co \expt{\Co^+ L \Co} \Co^+ 
\end{eqn}
where $\Co$ is the rescaled committor of the process with respect to subset $\I$ (cf. \Cref{def:sym-operators}).
This approximation is not fully structure-preserving in the sense that (1)
$\OrthP_{\I}(t)$ is not guaranteed to be entrywise non-negative (hence in particular the untransformed dynamics $\NonSymOrthP_{\I}(t)$ cannot be row stochastic), (2) the generator $\Co^+ L C$ is not guaranteed to have Laplacian sign structure, and (3) the compressed dynamics cannot be interpreted in terms of a Markov chain on a smaller state space.

In this section, we introduce a structure-preserving compression that addresses each of these shortcomings. Although the derivation and analysis are considerably more complicated, the approximation ultimately admits the following simple formula:  
\begin{eqn}
    \label{eq:oblique-p}
    \ObliqueP_{\I}(t) \Eq \Co \expt{\Lr} \Co^\t, \ \text{ where } \Lr \Eq \Co^\t L \Co.
\end{eqn}
In this formula, $\Lr$ is a Laplacian in the sense of \Cref{s:symmetrization} with stationary eigenvector $\hr$ defined as in \cref{eq:L-h-def}. As we shall explain in \Cref{s:induced}, $\Lr$ corresponds to a reversible Markov chain over $\Ia$ states. Meanwhile, $\ObliqueP_{\I}(t)$ is guaranteed to be entrywise non-negative and symmetric. Moreover, it is similar to the matrix 
\begin{eqn} \label{eq:oblique-nonsym}
    \NonSymObliqueP_{\I}(t) \Eq \DiagI(h) \ObliqueP_{\I}(t) \Diag(h),
\end{eqn}
which is entrywise non-negative and row stochastic (cf. \Cref{th:sp-nonnegativity} below).

While \cref{eq:ortho-p} and \cref{eq:oblique-p} appear very similar, the error analyses necessitate substantially different techniques, as the columns of $C$ are not orthonormal in general.
In spite of this difficulty, we can prove the following simple bound on the nuclear norm error of the approximation \cref{eq:oblique-p}:

\booltrue{NOLINK}
\label{th:sp-nuclear-bound:section}
\ThSpNuclearBound*
\boolfalse{NOLINK}
\Cref{th:sp-nuclear-bound} (which is a direct consequence of Theorem \hyperlink{Theorem3*}{3*} and \Cref{th:probabilistic-obliqueness-factor}, to be introduced and proved below) represents a worst-case \textit{a priori} bound on the structure-preserving approximation error.
Here we use \textit{a priori} to indicate that the bound that can be formulated solely in terms of the \Nystrom{} error $\En (\I)$, which can in turn be controlled via the nuclear column selection algorithm \cite{Fornace2024-column}, as we shall elaborate in \Cref{s:algs}.

We also prove empirically tighter \textit{a posteriori} spectral and nuclear norm error bounds in Theorem \hyperlink{Theorem3*}{3*}. This result is more complicated to formulate, though here we comment that no explicit factor of $\Ia$ appears, in contrast with \eqref{eq:thm3bound}. These \textit{a posteriori} bounds can be computed in practice given a choice of subset $\I$, after which point they can be used to bound the approximation error for all times $t > 0$.

To proceed, first we motivate \cref{eq:oblique-p} and prove its structure-preserving properties using probabilistic arguments in \Cref{s:induced,s:lifting}. 
Then in \Cref{s:marked} we will construct the marked chain, which conveys useful probabilistic insights as we finally turn to the proof of \Cref{th:sp-nuclear-bound} in \Cref{s:oblique-approximation}.


\subsection{The induced chain \label{s:induced}}

In this section, we motivate \cref{eq:oblique-p} in terms of an \emph{induced} Markov chain $\InducedChain$ over a non-empty subset of states $\I \subset \Indices$. We introduce the nomenclature ``induced chain'' as a generalization of the construction appearing in \cite{Aldous1995-reversible}, which is recovered as a special case if $\Cu$ is a strict partitioning, i.e., 0-1 entrywise. More generally, the entries in our committor matrix $\Cu$ can lie in $[0, 1]$.

Concretely, any choice of subset $\I$ determines an induced rate matrix
\begin{eqn}
    \Rr &\Eq \Ct R \Cu
    \label{eq:simple-induced}
\end{eqn}
which governs the induced chain $\InducedChain$, where 
\begin{eqn}\label{eq:csharp}
\Ct \Eq \DiagI(\pr) \Cu^\t \Diag(\pi) \in \R^{\Ia \times n}
\end{eqn}
can be viewed as the adjoint of $\Cu$ (cf. \Cref{def:sym-operators}) accounting for the stationary distribution. 
Here $\pr \Eq \Cu^\t \pi$, in accordance with \cref{eq:L-h-def} above. 
We denote the stochastic process corresponding to $\InducedChain$ as $\Process{\XI}$, which is governed by rate matrix $\Rr$ and is initialized via $\XI_0 \sim \pr$.
We will provide a probabilistic interpretation of this chain below, but for now we verify several formal properties.

\label{th:induced-basics:section}
\begin{restatable}[Elementary properties of the induced chain \ProofLink{th:induced-basics}]{lemma}{ThInducedBasics}\label{th:induced-basics}
    The induced chain $\InducedChain$ determined by  $\Rr$  \cref{eq:simple-induced} is an irreducible reversible Markov chain with stationary distribution $\pr$.
    Specifically, $0 < \pr_i \leq 1$ for all $i$, $\One^\t \pr = 1$, $\Rr_{i,i} < 0$ for all $i$, $\Rr_{i,j} \geq 0$ for all $i \neq j$, $\Rr^\t \pr = \Zero$, $\Rr \One = \Zero$, and $\Diag(\pr) \Rr = \Rr^\t \Diag(\pr)$. 
\end{restatable}

Recall that $\Process{X}$ denotes the stochastic process corresponding to $\FullChain$ with initial probability distribution $X_0 \sim \pi$ and that $\tau_{\,\mathcal{S}}$ denotes the hitting time for a subset $\mathcal{S} \subset \Indices$, cf. \Cref{eq:stoppingtime}.
Then the objects $\Rr$ and $\pr$ can then be interpreted probabilistically as follows.

\label{th:induced-interpretation:section}
\begin{restatable}[Characteristics of induced chain  \ProofLink{th:induced-interpretation}]{proposition}{ThInducedInterpretation} \label{th:induced-interpretation}
    Assuming $\Ia > 1$:
    \begin{enumerate}
        \item The average time for $\FullChain$ started at state $i \in \I$ to reach any state in $\Iminus{i}$ is $\Minus \Rr_{i,i}^\Mo$. More precisely, 
        \begin{eqn}
            \Minus \Rr_{i,i}^\Mo = \Ex(   \tau_{\,\Iminus{i}} \mid X_0 = i).
        \end{eqn}
        \item Starting $\FullChain$ in $i$ and stopping it once it hits any state $j \in \Iminus{i}$, the probability of stopping in state $j$ is $\Minus \Rr_{i,j}/\Rr_{i,i}$. More precisely, 
        \begin{eqn}
            \Minus \Rr_{i,j}/\Rr_{i,i} = \Prob(X_{\tau_{\,\Iminus{i}}} = j \mid X_0 = i).
        \end{eqn} 
        \item The probability of reaching state $i$ first of the states in $\I$ is $\pr_i$. More precisely,
        \begin{eqn}
            \pr_i = \Prob(X_{\tau_{\,\I}} = i ).
        \end{eqn}
    \end{enumerate}
\end{restatable}
Given \Cref{th:induced-interpretation}, it can be shown that $\InducedChain$ exactly preserves the mean first passage times between the chosen states $\I$ (\Cref{th:hitting-times-preserved}).



It is also possible to derive $\InducedChain$ via consideration of the stationary flow in the original chain $\FullChain$.
For this we define the related operators: 
\newcommand{\Flow}{\Delta}
\newcommand{\InducedFlow}{\hat{\Delta}}
\begin{eqn} \label{eq:flow-matrices}
    \Flow \Eq \Minus \Diag(\pi) R, \qquad \hat \Flow \Eq \Minus \Diag(\pr) \Rr.
\end{eqn}
By reversibility, these are symmetric positive semidefinite matrices satisfying $\Flow \One = \Zero$, $\hat \Flow \One = \Zero$, and they coincide with expectations of the flow rates of $\FullChain$ as follows:
\label{th:induced-flow:section}
\begin{restatable}[Induced chain from stationary flow \ProofLink{th:induced-flow}]{proposition}{ThInducedFlow} \label{th:induced-flow}
    Consider the stochastic processes $\Process{X}$ and $\Process{\XI}$ corresponding to the original chain $\FullChain$ and the induced chain $\InducedChain$, respectively.
    Let $\Flow$ and $\InducedFlow$ be as in \cref{eq:flow-matrices} .
    Then:
\begin{eqn}
    \Flow &= \lim_{t \ra 0} \frac{1}{2t} \Ex \lrs{ (\Id_{X_t,:} - \Id_{X_0,:})^\t (\Id_{X_t,:} - \Id_{X_0,:})  }, \qquad 
    \InducedFlow &= \lim_{t \ra 0} \frac{1}{2t} \Ex \lrs{ (\Id_{\XI_t,:} - \Id_{\XI_0,:})^\t (\Id_{\XI_t,:} - \Id_{\XI_0,:})  } 
\end{eqn}
and $\InducedFlow$ is the projection of the stationary flow onto the committor functions: 
\begin{eqn}
    \InducedFlow = \lim_{t \ra 0} \frac{1}{2t} \Ex \lrs{ (\Cu_{X_t,:} - \Cu_{X_0,:})^\t (\Cu_{X_t,:} - \Cu_{X_0,:}) }. 
\end{eqn}
\end{restatable}
\noindent $\Rr$ can be recovered from $\InducedFlow$ via \cref{eq:flow-matrices}, yielding each rate constant $\Rr_{i,j}$ as the flow divided by the stationary occupancy of the committor $\pr_i$. In fact this approach yielded our original derivation of $\hat{R}$.

Finally, the induced chain can also be obtained as a result of a Dirichlet form minimization of the following type, generalizing a known result for $\Ia=2$ \cite{Doyle1984-random,Aldous1995-reversible}:
\label{th:dirichlet-trace:section}
\begin{restatable}[Induced chain via Dirichlet form minimization \ProofLink{th:dirichlet-trace}]{proposition}{ThDirichletTrace} \label{th:dirichlet-trace}
    Consider $\Ia > 1$ and $\Flow$, $\InducedFlow$ as defined above.
    Then the stationary flow of $\InducedChain$ is yielded by the following minimization:
    \begin{eqn}
        \InducedFlow &= \Cu^\t \Flow \Cu, 
        \quad
    \text{ where } \Cu = \argmin_{F \in \R^{n \times \Ia}} \left\{ \Tr [F^\t \Flow F] :\; F_{\I,:}  = \Id \right\}
    \label{eq:dirichlet-trace}
    \end{eqn}
    where $\Cu$ is the committor induced by states $\I$.
\end{restatable}


\vspace{-1mm}
\subsection{Structure-preserving compression via the induced chain \label{s:lifting}}

In the same way that we symmetrize $\FullChain$ to yield $h$, $L$, and $K$ (\Cref{s:symmetrization}), we symmetrize $\InducedChain$ to yield suitable definitions for $\hr$, $\Lr$, $\Kr$: 
\begin{eqn}
    \label{eq:reduced-h-L-K}
    \hr = \pr^\Oh, \qquad \Lr = \Minus \Diag(\hr) \Rr \DiagI(\hr), \qquad \Kr := \Lr^+,
\end{eqn}
where this prescription of $\Lr$ recovers its definition in \cref{eq:oblique-p}.
Recall that \Cref{th:committor} offers a closed-form expression for $\Cu$ in terms of  $K = L^+$. Using this, we can derive closed-form expressions for $\pr$, $\Rr$, $\Lr$, $\Kr$, which are reported in \Cref{th:induced-from-k}.

The compression $\hat{L}$ of $L$ is structure-preserving in the following sense:
\label{th:induced-is-laplacian:section}
\begin{restatable}[$\Lr$ as a Laplacian \ProofLink{th:induced-is-laplacian}]{corollary}{ThInducedIsLaplacian} \label{th:induced-is-laplacian}
    $\Lr$  is a Laplacian with stationary distribution $\hr$.
\end{restatable}

We now complete our construction of the structure-preserving compression by describing how the dynamics of $\InducedChain$ should be lifted to an approximation of the dynamics of $\FullChain$ on the full state space:
\begin{eqn}
    \label{eq:sp-lifting}
    \ObliqueP(t) = C \expt{\Lr} C^\t, \qquad
    \NonSymObliqueP(t) = \Cu \ert{\Rr} \Ct,
\end{eqn}
so that $\ObliqueP(t) \approx \SymP(t)$, $\NonSymObliqueP(t) \approx \FullP(t)$, and
$\ObliqueP_{\I}(t) =  \Diag(h) \NonSymObliqueP_{\I}(t) \DiagI(h)$ (as in \cref{eq:oblique-nonsym}).

First, we show that these approximations naturally preserve the essential structure of the original probability evolution:
\label{th:sp-nonnegativity:section}
\begin{restatable}[Preservation of structure \ProofLink{th:sp-nonnegativity}]{corollary}{ThSpNonnegativity} \label{th:sp-nonnegativity}
    For any $t \geq 0$, the entries of $\NonSymObliqueP_{\I}(t)$ and $\ObliqueP_{\I}(t)$ are non-negative.
    Furthermore, $\NonSymObliqueP_{\I}(t)$ is row stochastic (i.e., $\NonSymObliqueP_{\I}(t) \One = \One$), while $\ObliqueP_{\I}(t)$ is symmetric with $\lim_{t \ra \infty} \ObliqueP_{\I}(t) = h h^\t$.
\end{restatable}

Next we show how the structure-preserving dynamics recover the \Nystrom{} approximation via integration in time. In order to ensure convergence of the integral, we must consider more generally the 
structure-preserving dynamics killed at rate $\Kill$:
\begin{eqn} \label{eq:Pspgamma}
    \ObliqueP_{\I,\Kill} (t) \Eq \Ca \expt{ \Lrk } \Ca^\t, \ \text{ where }  \Lrk \Eq \Ca^\t \La \Ca.
\end{eqn}
Here $\Ca$ is the killed analog of $\Co$, defined in \Cref{def:killed-committor}.

\label{th:sp-integral:section}
\begin{restatable}[\Nystrom{} approximation as an integral over time \ProofLink{th:sp-integral}]{lemma}{ThSpIntegral} \label{th:sp-integral}
With $\ObliqueP_{\I,\Kill}(t)$ as defined above and $\Ka$ as in \cref{eq:killed-L-K}:
\begin{eqn} \label{eq:gammanystrom}
    \int_0^\infty \ObliqueP_{\I,\Kill}(t) \; \odif t &= \KA_{:, \I} \KA_{\I,\I}^\Mo \KA_{\I,:}.
\end{eqn}
\end{restatable}

%
\noindent Observe that the right-hand side of \Cref{eq:gammanystrom} is the \Nystrom{} approximation of $K_\gamma$, which is exact on the principal submatrix corresponding to the subset of indices $\I$. After considering the diagonal similarity transformation relating $\ObliqueP_{\I,\Kill}(t)$ to $\NonSymObliqueP_{\I,\Kill}(t)$, we arrive at a key interpretation of the structure-preserving dynamics: for any killing rate $\gamma>0$ and any initialization, the integrated occupation time in each of the selected states $i \in \I$ exactly matches that of the original dynamics. Note that we must consider a positive killing rate in order for the integrated occupation time to be finite.

Moreover, observe that \cref{eq:integrated}, applied with $K_\gamma$ in the place of $K$, establishes that the \emph{projective} compression satisfies a similar identity:
\begin{eqn}
\int_0^\infty P_{\I,\Kill}(t) \odif t = \KA_{:, \I} \KA_{\I,\I}^\Mo \KA_{\I,:}.
\end{eqn}
This correspondence can be taken as a motivation for our lifting \Cref{eq:sp-lifting} of the induced chain dynamics on $\I$ to the structure-preserving dynamics on the full state space $[n]$.

In the next section we will explore a deeper motivation for the same structure-preserving dynamics $\ObliqueP_{\I}(t)$, as well as $\ObliqueP_{\I,\Kill} (t)$ for $\gamma > 0$.

\subsection{The marked chain \label{s:marked}}

\begin{figure}
    \centering
    \includegraphics[width=0.75\textwidth]{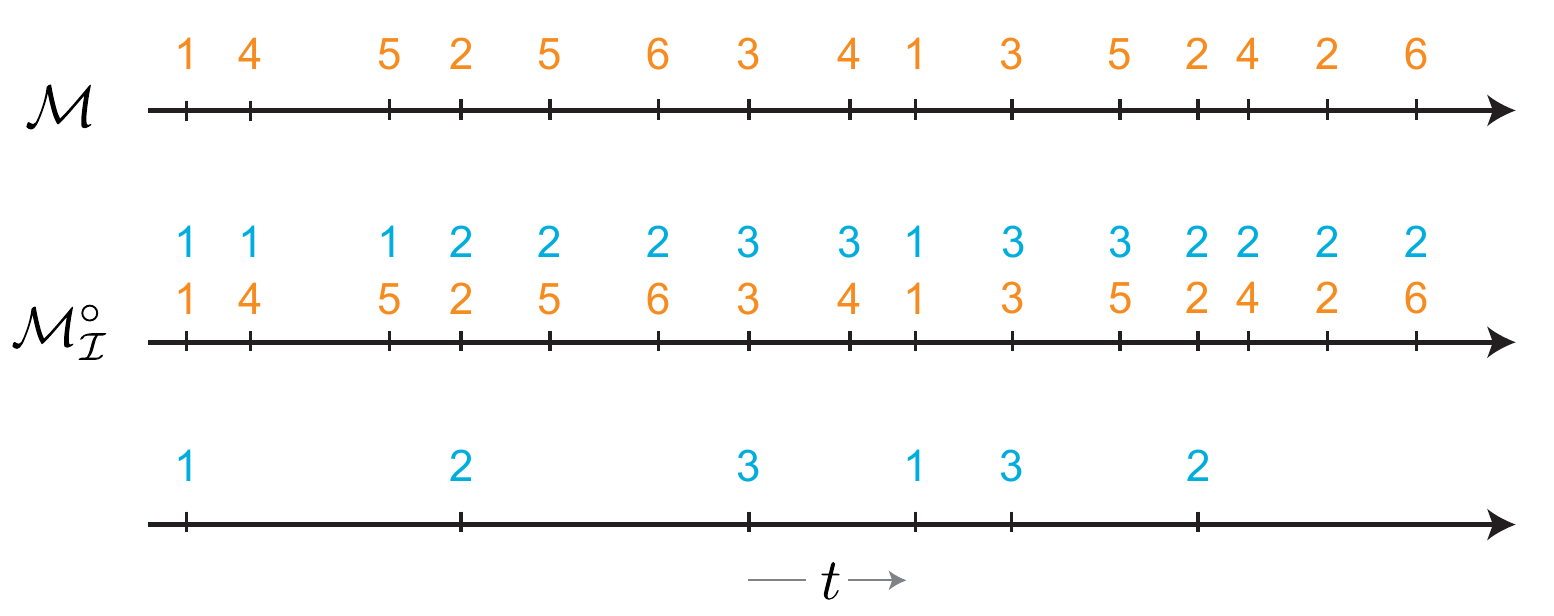}
    \caption{
        \label{fig:marked-chain}
        Depiction of sample trajectories from the original and marked chains $\FullChain$ and $\MarkedChain$, where $\I = \lrb{1,2,3}$ and $\Ic = \lrb{4,5,6}$. 
        The marked chain is obtained by labeling the original chain with the last visited state from the subset $\I$. The sample trajectory at the bottom is obtained by deleting the original position state from the marked chain samples. This does not define a Markov process but can be viewed as a projection $\So^\t \exp(-\Le t) \So$ of the marked chain dynamics, cf. \Cref{def:marked-laplacian} and \Cref{eq:marked-projections}.
    }
\end{figure}

In this section, we consider the construction of a \emph{marked chain} $\MarkedChain$ over a space augmenting that of the original chain. 
The construction of the marked chain will allow us (1) to interpret the lifting of the induced chain to the the full state space (cf. \cref{eq:sp-lifting}) and (2) to analyze the approximation error of this lifting (cf. \Cref{th:sp-nuclear-bound}).

Sample paths of the marked chain $\MarkedChain$ are generated by running the original chain without modification, while additionally applying a label $i$ during each traversal of a state $i \in \I$ (\Cref{fig:marked-chain}). 
Concretely, a state in the marked chain may be referenced as an ordered pair consisting of:
\begin{enumerate}
    \item An auxiliary marking in $\I$, which we call the \emph{marking space}, indicating which element of $\I$ the system last reached.
    \item A state in the original state space $[n]$ of $\FullChain$, which we term the \emph{position space}, corresponding to the current state of $\FullChain$.
\end{enumerate}
\newcommand{\AugmentedSpace}{\mathcal{A}^{\otimes}_{\I}}
Note that it is impossible for a state in position $i \in \I$ to have a marking other than $i$.
Therefore, we can define the \emph{augmented space}, i.e., the state space of the marked chain, as the subset of the product space $\I \times \Indices$ obtained by omitting all impossible states $\lrb{(i,j): i \in \I, j \in \I, i \neq j}$. For our linear-algebraic manipulations, it is useful to identify the augmented space with the set $\I \cup [\I \times \Ic]$, where  $\Ic \Eq \Indices \setminus \I$. To understand the bijection between $\I \cup [\I \times \Ic]$ and the augmented space, observe that the former is comprised of:
\begin{enumerate}
    \item The position states $i \in \I$ at which the original process is marked, which correspond to the ordered pairs or augmented states $(i,i) \in \I \times \I \subset \I \times [n]$. We call such pairs ``marked states.'' 
    \item The augmented states $(i,j) \in \I \times \Ic \subset \I \times [n]$ in which the position state $j$ lies outside of $\I$ but carries a marking $i$ from $\I$. We call such pairs ``unmarked states.''
\end{enumerate}


The dynamics of $\MarkedChain$ are Markovian but not reversible.
To see this concretely, note that for $i \neq j$ and $k \in \Ic$, direct transitions $(i,k) \ra (j,j)$ are possible while direct transitions $(j,j) \ra (i,k)$ are impossible.

It is possible that there could exist an unmarked state $(j,k) \in \I \times \Ic$ that is not reachable from the marked state $(j,j)$ and therefore not reachable from any marked state. This occurs if $k \in \Ic$ lies in a subset of the position space that cannot be reached from $j$ except by passing through some $i \in \I \setminus \{j \}$. Then the ``unphysical'' state $(j,k)$ cannot be reached from $(j,j)$ because along any path from $j$ to $k$ in the position space, we must pick up the marking from $i$ on the way. In such a case, the marked chain is not irreducible. Note that from an arbitrary initialization, the probability of the unphysical states converges to zero in the long-time limit because, by the irreducibility of the original chain, all states in the augmented space must eventually transition to some marked state, after which point the unphysical states become unreachable. However, it is sometimes convenient to adopt a convention explicitly ruling out such unphysical states (i.e., states not reachable from any marked state) by deleting them from the augmented space. Likewise, in linear-algebraic formulations of the dynamics, we sometimes understand suitable vector entries and matrix rows/columns to be omitted, in order to avoid divisions by zero due to the zero entries in the stationary distribution. (In Markov chain terminology \cite{LevinPeresWilmer2006}, we can restrict the augmented chain to its unique essential communicating class.)

After adopting this convention, the Marked chain $\MarkedChain$ is by construction \emph{irreducible}, hence possesses a unique stationary distribution with positive entries. It follows that even if the convention is not adopted, the marked chain still admits a unique stationary distribution, which is extended by zeros to the unphysical states. In the ensuing developments, the interpretation should be clear from context.

Next we turn toward a linear-algebraic specification of the dynamics of the marked chain. To do so, we identify the augmented space with the indexing set $\I \cup [\I \times \Ic]$. Note that vectors over this indexing set can be viewed as elements of $\R^{\Ia} \oplus \big( \R^{\Ia} \otimes \R^{\vert \Ic \vert} \big)$, i.e., the indexing convention induces a two-block structure in which the first block is simply indexed by $i \in \I$, while the second block has a tensor product structure and may be indexed by a composite index $(j,k) \in \I \times \Ic$.

Accordingly, the rate matrix for the marked chain $\MarkedChain$ can be viewed as a $2 \times 2$ block matrix, in which the rows of the lower blocks, as well as the columns of the right blocks, are indexed with such a composite index. Meanwhile, the rows of the upper blocks, as well as the columns of the left blocks, are indexed by $\I$. In order to express the blocks of the rate matrix, we use the Kronecker $\otimes$ and face-splitting $\bullet$ products \cite{Slyusar1999-family} in our notation, which are defined by $(A \otimes B)_{(i,k), (j,l)} \Eq A_{i,j} B_{k, l}$ for all $i,j,k,l$ and $(A \bullet B)_{i,(j,k)} = A_{i,j} B_{i,k}$ for all $i,j,k$.

Then the following result computes the marked chain's rate matrix explicitly.

\label{th:marked-chain-dynamics:section}
\begin{restatable}[Marked chain dynamics \ProofLink{th:marked-chain-dynamics}]{proposition}{ThMarkedChainDynamics} \label{th:marked-chain-dynamics}
    Let $\Rex$ be the rate matrix corresponding to the probability evolution of the marked chain $\MarkedChain$, and let $\pe$ be its unique stationary distribution. Then according to our block indexing:  
    \begin{eqn}
        \label{eq:expanded-R}
            \Rex \Eq \begin{bmatrix}
            R_{\I, \I} & \Id_{\I,\I} \bullet R_{\I,\Ic} \\
            \One_{\I} \otimes R_{\Ic,\I} & \Id_{\I,\I} \otimes R_{\Ic,\Ic}
        \end{bmatrix}, \qquad \pe \Eq \begin{bmatrix} \pi_{\I} \\ \mathrm{vec}(\Diag(\pi_{\Ic}) \Cu_{\Ic,:}) \end{bmatrix}.
    \end{eqn}
\end{restatable}
\noindent It is useful to observe that, indexing the augmented space by $(i,j) \in \I \times [n]$, we can write 
\begin{eqn} \label{eq:pealt}
    \pe_{(i,j)} = \tilde{C}_{j,i} \, \pi_{j},
\end{eqn}
since $\tilde{C}_{\I,:} = \Id_{\I,\I}$. Note that this formula extends the expression in \Cref{eq:expanded-R} by zeros to the impossible states $(i,j)$ for which $j \in \I \setminus \{i\}$. 

Given \Cref{th:marked-chain-dynamics}, we define the stochastic process corresponding to the marked chain $\MarkedChain$ as $\Process{\XM}$, where the dynamics are governed by $\Rex$ and the process is initialized via $\XM_0 \sim \pe$.

Note that \cref{eq:expanded-R} directly implies that $\MarkedChain$ is irreversible, since the off-diagonal blocks have different sparsity patterns.
Although the marked chain is not reversible, it is still useful to define suitably rescaled quantities analogous to those constructed in our symmetrization of $\FullChain$ (\Cref{s:symmetrization}):

\begin{definition}[Rescaled marked chain] \label{def:marked-laplacian}
    In terms of the marked chain rate matrix $\Rex$ and stationary distribution $\pe$, define:
    \begin{eqn}
        \he \Eq \pe^\Oh, \qquad \Le \Eq \Minus \Diag(\he) \Rex \DiagI(\he).
    \end{eqn}
\end{definition}

\noindent Note carefully that $\Le$ is asymmetric (following the irreversibility of $\MarkedChain$), hence not technically a ``Laplacian'' in the sense that we have defined. 
But though $\Le$ is asymmetric, it is still purely dissipative in the sense that it has non-negative real eigenvalues (precisely one of which is zero), as we show in the following result. Although it is not used in our analysis of structure-preserving compression, it conveys potentially interesting intuition about the marked chain.

\label{prop:markeig:section}
\begin{restatable}[Eigenvalues of marked chain \ProofLink{prop:markeig}]{proposition}{PropMarkeig} \label{prop:markeig}
    The eigenvalues of $\Rex$ and hence $-\Le$ are given by the $n$ eigenvalues of $R$, each according to its multiplicity, together with the $n - \vert \I\vert$ eigenvalues of $R_{\Ic,\Ic}$, each repeated with multiplicity $\vert \I \vert - 1$.
\end{restatable}
\noindent For context, note that $R_{\Ic,\Ic}$ is diagonally similar to $-L_{\Ic,\Ic}$,
hence diagonalizable with the same real negative eigenvalues.

Next we consider projections from the augmented space onto the marking and position spaces.
For this purpose, we first define the respective indicator matrices:
\begin{eqn}
    \label{eq:tildeproj}
    \Su \Eq \begin{bmatrix} 
        \Id_{\I,\I} \\ 
        \Id_{\I,\I} \otimes \One_{\Ic}
    \end{bmatrix}, \qquad
    \Qu \Eq \begin{bmatrix}
        \Id_{\I,\I} & \Zero_{\I, \Ic} \\
        \Zero_{\I \times \Ic, \I} & \One_{\I} \otimes \Id_{\Ic,\Ic}
    \end{bmatrix}.
\end{eqn}
Denoting the size of the augmented space as $m = \Ia + \Ia \cdot \Abs{\Ic}$, the sizes of the matrices $\Su$ and $\Qu$ are $m \times \Ia$ and $m \times n$, respectively.
By construction, $\Su_{(i,j),k} = \delta_{i,k}$ and $\Qu_{(i,j),l} = \delta_{j,l}$ for all $(i,j)$ in the augmented space, $k \in \I$, and $l \in [n]$. Indeed, these identities can be viewed as an alternative construction equivalent to \Cref{eq:tildeproj}.
Thus $\Su$ extracts the marking of a given state in the augmented space, while $\Qu$ extracts its position. In defining $\Qu$ blockwise, we have assumed without loss of generality that the states in $[n]$ are ordered so that the states in $\I$ appear first.

To work with \emph{orthogonal} projectors, we rescale $\Su$ and $\Qu$ to define 
\begin{eqn} \label{eq:marked-projections}
    \So \Eq \Diag(\he) \Su \DiagI(\hr), \qquad
    \Qo \Eq \Diag(\he) \Qu \DiagI(h)
\end{eqn}
and show that $\So$ has orthonormal columns (\Cref{th:s-orthogonality}), $\Qo$ has orthonormal columns (\Cref{th:q-contraction}), and $\Qo^\t \So = \Co$ (\Cref{th:q-s-equals-c}).

We can connect the dynamics of $\MarkedChain$ (the marked chain) to those of $\FullChain$ (the original chain) and $\InducedChain$ (the induced chain) using these projections (see also \Cref{fig:marked-chain} for intuition).
Note that observing the dynamics of $\MarkedChain$ in position space is equivalent to observing the dynamics of $\FullChain$.
This ultimately implies that $\Qo$ defines an invariant subspace of $\Le$ and that $L = \Qo^\t \Le \Qo$ (\Cref{th:qlq}). 
On the other hand, we can also establish that $\Lr = \So^\t \Le \So$ (\Cref{th:sls}), thereby connecting $\MarkedChain$ to $\InducedChain$.

Now we establish that while the exact dynamics $P(t)$ of the original Markov chain can be viewed as a projection of the marked chain dynamics, the structure-preserving compression  $\ObliqueP_{\I}(t)$ can be viewed as a projection of the marked chain dynamics after perturbation by a dissipative term which mixes states of a given marking together infinitely fast. More concretely, we consider the perturbation $\Le \rightarrow \Le + \alpha (\Id - \So \So^\t)$  and show that this perturbation yields our structure-preserving compression in the $\alpha \ra \infty$ limit.
Here, note that $\Id - \So \So^\t$ may be viewed as a (reducible) Laplacian which mixes states of a shared marking in the augmented space. Moreover, the null vector $\he$ of $\Le$ is contained in the null space of $\Id - \So \So^\t$ (\Cref{th:s-orthogonality}), so the stationary distribution is preserved by the perturbation. As promised, this result offers another probabilistic derivation of $\ObliqueP_{\I}(t)$. 
\label{th:induced-from-marked:section}
\begin{restatable}[Original and induced chains from the marked chain \ProofLink{th:induced-from-marked}]{proposition}{ThInducedFromMarked} \label{th:induced-from-marked}
    For $t > 0$, $\Le$ as in \Cref{def:marked-laplacian}, and $\Qo$, $\So$ as in \Cref{eq:marked-projections}:
    \begin{eqn}
        \SymP(t) =  \Qo^\t \exp( - \Le t) \,   \Qo, \qquad
        \ObliqueP_{\I}(t)  = \lim_{\alpha \ra \infty} \Qo^\t \Expt{\left[\Le + \alpha (\Id - \So \So^\t)\right]}  \Qo.
    \end{eqn}
\end{restatable}

While our construction of the marked chain and its linear-algebraic formulation may be of independent interest, in this work we are specifically motivated by our goal of offering an approximation theory for the structure-preserving compression. This theory is enabled precisely by the orthogonality of the projector $\So$. Recall that a simpler approach was blocked by the fact that $\Co$ does not have orthonormal columns.

Indeed, although it is not technically necessary for our approximation theory (\Cref{th:sp-nuclear-bound}) to define the ``marked chain Laplacian'' $\Le$, it \emph{is} necessary to define the projectors $W$ and $Q$,  whose definitions \Cref{eq:marked-projections} involve the null vector of $\Le$ and indeed are strongly motivated by the identities $L = \Qo^\t \Le \Qo$ (\Cref{th:qlq}) and  $\Lr = \So^\t \Le \So$ (\Cref{th:sls}) explaining how the original and reduced Laplacian may be obtained by projections of the ``marked chain Laplacian.''

Moreover, since our analysis makes use of the limit $\gamma \ra 0$ of killed chains, it will be important for our analysis to extend the definitions of $W$ and $Q$ to the killed setting where $\gamma > 0$, ensuring that two properties analogous to those mentioned above still hold: (1) $W_\gamma$ has orthonormal columns and (2) $Q_\gamma^\t W_\gamma = C_\gamma$. We define such extensions $\Sa$  and $\Qa$ in \Cref{def:killed-projections} of \Cref{a:marked-properties} in terms of suitable $\mathring{\pi}_\gamma$ and $\mathring{h}_\gamma$. It follows closely from their definitions that these objects converge to $W$, $Q$, $\mathring{\pi}$, and $\mathring{h}$, respectively, as $\gamma \ra 0$. Given these definitions, the aforementioned \Cref{th:s-orthogonality} and \Cref{th:q-s-equals-c} also ensure the extended properties (1) and (2), as desired.

For additional context, observe that the appropriate definition of the rate matrix and its rescaling are as follows:
\begin{eqn} \label{eq:Lea}
    \mathring{R}_\gamma \Eq  \mathring{R} - \gamma \Id, \qquad  \Lea \Eq - \mathrm{diag}(\mathring{h}_\gamma) \, \mathring{R}_\gamma \, \mathrm{diag}^{-1} (\mathring{h}_\gamma).
\end{eqn}
Under these definitions, the property $\Lr = \So^\t \Le \So$ extends to $\Lrk = \So^\t_\gamma \Le_\gamma \So_\gamma$ for $\gamma > 0$ (also shown in \Cref{th:sls}). (Meanwhile, note with caution that certain identities listed above do \emph{not} extend to the killed case: $Q_\gamma$ does not have orthonormal columns and $L_\gamma \neq Q_\gamma^\t \Le Q_\gamma$, for $\gamma > 0$.)

Then the same proof as for \Cref{th:induced-from-marked}, given that $\Lrk = \Sa^\t \Lea \Sa$ by \Cref{th:sls}, establishes that $ \ObliqueP_{\I,\Kill}(t)$ as in \Cref{eq:Pspgamma} can be recovered as: 
\begin{eqn} \label{eq:oblique-p-killed}
    \ObliqueP_{\I,\Kill}(t) = \lim_{\alpha \ra \infty} \Qa^\t \Expt{\left[\Lea + \alpha (\Id - \Sa \Sa^\t)\right]} \Qa.
\end{eqn}
We will not use this fact in our approximation theory, but it offers a natural motivation for the definition of $\ObliqueP_{\I,\Kill}(t)$ in terms of the marked killed chain.

\subsection{Approximation theory \label{s:oblique-approximation}}

In this section, we develop a theory supporting the accuracy of our structure-preserving approximation, culminating with \Cref{th:sp-nuclear-bound}. 
The  proof requires substantially more work and complication than those of \Cref{th:dissipative,th:ortho-markov} due to the ``obliqueness'' of the compression, i.e., the fact that $\Co$ does not have orthonormal columns.
We will take the following approach:
\begin{enumerate}
    \item Bound the error of structure-preserving compression relative to projective compression, in terms of quantities $\Obs$ related to a notion of obliqueness (\Cref{def:obliqueness}). 
    The bound is proved via contour integration by way of the marked chain, using the fact that our structure-preserving compression can be viewed as an orthogonal projection from the marked chain's expanded state space.
    \item Show by probabilistic arguments that $\On$ is at most $\Ia$ times larger than $\En$, yielding \Cref{th:sp-nuclear-bound}.
\end{enumerate}

\subsubsection{Contour integration bound}



In this section, we show how we can invoke Cauchy's integral formula to bound the difference between the structure-preserving compression and the \emph{projective compression}. Note that we do not \emph{directly} bound the difference between the structure-preserving compression and the exact dynamics. To bound this difference, it suffices to examine the difference between the two compressions via \Cref{th:ortho-markov} and the triangle inequality.

To rigorously achieve our results, we will work throughout this section with Markov chains killed at rate $\Kill$ to avoid singular matrices related to the Markov chain steady state.
Therefore, we will pose intermediate results in terms of the killed Laplacian $\La$, its inverse $\Ka$, and the committor analog $\Ca$ as defined in \Cref{def:killed-committor}.
In the final results, we will take the limit of $\Kill \ra 0$, and on first reading the $\Kill$-dependence can mostly be overlooked for the purpose of intuitive understanding.

To motivate the structure of our argument, recall that we introduced a contour integral approach in \Cref{s:contour} to bound the difference of matrix exponentials of equal dimension. In particular, for the \emph{projective} compression specified by $\Va \Eq \Orth( [K_\gamma]_{:,\I} ) = \Orth(\Ca)$  (cf. \Cref{def:killed-committor}) we obtained by Cauchy's integral formula (cf. \Cref{lem:contour0}):
\begin{eqn}
    \label{eq:orthogonal-cauchy}
    \OrthP_{\I,\Kill}(t) = \Va \Expt{\Va^\t \La \Va} \Va^\t = \frac{1}{2 \pi i} \oint_{\Contour} e^{\Minus 1/z} \OrthRes^\Kill(z) \; \odif z
\end{eqn}
where $\Contour$ is a contour surrounding the eigenvalues of $K_\gamma / t$ and $\OrthRes^\Kill$ is the resolvent (identical to \Cref{eq:orth-resolvent} with $\Ka$ in place of $K$), which can be rewritten: 
\begin{eqn}
    \label{eq:orth-res-def}
    \OrthRes^\Kill(z) \Eq (z \Id - t^{-1} \Ca \Lrk^\Mo \Ca^\t)^\Mo.
\end{eqn}

Next we establish a similar formula for \emph{structure-preserving} compression, using the fact that $\Lrk = \Sa^\t \Lea \Sa$ by \Cref{th:sls}.
Here we interpret the structure-preserving compression of $\FullChain$ using $\Ca$ (which \emph{is not} an orthogonal projection) as the projective compression of $\MarkedChain$ using $\Sa$ (which \emph{is} an orthogonal projection) followed by later projection into the position space via $\Qa$.
($\Sa$ and $\Qa$ are precisely defined in \Cref{def:killed-projections}; they are analogs of $\So$ and $\Qo$ for the killed chain, respectively.) 
Using this logic and working through the linear algebraic simplications, we obtain:

\label{th:oblique-cauchy:section}
\begin{restatable}[Structure-preserving contour integral using marked chain \ProofLink{th:oblique-cauchy}]{proposition}{ThObliqueCauchy} \label{th:oblique-cauchy}
    For a Laplacian $\La$ killed at rate $\Kill > 0$, $\Ca$ as in \Cref{def:killed-committor}, $\ObliqueP_{\I,\Kill}$ as in \cref{eq:oblique-p-killed}, and $t>0$: 
    \begin{eqn} \label{eq:SplusT}
        \ObliqueP_{\I,\Kill}(t) = \Ca e^{- \Lrk t} \Ca^\t = \frac{1}{2 \pi i} \oint_\Contour e^{\Minus 1/z} \left[ \ObliqueRes^\Kill(z) + \mathcal{T}^{\,\Kill}(z) \right] \; \odif z
    \end{eqn}
    where $\Contour$ is a contour in the strict right half-plane enclosing the eigenvalues of $(t \Lrk)^{-1}$ and
    \begin{eqn}
        \label{eq:oblique-res-def}
        \ObliqueRes^\Kill(z) \ & \Eq \  \lrp{z \Id - \Ca \left[ \Lrk t - z^\Mo (\Id - \Ca^\t \Ca)\right]^\Mo \Ca^\t}^\Mo, \\ 
        \mathcal{T}^{\,\Kill}(z) \ & \Eq \  z^{-1} (Q_\gamma^\t Q_\gamma - \Id).
    \end{eqn}
\end{restatable}

\noindent We comment that the formula yielded by \Cref{th:oblique-cauchy} is substantially different than what is yielded by direct application of Cauchy's integral formula to expand the inner exponential $e^{-\Lrk t}$ as a contour integral. Such an alternative approach is ultimately not useful to us. Also we point out that $\mathcal{T}^\gamma (z)$, defined in \Cref{eq:oblique-res-def}, vanishes as $\gamma \ra 0$ and does not end up contributing to the error bound of Theorem \hyperlink{Theorem3*}{3*}. However, we have to treat it delicately in our analysis because the contour must depend on $\gamma$ to enclose the spectrum of $(t \Lrk)^{-1}$.

Comparing $\OrthRes^\Kill(z)$ from \Cref{eq:orth-res-def} and $\ObliqueRes^\Kill(z)$ from \Cref{eq:oblique-res-def}, we see that the non-orthogonality $\Id - \Ca^\t \Ca$ distinguishes the projective and structure-preserving compressions from one another.
Given that $\norm{\Ca}_2 \leq 1$ (\Cref{th:committor-norm}), $\Id - \Ca^\t \Ca$ is positive semidefinite, and it is tempting to develop bounds based directly on $\norm{\Id - \Ca^\t \Ca}_2$. However, we have observed empirically that such bounds are not very tight. Indeed, we observe that non-orthogonality may be significant even in a good approximation.

Instead, we examine the difference between resolvents of the projective and structure-preserving compressions $\OrthRes^\Kill(z) - \ObliqueRes^\Kill(z)$ in a different way. 
(See \Cref{a:sp-contour} for proofs and additional motivation.)
Specifically, we consider the following notion of ``obliqueness,'' which (having units of time) weights the lack of orthogonality more heavily along the slower modes of the compressed dynamics:

\begin{definition}[Obliqueness of the structure-preserving compression]
    \label{def:obliqueness}
    Given a Laplacian $\La$ killed at rate $\gamma > 0$, a non-empty subset $\I$, $\Ca$ as in \Cref{def:killed-committor}, and recalling that $\Lrk = \Ca^\t \La \Ca$, define: 
    \begin{eqn}
        \Ob_\Kill &\Eq \Lrk^\Mh (\Id - \Ca^\t \Ca) \Lrk^\Mh.
    \end{eqn}
    Further define 
    \begin{eqn} 
        \Obs^\Kill &\Eq \Norms{ \Ob_\Kill }, 
        \qquad \Obs \Eq \lim_{\Kill \ra 0} \Obs^\Kill.
        \label{eq:oblique-norm-def}
    \end{eqn}
\end{definition}
\noindent We prove in \Cref{th:recurrent-obliqueness} that the limits in \Cref{eq:oblique-norm-def} do in fact exist and moreover admit closed-form expressions in terms of $K$. Note that $\Lrk = \Ca^\t \La \Ca$  and $\Id - \Ca^\t \Ca$ are both symmetric positive definite (the latter by \Cref{th:committor-norm}). Therefore $\Ob_\Kill$ is symmetric positive definite, and, in particular, $\On^\Kill = \Tr[\Ob_\Kill]$.

This notion of obliqueness allows us to state the following bound on the resolvent difference.

\newcommand{\SqrtResolvent}{\mathcal{W}}

\label{th:resolvent-obliqueness:section}
\begin{restatable}[Resolvent difference bound using obliqueness \ProofLink{th:resolvent-obliqueness}]{lemma}{ThResolventObliqueness} \label{th:resolvent-obliqueness}
    For $z \in \mathbb{C}$ with $\re(z) > 0$, $t > 0$, $\Obs^\Kill$ as in \Cref{def:obliqueness}, $\OrthRes^\Kill$ as in \cref{eq:orth-res-def}, and $\ObliqueRes^\Kill$ as in \cref{eq:oblique-res-def}:
    \begin{eqn}
        \Norms{\ObliqueRes^\Kill(z) - \OrthRes^\Kill(z)} 
        \leq \frac{1}{2}\lrp{1 - \frac{\re(z)}{\Abs{z}}}^\Mo \frac{\Obs^\Kill}{t \cdot \Abs{z}^2}. 
        \label{eq:resolvent-diff-bound}
    \end{eqn}
\end{restatable}

Using the resolvent difference bound, we achieve a bound on $\Norms{\ObliqueP_{\I}(t) - \OrthP_{\I}(t)}$, the difference between the structure-preserving and projective compressions. In \Cref{a:sp-contour}, immediately before the proof of Theorem \hyperlink{Theorem3*}{3*}, we offer evidence supporting the optimality of the proof technique based on \Cref{th:resolvent-obliqueness}.

\label{th:oblique-sum-bound:section}
\begin{restatable}[Composite bound on structure-preserving compression \ProofLink{th:oblique-sum-bound}]{Theorem 3*}{ThObliqueSumBound} \label{th:oblique-sum-bound}
    \hypertarget{Theorem3*}{} With $\Obs = \Obs(\I)$ as in \Cref{def:obliqueness}, and $\Eps = \Eps(\I)$ as in \Cref{def:markov-epsilon}, the difference between the projective and structure-preserving compressions is bounded as:
    \begin{eqn}
        \label{eq:oblique-bound}
        \Norms{\ObliqueP_{\I}(t) - \OrthP_{\I}(t)} \leq \frac{2}{\pi}  \frac{\Obs(\I)}{t},
    \end{eqn}
    and the approximation error of the structure-preserving compression is bounded as:
    \begin{eqn}
        \label{eq:composed-bound}
        \Norms{\ObliqueP_{\I}(t) - \SymP(t)} \leq \frac{1}{t} \lrp{\frac{3\sqrt3}{2\pi} \Eps(\I) + \frac{2}{\pi} \Obs(\I)}.
    \end{eqn}
\end{restatable}
\noindent 
The bound \Cref{eq:oblique-bound} is proved with a contour integration argument similar to that in the proof of \Cref{th:dissipative}, using the resolvent difference bound of \Cref{th:resolvent-obliqueness} and taking the limit of $\Kill \ra 0$ appropriately.
Note that \Cref{eq:composed-bound} follows directly from \Cref{eq:oblique-bound} via \Cref{th:ortho-markov} and the triangle inequality.


We will demonstrate empirically in \Cref{s:experiments} that the \emph{a posteriori} bounds \cref{eq:composed-bound} are relatively tight in practice.
Indeed, $\Obs$ is comparable to or smaller than $\Eps$ in practice, so the structure-preserving compression bound \Cref{eq:composed-bound} is empirically similar to our projective compression error bound from \Cref{th:ortho-markov}.


\subsubsection{Bounding obliqueness via probabilistic interpretation}

Next we turn toward showing that the obliqueness $\On$ appearing in Theorem \hyperlink{Theorem3*}{3*} can be controlled in terms of the nuclear norm \Nystrom{} error $\En$. In the language of the \Cref{s:structure-preserving} introduction, this argument converts an \emph{a posteriori} bound (which is empirically quite tight, cf., \Cref{s:experiments}) with a somewhat looser but still quite functional \emph{a priori} bound that can be controlled via nuclear column selection (cf. \Cref{s:algs}).
We rely on probabilistic arguments to achieve this proof.


In \Cref{th:recurrent-obliqueness}, we compute $\On$ in closed form in terms of $K$.
To leverage this result, we define the $n \times n$ matrix $H$ of hitting times (also known as mean first passage times):
\begin{eqn} \label{eq:hitting-time-def}
    H_{i,j} \Eq \Ex(\tau_\lrb{j} \mid X_0 = i)
 = \Ex(\min \lrb{t: t \geq 0, \, X_t = j} \mid X_0 = i).
\end{eqn}
With this definition we can obtain a remarkably simple expression for $\On$:

\label{th:obliqueness-from-hitting:section}
\begin{restatable}[Obliqueness in terms of first passage times \ProofLink{th:obliqueness-from-hitting}]{proposition}{ThObliquenessFromHitting} \label{th:obliqueness-from-hitting}
    Given $\On$ as in \Cref{def:obliqueness} and $H$ as in \cref{eq:hitting-time-def}:
    \begin{eqn} \label{eq:nuclear-using-hitting}
        \On = \Tr[H_{:,\I}^\t  \Diag(\pi) \Cu].
    \end{eqn}
\end{restatable}

Using \Cref{th:obliqueness-from-hitting} and further probabilistic arguments, we bound the obliqueness $\On$ in terms of the \Nystrom{} approximation error in the nuclear norm:
\begin{proposition}[Bounding obliqueness in terms of the \Nystrom{} approximation error]
\label{th:probabilistic-obliqueness-factor}
    Given a non-empty subset $\I$, $\En (\I)$ as in \Cref{def:markov-epsilon}, and $\On = \On (\I)$ as in \Cref{def:obliqueness}:
    \begin{eqn}
        \On (\I) \leq \Ia  \, \En(\I).
    \end{eqn}
\end{proposition}

\noindent Observe that \Cref{th:probabilistic-obliqueness-factor}, together with Theorem \hyperlink{Theorem3*}{3*}, yields \Cref{th:sp-nuclear-bound} immediately.
 

\begin{proof}
In the proof, we tackle each of the $\Ia$ contributions to the trace in \cref{eq:nuclear-using-hitting} separately. 
Indeed, observe that it suffices to show that $H_{:,i}^\t \Diag(\pi) \Cu_{:,i} \leq \En (\I)$ for each $i \in \I$. Therefore fix $i \in \I$, and consider the ratio:   
\begin{eqn}
    \frac{H_{:,i}^\t \Diag(\pi) \Cu_{:,i}}{\En (\I)} = \frac{\sum_{k \in \Ic} \pi_k H_{k,i} \Cu_{k,i}}{\sum_{k \in \Ic} (L_{\Ic,\Ic}^\Mo)_{k,k}}, 
\end{eqn}
where the denominator on the right-hand side is an alternative expression for $\En (\I)$ derived using Schur complement identities in \eqref{eq:killed-epsilon}.
Then in turn it suffices to show that $\pi_k H_{k,i} \Cu_{k,i} \leq (L_{\Ic,\Ic}^\Mo)_{k,k}$ for any $k \in \Ic$.

Therefore fix $k \in \Ic$ as well as $i \in \I$. 
In fact, we will show that the stronger condition  
\begin{eqn} \label{eq:obliqueWTS}
    \frac{ {\pi_k (H_{k,i} + H_{i,k}) \Cu_{k,i}}}{ (L_{\Ic,\Ic}^\Mo)_{k,k} } \leq 1
\end{eqn}
holds using a probabilistic technique.

\begin{figure}
    \centering
    \includegraphics[width=5.5in]{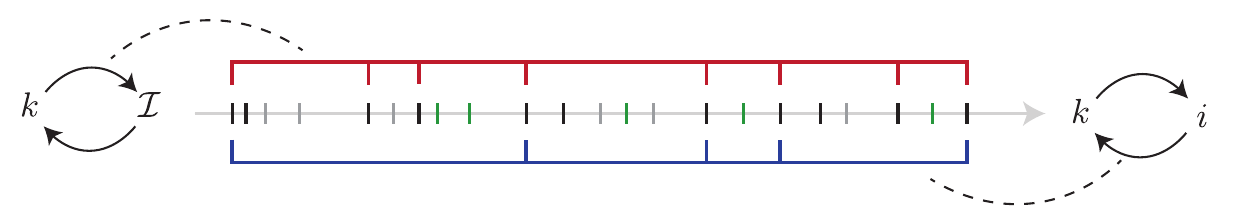}
    \caption{
        \label{fig:timeline-equivalence}
        Decomposition of a trajectory into cycles for example states $k \notin \I$ and $i \in \I$.
        The middle timeline depicts each jump in the Markov chain according to the state being jumped to, with $k$ annotated in black, $i \in \I$ in green, $\Iminus{i}$ in gray, and all other states unlabeled.
        Cycles corresponding to $k \ra \I \ra k$ are demarcated on the top in red, while
        cycles corresponding to $k \ra i \ra k$  are demarcated on the bottom in blue.
        Observe that endpoints of the latter comprise a subset of endpoints of the former.
        In \Cref{th:probabilistic-obliqueness-factor}, the average lengths and probabilities of these cycle types are exploited in order to bound the the obliqueness norm $\On$ in terms of the \Nystrom{} error $\En$.
    }
\end{figure}

We will define two key stopping times. The first is $\tau_{ i\ra k }$, ``the first return to $k$ after reaching $i$'': 
\begin{eqn} \label{eq:i-cycle-stop}
    \tau_{ i\ra k } = \min\lrb{t: t \geq \tau_{\{i\}}, \, X_t = k }.
\end{eqn}
The second is $\tau_{ \I \ra k }$, ``the first return to $k$ after reaching any state in $\I$'':
\begin{eqn} \label{eq:I-cycle-stop}
    \tau_{ \I \ra k } = \min\lrb{t: t \geq  \tau_{\I} , \,  X_t = k }.
\end{eqn}

We will think about the conditional expectations of these quantities given $X_0 = k$, which are the expected lengths of cycles $k \ra i \ra k$ and $k \ra \I \ra k$, respectively. 
We can directly compute the first of these, also known as the commute time: 
\begin{eqn} \label{eq:i-cycle}
    \Ex_{k} ( \tau_{ i \ra k }  ) = H_{k,i} + H_{i,k},
\end{eqn}
where we use $\Ex_k$ going forward to denote the expectation conditioned on $X_0 = k$.

A direct consequence of the renewal-reward theorem (cf. \cite{Aldous1995-reversible}, Proposition~2.3 for a statement) is that 
\begin{eqn}\label{eq:renewalreward}
\Ex_k ( \, \text{total time spent in $k$ before \ensuremath{\tau_{ \I \ra k }}} \, ) = \pi_k \, \Ex_k ( \tau_{ \I \ra k }  ).
\end{eqn}
It is also useful to realize that the left-hand side can alternatively be written as 
\begin{eqn}\label{eq:expectationtexts}
    \Ex_k ( \, \text{total time spent in $k$ before leaving \ensuremath{\Ic}} \, ),
\end{eqn}
since the stopping time $\tau_{ \I \ra k }$ occurs upon the \emph{first} return to $k$ after hitting $\I$. Then \Cref{eq:expectationtexts} can be computed as 
\begin{eqn}\label{eq:RRRHS}
    \int_0^\infty \Id_{k,:} \, e^{R_{\Ic,\Ic}t} \, \Id_{:,k}\, \odif t= -(R_{\Ic,\Ic}^{-1})_{k,k}= (L_{\Ic,\Ic}^{-1})_{k,k},
\end{eqn}
which furnishes the identity 
\begin{eqn} \label{eq:exstop2}
    \Ex_k ( \tau_{ \I \ra k }  ) = \frac{(L_{\Ic,\Ic}^{-1})_{k,k}}{\pi_k}
\end{eqn}
via \Cref{eq:renewalreward}.


Now we consider segmenting a long trajectory of $\FullChain$ into either type of cycle (\Cref{fig:timeline-equivalence}). Concretely we fix $X_0 = k$ and let $T$ denote a trajectory length. By defining the renewal process whose inter-renewal time is distributed as $\tau_{i \ra k}$, we can apply the elementary renewal theorem \cite{Ross1983} to deduce that (almost surely) 
\begin{eqn}\label{eq:renewal1}
    \lim_{T \ra \infty} \frac{\# (\text{cycles $k \ra i \ra k$})}{ T } = \frac{1}{\Ex_k (\tau_{i\ra k} ) },
\end{eqn}
where on the left-hand side we understand the numerator to indicate the number of cycles completed before time $T$. Similarly, 
\begin{eqn}\label{eq:renewal2}
    \lim_{T \ra \infty} \frac{\# (\text{cycles $k \ra \I \ra k$})}{ T } = \frac{1}{\Ex_k (\tau_{\I\ra k} )}
\end{eqn}
holds almost surely.

Then by combining \Cref{eq:renewal1}, \Cref{eq:renewal2}, \Cref{eq:i-cycle}, and \Cref{eq:exstop2}, we see that (almost surely)
\begin{eqn} \label{eq:renewallim}
    \lim_{T \ra \infty} \ \frac{\# (\text{cycles $k \ra \I \ra k$})}{\# (\text{cycles $k \ra i \ra k$})}  = \frac{\pi_k (H_{k,i} + H_{i,k})}{(L_{\Ic,\Ic}^\Mo)_{k,k}}.
\end{eqn}
Meanwhile, by the law of large numbers: 
\begin{eqn}\label{eq:lln}
    \lim_{T\ra\infty}\frac{\#(\text{cycles \ensuremath{k\ra\I\ra k} including \ensuremath{i} first from \ensuremath{\I})}}{\#(\text{cycles \ensuremath{k\ra\I\ra k})}}= \Prob_k ( \,  \text{hit $i$ first from $\I$} \, )  = \tilde{C}_{k,i},
\end{eqn}
where $\Prob_k$ denotes the expectation conditioned on $X_0 = k$.

Then by multiplying the limits \Cref{eq:renewallim} and \Cref{eq:lln}, we deduce that 
\begin{eqn}\label{eq:almostdone}
    \frac{\pi_k (H_{k,i} + H_{i,k}) \Cu_{k,i}}{(L_{\Ic,\Ic}^\Mo)_{k,k}} = \lim_{T \ra \infty} \frac{ \#(\text{cycles $k \ra \I \ra k$ including $i$ first from $\I$})}{ \#(\text{cycles $k \ra i \ra k$})}.
\end{eqn}
Note that the cycles $k \ra i \ra k$ can be put in bijection with the subset of cycles $k \ra \I \ra k$ that include $i$. Indeed, every cycle $k \ra i \ra k$ is composed of one or more cycles $k \ra \I \ra k$, and it is precisely the last of these which is a cycle $k \ra \I \ra k$ that includes $i$. (See \Cref{fig:timeline-equivalence} for illustration.) Therefore we can rewrite the right-hand side of \Cref{eq:almostdone} as 
\begin{eqn}\label{eq:done}
    \lim_{T \ra \infty} \frac{ \#(\text{cycles $k \ra \I \ra k$ including $i$ first from $\I$})}{ \#(\text{cycles $k \ra \I \ra k$ including $i$})}.
\end{eqn}
On the right-hand side, since cycles satisfying the numerator condition automatically satisfy the denominator condition, the ratio is at most one. 
This establishes \Cref{eq:obliqueWTS}, as was to be shown, and concludes the proof.
\end{proof}

\section{Controlling \Nystrom{} error with nuclear column subset selection \label{s:algs}}

Overall, \Cref{s:projective,s:structure-preserving} show that the \Nystrom{} error $\En(\I)$ can be used to bound the approximation error of both the projective and structure-preserving compressions of Markov dynamics. 
Now we turn toward the actual algorithmic selection of the subset $\I$, with a view toward controlling $\En(\I)$ while maintaining $\Ia$ as small as possible. Indeed, note that a smaller subset $\I$ yields a cheaper and simpler approximation of the original dynamics, which is useful for both interpretation and computation.
The problem of selecting such $\I$ can be viewed as an extension of the common column subset selection problem (CSSP, e.g., \cite{Deshpande2010-efficient,Boutsidis2009-improved,Cortinovis2024-adaptive,Chen2022-randomly}), where some additional care must be taken to account for the stationary distribution~\cite{Fornace2024-column}.
In particular we approach this problem using nuclear maximization, a method developed and studied in our prior work \cite{Fornace2024-column}, where it was shown to be highly effective in theory and practice. In this section, we explain the relevance of this algorithm and prove an extension of our previous theoretical results.

\Cref{alg:cssp} is a high-level pseudocode for selection of a subset $\I$ of size $k$ using nuclear maximization. 
We allow for estimation of the so-called nuclear scores $v_i$ up to a relative error $\Rand \geq 0$. 
In \cite{Fornace2024-column}, it is shown that a fixed relative error may be obtained with high probability by using randomized sketching techniques \cite{Meyer2021-hutch++} and fast Laplacian solvers \cite{Chen2020-rchol,Kyng2016-approximate,Gao2023-robust}, achieving an overall computational complexity of $\tilde{O}(nk^2)$ for graph Laplacians of constant degree. 
More precisely, assuming an approximate Cholesky factorization $U$ with $O(\mathrm{nnz}(L))$ entries and a preconditioned Laplacian $U^{+} L U$ of condition number $O(1)$, our algorithm runs in $\tilde{O}(n k^2 z + k z \, \mathrm{nnz}(L))$ time for sketching dimension $z$, where $z = O(\log n)$ to guarantee relative error $\Rand$ with high probability. 
Up to the specification of the Laplacian preconditioner, the algorithm is matrix-free in terms of $L$. We refer the reader to \cite{Fornace2024-column} for further algorithmic details as well as \url{https://github.com/mfornace/nuclear-score-maximization} for an open-source implementation.

\begin{algorithm}[h] \begin{algorithmic}
    \setstretch{1.2}
    \caption{Markov chain reduction via nuclear maximization \label{alg:cssp}}
  \Function{SelectMarkovChainIndices}{$k, \Rand$}
    \State $\I \gets \emptyset$
    \For{$j = 1, \dots, k$}
        \ForAll{$i \in \Ic$}
            \State $v_i \gets \En(\I) - \En(\I \cup \lrb{i})$, calculated up to relative error at most $\Rand$
        \EndFor
        \State $\I \gets \I \cup \lrb{\argmax_{i \in \Ic} v_i}$
    \EndFor
    \State \Return $\I$
  \EndFunction
\end{algorithmic} \end{algorithm}

In the more typical setting of \Nystrom{} approximation of a positive semidefinite matrix $K$, CSSP algorithms are often designed to control the approximation error $\nnorm{K - \Kn}$ in terms of the size of the selected subset $\I$. Another recent example includes the randomly pivoted Cholesky algorithm (RPCholesky) \cite{Chen2022-randomly}. 
Meanwhile, our nuclear maximization algorithm attempts to control this nuclear norm approximation error by greedy augmentation of the selected subset. The approach can be equivalently motivated as greedy maximization of the nuclear norm of the \Nystrom{} approximation, and this derivation inspires its name.
In our graph Laplacian setting, this approach is applied to the diverging matrix $\Ka$ in the $\Kill \ra 0$ limit (cf. \Cref{def:markov-epsilon}), and it is shown in \cite{Fornace2024-column} that a stable algorithm emerges in the limit.  Note that in \cite{Fornace2024-column}, we also introduced a suitable matrix-free implementation of RPCholesky suited to this graph Laplacian setting~\cite{Fornace2024-column}, but going forward we focus only on the nuclear maximization algorithm.


Now we discuss the theoretical guarantees enjoyed by \Cref{alg:cssp}.
First, column selection via nuclear maximization is exponentially close to optimal in the following sense:
\begin{restatable}[Optimality bound for nuclear maximization on Laplacians \cite{Fornace2024-column}]{Theorem 4A}{ThNuclearOptimality}\label{th:nuclear-optimality}
    \hypertarget{Theorem4A}{} Given a Laplacian $L$ with $K = L^+$ and defining $\En(\I)$ as in \Cref{def:markov-epsilon}, let $\I$ be a subset of size $k$ yielded by \Cref{alg:cssp} with maximum relative error $\Rand \geq 0$.
    Let $\Opt_s$ denote an $s$-subset which exactly minimizes $\En(\I)$ over all $s$-subsets, $1 \leq s \leq k$. Then 
    \begin{eqn} \label{eq:laplacian-trace-bound}
        \frac{\En(\I) - \En(\Opt_s)}{\Tr[K]}
    \leq (2 + \Rand) e^{-(k - 1)/(s \Orand)}.
    \end{eqn}
\end{restatable}
\noindent In particular, the analysis of \cite{Fornace2024-column} relied on the Markov chain interpretation of $L$ to prove Theorem \hyperlink{Theorem4A}{4A}.
Among algorithmic guarantees for CSSP, which is NP-complete in general \cite{Shitov2021-column}, such a bound relative to the optimal combinatorial solution appears to be unique to our knowledge.

Now we prove that the performance of nuclear maximization can also be guaranteed in terms of the eigenvalues of $L$ by using determinantal point process (DPP) sampling \cite{Guruswami2012-optimal,Belabbas2009-spectral} to furnish a subset. We make use of the celebrated bound on the DPP expectation \cite{Chen2022-randomly,Guruswami2012-optimal,Belabbas2009-spectral,Derezinski2021-determinantal} but must make suitable modifications to deal with the stationary distribution.

\label{th:spectral-cssp:section}
\begin{restatable}[Spectral bound for nuclear maximization on Laplacians \ProofLink{th:spectral-cssp}]{Theorem 4B}{ThSpectralCSSP} \label{th:spectral-cssp}
    \hypertarget{Theorem4B}{} Given a Laplacian $L$ and defining $\En(\I)$ as in \Cref{def:markov-epsilon}, let $\I$ be a subset of size $k$ yielded by \Cref{alg:cssp} with maximum relative error $\Rand \geq 0$.
    Let $\Tr^{(r)}[K]$ denote the sum of the $r$ largest eigenvalues of $K = L^+$.
    Then for any $s$ with $r < s \leq k$:
\begin{eqn} \label{eq:spectral-cssp}
    \En(\I) \leq (2 + \Rand) \, \Tr[K] \, e^{-(k - 1)/(s \Orand)}  + \frac{s+1}{s-r} \, (\Tr[K] - \Tr^{(r)}[K]). 
\end{eqn}
\end{restatable}

In the next section, we demonstrate in practice the performance of nuclear maximization as applied to Markov chain reduction, showing that high-fidelity approximations are achieved for both projective and structure-preserving compressions.

\section{Computational experiments \label{s:experiments}}

To examine the quality of our Markov chain compression schemes, we carried out numerical studies of three example systems (\Cref{tab:systems}).

System 1 is derived from a webgraph of chameleon-related Wikipedia pages with edges yielded by hyperlinks between the pages \cite{Rozemberczki2021-multi-scale} (downloaded from the SNAP dataset collection \cite{Leskovec2014-snap}).
Given the undirected adjacency matrix $\Delta$ of 0's and 1's, we assumed a stationary distribution of $\pi \propto \Delta \One$ to yield rate matrix $R = \DiagI(\pi) (\Delta - \Diag(\Delta \One))$.
We assumed an undirected adjacency matrix to ensure the reversibility of the dynamics, for simplicity.

Systems 2 and 3 are derived from kinetic modeling of nucleic acid secondary structures.
For each system, the continuous time Markov chain was constructed over all unpseudoknotted secondary structures compatible with a randomly generated DNA strand.
We utilized the \texttt{dna04} parameters in NUPACK 4.0.1 \cite{Fornace2022-nupack} at default (37\textdegree C, 1 M Na$^+$) conditions in the \texttt{nostacking} ensemble; with these parameters, only Watson-Crick (\texttt{A}$\cdot$\texttt{T} and \texttt{C}$\cdot$\texttt{G}) base pairs are considered \cite{Fornace2020-unified}.
By construction, the stationary probability of structure $s$ is proportional to its Boltzmann factor $\pi(s) \propto \exp(\Minus \Delta G(s) / k_B T)$, with $\Delta G(s)$ defined by the NUPACK 4 free energy model and \texttt{dna04} parameters.
Transitions were assumed to occur via the cleavage or formation of a single base pair \cite{Flamm2000-rna,Schaeffer2015-stochastic,Fornace2022-computational}.
Transition rates were generated according to the Kawasaki rate function, i.e., with rate $R_{i,j} \propto \sqrt{\pi_j / \pi_i} = \h_j / \h_i$ for neighboring secondary structures $s_i$ and $s_j$ \cite{Schaeffer2015-stochastic}.
System~2 was chosen to yield $n$ at the upper end of the regime in which $P(t)$ may be exactly computed via diagonalization.
On the other hand, System~3 (from \cite{Fornace2024-column}) is of a dimension ($n \approx 10^6$) which prohibits direct linear-algebraic techniques.

\begin{table}
    \centering
    \begin{tabular}{crrl}
        \toprule
        System & $n$ & nnz & Construction \\
        \midrule
        1 & 2,277 & 65,019 & webgraph of chameleon-related Wikipedia articles \cite{Rozemberczki2021-multi-scale} \\
        2 & 10,250 & 87,304 & random DNA sequence of 28 nucleotides \\
        3 & 1,087,264 & 12,519,798 & random DNA sequence of 40 nucleotides \cite{Fornace2024-column} \\
        \bottomrule
    \end{tabular}
    \caption{ \label{tab:systems}
        Dimension, number of nonzeros, and origin of systems investigated in our computational experiments.
    }
\end{table}

\subsection{Minimization of error via column subset selection \label{s:experiment-cssp}}

We performed nuclear maximization to choose $\I$ up to $\Ia = 100$ for Systems~1 and 2 and up to $\Ia = 25$ for System 3. 
For System~1, we used the deterministic Algorithm~3 of \cite{Fornace2024-column}.
For Systems~2 and 3, we used the randomized matrix-free Algorithm~6 of \cite{Fornace2024-column} with $z=200$ random vectors used for score estimation.
This algorithm uses fast Laplacian solvers with nearly linear scaling in order to efficiently optimize $\En (\I)$.
We used the rchol library \cite{Chen2020-rchol} to compute an approximate Cholesky factorization of the input Laplacian.
See \cite{Fornace2024-column} for precise procedures and additional performance validation versus alternative algorithms.

\begin{figure}
    \centering
    \includegraphics[width=0.95\textwidth]{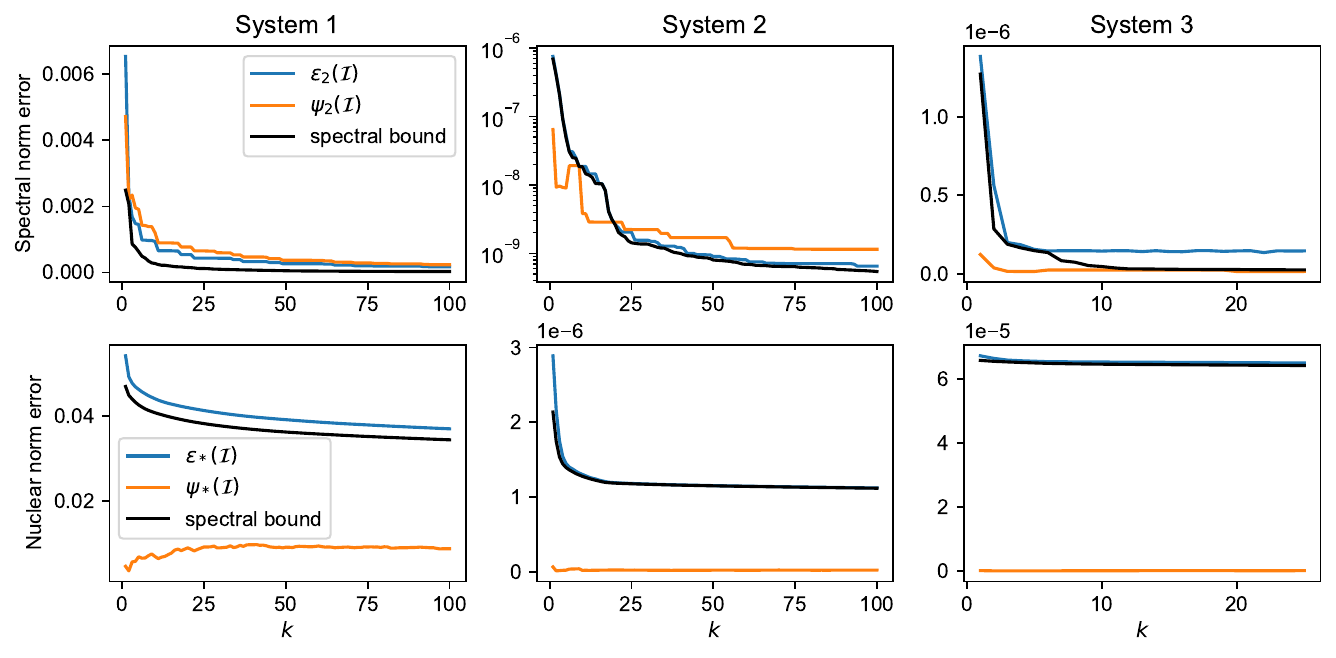}
    \caption{\label{fig:cssp}
        Performance of nuclear maximization as a function of $k = \vert \I \vert $, the number of columns selected, on Systems 1, 2, and 3 (\Cref{tab:systems}).
        The top row shows the \Nystrom{} approximation spectral norm error $\Es(\I)$ compared to the theoretical lower bound derived from the eigenvalues of $L^+$, as well as the spectral norm of the obliqueness $\Os(\I)$.
        The bottom row depicts the \Nystrom{} approximation nuclear norm error $\En(\I)$ compared to the theoretical lower bound derived from the eigenvalues of $L^+$, as well as the nuclear norm of the obliqueness $\On(\I)$.
    }
\end{figure}


\Cref{fig:cssp} shows the low-rank approximation quality $\Eps (\I)$ for each of Systems 1-3 achieved by solving the column subset selection problem with nuclear maximization, as a function of $k = \Abs{\I}$, the number of selected columns.
(For System~3, $\Es (\I)$ was estimated using the power method with random initialization and 20 iterations, while $\En (\I)$ was estimated using the Hutchinson trace estimator \cite{Meyer2021-hutch++} with $z=200$ random vectors.
The results were compared against the power method with 30 iterations and the trace estimator with $z=100$ random vectors, showing negligible difference.)

The performance of nuclear maximization on the nuclear \Nystrom{} norm $\En (\I) $ closely follows the spectral lower bound (\Cref{fig:cssp}, second row).
On the other hand, the spectrum decays relatively slowly for our example systems, meaning that the spectral lower bound on $\En$ itself is relatively flat for large $k$.

\Cref{fig:cssp} also displays the obliqueness norms $\Obs (\I)$ for each system as a function of $k$.
While the obliqueness is not directly targeted by nuclear maximization, empirically we find that (1) $\Os (\I)$ is roughly comparable to the spectral norm error $\Es (\I)$ and (2) $\On (\I)$ is far smaller than $\En (\I)$.
The latter finding indicates that in this case (and perhaps, many others), the worst-case bound given by \Cref{th:probabilistic-obliqueness-factor} is empirically loose.
Such a result is not necessarily surprising, as we believe that the columns chosen by nuclear maximization are likely to be spread out around the kinetic landscape in such a way that their effective obliqueness is low.

\subsection{Approximation quality \label{s:experiment-approximation}}

\begin{figure}
    \centering
    \includegraphics[width=0.9\textwidth]{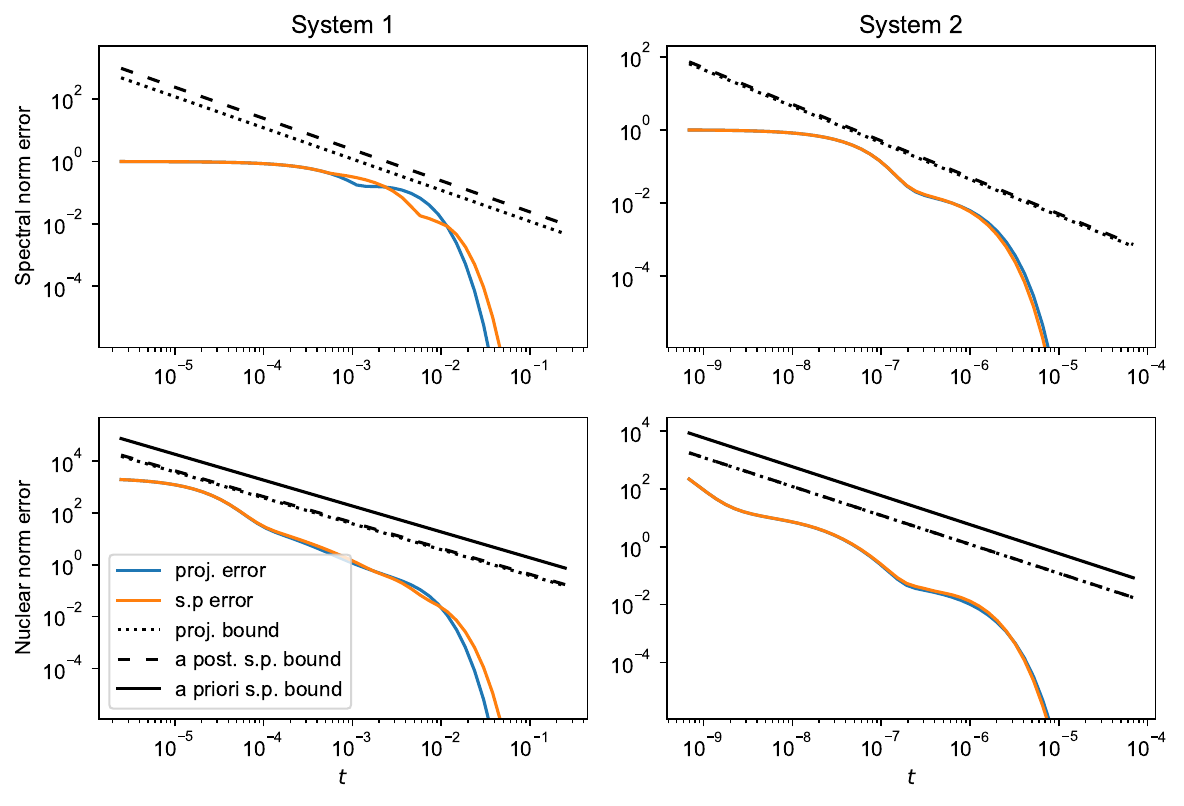}
    \caption{\label{fig:approximation-error}
        (Top) Spectral norm approximation error for projective (proj.) and structure-preserving (s.p.) compression ($\Ia = 5$), as well as bounds based on $\Es$ and $\Os$ (from \Cref{th:ortho-markov} and Theorem 3*).
        (Bottom) Nuclear norm approximation error for projective and structure-preserving compression ($\Ia = 5$), as well as bounds based on $\En$ and $\On$ (from \Cref{th:ortho-markov}, Theorem 3*, and \Cref{th:sp-nuclear-bound}).
    }
\end{figure}

In this section we examine the errors of the projective and structure-preserving compressions
\begin{eqn}
    \Norms{\OrthP_{\I}(t) - \SymP(t)}, \quad  
    \Norms{\ObliqueP_{\I}(t) - \SymP(t)}& 
\end{eqn}
induced by the selection of columns $\I$. 
We consider $\Ia = 5$ in all cases and use the subsets $\I$ as computed by nuclear maximization in \Cref{s:experiment-cssp}.
In 
\Cref{fig:approximation-error}, we consider Systems 1 and 2 as above and plot the errors as a function of $t$, comparing against the bounds furnished by \Cref{th:ortho-markov}, \Cref{th:sp-nuclear-bound}, and Theorem \hyperlink{Theorem3*}{3*}. 
In this section, we only test systems of moderate size (up to $n \approx 10^4$) in order to estimate the exact approximation error relative to the exact dynamics $P(t)$, which we compute by brute-force diagonalization. 
In other words, the computational bottleneck of these experiments is computing the exact reference solution $P(t)$. 
\Cref{fig:approximation-error} confirms that the bounds of \Cref{th:ortho-markov}, \Cref{th:sp-nuclear-bound}, and Theorem \hyperlink{Theorem3*}{3*} hold reasonably tightly, and moreover we observe that the structure-preserving compression is comparable in accuracy to the projective compression in approximating the full dynamics.



\begin{figure}
    \centering
    \includegraphics[width=0.95\textwidth]{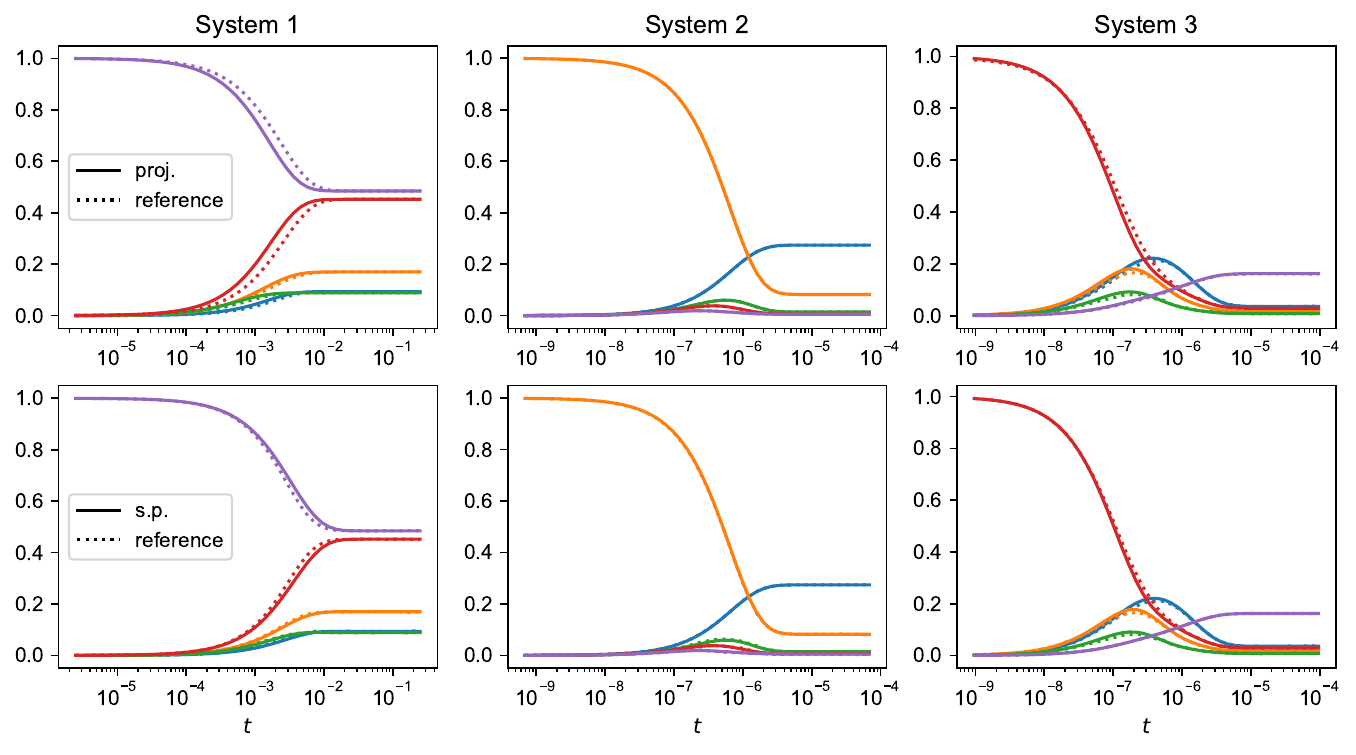}
    \caption{\label{fig:reduced}
        (Top) Comparison of the reduced dynamics using projective compression: compressed dynamics $\expt{V^\t L V}$ vs. the reference dynamics $V^\t \expt{L} V$.
        (An arbitrary row of each matrix is plotted.)
        (Bottom) Comparison of the reduced dynamics using structure-preserving compression: compressed dynamics $\expt{\So^\t \Le \So}$ vs. the reference dynamics $\So^\t \expt{\Le} \So$ (for the same rows as on top).
    }
\end{figure}

\subsection{Approximation in reduced subspace}

Finally we offer an illustration of the reduced models that are furnished by our Markov chain compression. 
While our main theorems (i.e., \Cref{th:ortho-markov,th:sp-nuclear-bound}) guarantee that such reduced models can be lifted to dynamics on the full space, matching the original dynamics with error controlled by the \Nystrom{} approximation, the reduced dynamics themselves may be of interest for inspection and downstream analysis.

In the case of projective compression, the reduced model dynamics are given by $e^{-V^\t L V t}$ where $V = \mathrm{orth} (C)$. In the top row of \Cref{fig:reduced}, we plot an arbitrary row of this matrix as a function of time for Systems~1, 2, and 3 and compare against the appropriate orthogonal projection  $V^\t e^{-Lt} V$ of the full dynamics as a reference.
For Systems~1 and 2, the reference solution is computed via diagonalization.
For System~3, the reference solution is computed as an expectation over $10^6$ trajectories of $\FullChain$ initialized from the stationary distribution $\pi$.
We again consider $\Ia = 5$ in all cases and use the subsets $\I$ as computed by nuclear maximization in \Cref{s:experiment-cssp}.

In the case of structure-preserving compression, the reduced dynamics are given by $e^{-C^\t L C t } = e^{-\Lr t}$. 
In the bottom row of \Cref{fig:reduced}, we plot the same row of this matrix as a function of time for Systems~1,2, and 3 and compare against the appropriate orthogonal projection $\So^\t e^{-\Le t} \So$ of the marked chain dynamics as a reference.
For all systems, the reference solution is computed as an expectation over $10^6$ trajectories of $\MarkedChain$ initialized from the stationary distribution $\pe$.

We comment that in both the projective and structure-preserving cases, the approximations in the reduced subspace are exact at $t=0$ and in the $t \ra \infty$ limit, by construction.

\paragraph{Acknowledgments:}
This work was supported in part by the U.S. Department of Energy, Office of Science, Office of Advanced Scientific Computing Research's Applied Mathematics Competitive Portfolios program under Contract No. AC02-05CH11231. M.L. was also partially
supported by a Sloan Research Fellowship.

\paragraph{Software and computational details:}
Algorithms were implemented and benchmarked in C++17 using 
    {armadillo} \cite{Sanderson2019-practical,Sanderson2016-armadillo} for common linear algebra routines,
    {rchol} \cite{Chen2020-rchol} for approximate Cholesky factorization of graph Laplacians,
    NUPACK 4.0.1 \cite{Fornace2022-nupack} for DNA secondary structure investigations, 
    and the {nuclear-score-maximization} software \cite{Fornace2024-nuclear-score-maximization} from our prior work \cite{Fornace2024-column}.
All computational studies were run on an Apple M3 Max Macbook Pro (36 GB RAM, 10 performance cores, 4 efficiency cores).

\newrefcontext[sorting=anyt]
\printbibliography[heading=bibintoc, title={Bibliography}]

\appendix

\boolfalse{STATING}

\section{Proofs for \Cref{s:compression}}


\subsection{Proof of Theorem \ref{th:dissipative} \label{si:thm}}

In this section we prove Theorem \ref{th:dissipative}.




\ThDissipative*
\begin{Proof}{th:dissipative}
Given \Cref{lem:contour0}, the identities (\ref{eq:int1}) and (\ref{eq:int2}) from the proof sketch rigorously establish that 
\begin{eqn}
    \SymP(t)-\OrthP_{\I}(t) = \frac{1}{2\pi}\oint_{\Contour} e^{-1/z} \, [\FullRes(z)-\OrthRes(z) ] \, {\odif z} \label{eq:intboundrecap}
\end{eqn}
for any contour $\mathcal{C}$ enclosing the eigenvalues of $K/t$.

Now we turn to the choice of contour. For $\delta \geq 0$,
define the triangular contour $\mathcal{C}=\mathcal{C}(T,\delta)$
by 
\begin{eqn}\label{eq:contourdef0}
\mathcal{C}(T,\delta)=\underbrace{\{\delta+\alpha t+i\beta t\,:\,t\in[0,T]\}\cup\{\delta+\alpha t-i\beta t\,:\,t\in[0,T]\}}_{=:\,\mathcal{C}_{1}(T,\delta)}\cup\underbrace{\{\delta+T+is\,:\,s\in[-T,T]\}}_{=:\,\mathcal{C}_{2}(T,\delta)}.
\end{eqn}
Assume that $T> \Vert K\Vert_2 /t$ and $\delta < 1/(\Vert L \Vert_2 t )$ so that the contour encloses the spectrum
of $K/t$ as required. Ultimately we will take $T\ra\infty$ to get
an infinite V-shaped contour, and we will argue that the vertical
crossbar segment $\mathcal{C}_{2}$ yields only a vanishing contribution
in this limit. Here $\alpha,\beta>0$ and we assume $\alpha^{2}+\beta^{2}=1$
so that $t\mapsto\delta+\alpha t\pm i\beta t$ is an arc-length parametrization. Refer to \Cref{fig:contour} for an illustration of the ensuing developments.

It is useful to note that if we write $z=a+ib$ for $a,b\in\R$, then
\begin{equation}
\vert e^{-1/z}\vert=e^{-\frac{a}{a^{2}+b^{2}}}.\label{eq:expidentity}
\end{equation}
 In particular, we have $\vert e^{-1/z}\vert\leq1$ on all of our
contours.

Now we will make the following moves in order: 
\begin{enumerate}
\item Take the limit $\delta\ra0^{+}$ of our contour integral, yielding
an exact expression for $\SymP(t)-P_{\I}(t)$ as a contour integral
over $\mathcal{C}(T,0)$. (\emph{A priori}, to apply Cauchy's integral
formula we must avoid touching the origin with our contour due to
the essential singularity of $e^{-1/z}$ at $z=0$, but we will see
that the limiting contour integral is nonetheless well-defined.)
\item Take the limit $T\ra\infty$ and argue that the contribution of the
vertical crossbar segment $\mathcal{C}_{2}$ vanishes.
\end{enumerate}

By \eqref{eq:intboundrecap} for $\mathcal{C} = \mathcal{C}(T,\delta)$, we deduce by changing variables that 
\begin{align*}
{ \SymP(t)-P_{\I}(t) } & = \frac{1}{2\pi i}\oint_{\mathcal{C}(T,0)}  e^{-1/(z+\delta)}  \, [ \mathcal{Q}_t (z+\delta)-\mathcal{R}_{t}(z+\delta) ] \, {\odif z}.
\end{align*}
 Note that the integrand $I_{\delta}(z):= e^{-1/(z+\delta)}  \, [ \mathcal{Q}_t (z+\delta)-\mathcal{R}_{t}(z+\delta)]$
converges pointwise to $I_{0}(z)$ on $\mathcal{C}(T,0)$ as $\delta\ra0^{+}$,
where we interpret $I_{0}(0) = 0$. 
Moreover, $\vert e^{-1/(z+\delta)}  \vert$  and $\vert  (z+\delta)^{-1} e^{-1/(z+\delta)}  \vert$ are uniformly bounded over the contour, independently of $\delta > 0$ sufficiently small, so the dominated convergence
theorem guarantees that we can take the limit as $\delta\ra0^{+}$
inside the integral to establish: 
\begin{eqn}\label{eq:contourT0}
{ \SymP(t)-P_{\I}(t) } =\frac{1}{2\pi i}\oint_{\mathcal{C}(T,0)}  e^{-1/z} [ \mathcal{Q}_t (z)-\mathcal{R}_{t}(z)] \, \odif z.
\end{eqn}
This completes step (1).

Now we recapitulate \Cref{eq:resdiffbound0}, which was established rigorously in the proof sketch: 
\begin{eqn}\label{eq:resdiffbound0recap}
    \Norms{\FullRes(z)-\OrthRes(z)} \leq \frac{\Eps}{t  \, \vert\im{(z)}\vert^2 }.
\end{eqn}
By plugging into \Cref{eq:contourT0} after separating the $\mathcal{C}_{1}$ and
$\mathcal{C}_{2}$ components, we get the bound 
\begin{equation}
\Vert \SymP(t)-P_{\I}(t)\Vert_{\{2,*\}} \leq \frac{\ve_{\{2,*\}} }{2 \pi t}\,\oint_{\mathcal{C}_{1}(T,0)}\frac{\vert e^{-1/z}\vert}{\vert\im (z) \vert^{2}}\, \Abs{\odif z} +\frac{1}{2\pi}\oint_{\mathcal{C}_{2}(T,0)}\Vert \mathcal{Q}_t (z)-\mathcal{R}_{t}(z) \Vert_{\{2,*\}} \, \Abs{\odif z}.\label{eq:contourbound}
\end{equation}
 We can split the vertical crossbar contour $\mathcal{C}_{2}(T,0)$
into two pieces, 
\begin{eqn}\label{eq:verticalcontourAB}
\mathcal{C}_{2A}(T):=\{T+is\,:\,\vert s\vert\leq\sqrt{T}\},\quad\mathcal{C}_{2B}(T):=\{T+is\,:\,\vert s\vert\geq\sqrt{T}\},
\end{eqn}
 one near and far from the real axis. Since $\Vert \mathcal{Q}_t (z)\Vert,\Vert \mathcal{R}_{t}(z)\Vert=O(T^{-1})$
on $\mathcal{C}_{2}(T,0)$ and the $\mathcal{C}_{2A}(T)$ contour
is of length $O(\sqrt{T})$, the $\mathcal{C}_{2A}$ contribution
vanishes as $T\ra\infty$.

Meanwhile, we can bound the $\mathcal{C}_{2B}$ contribution using
our less trivial bound \Cref{eq:resdiffbound0recap}: 
\begin{align*}
\frac{1}{2\pi}\oint_{\mathcal{C}_{2B}(T)}\Vert \mathcal{Q}_t (z)-\mathcal{R}_{t}(z) \Vert_{\{2,*\}} \, \Abs{\odif z} & \leq \frac{\ve_{\{2,*\}}}{2\pi t}\int_{\mathcal{C}_{2B}(T)}\frac{1}{\vert\mathrm{Im}(z)\vert^{2}}\,\Abs{\odif z}.
\end{align*}
 Then compute 
\begin{eqn}\label{eq:C2Bint}
\int_{\mathcal{C}_{2B}(T)}\frac{1}{\vert\mathrm{Im}(z)\vert^{2}}\, \Abs{\odif z} =2\int_{\sqrt{T}}^{T}\frac{1}{s^{2}}\,ds=2\frac{\sqrt{T}-1}{T}\ra0
\end{eqn}
 as $T\ra\infty$, so the $\mathcal{C}_{2B}$ contribution vanishes
as well.

Therefore by taking the limit as $T\ra\infty$ of our bound (\ref{eq:contourbound}),
we deduce that 
\begin{equation}
\Vert \SymP(t)-P_{\I}(t)\Vert_{\{2,*\}} \leq \frac{\ve_{\{2,*\}} }{2 \pi t}\ \lim_{T\ra\infty}\oint_{\mathcal{C}_{1}(T,0)}\frac{\vert e^{-1/z}\vert}{\vert\im(z)\vert^{2}}\, \Abs{\odif z}.\label{eq:contourbound2}
\end{equation}
Note that for $z=\alpha t\pm i\beta t\in\mathcal{C}_{1}(T,0)$ where
$\alpha^{2}+\beta^{2}=1$, following (\ref{eq:expidentity}) we can
write 
\[
\frac{\vert e^{-1/z}\vert}{\vert\im(z)\vert^{2}}=\frac{e^{-\frac{\alpha}{t}}}{\beta^{2}\,t^{2}},
\]
 hence 
\[
\lim_{T\ra\infty}\oint_{\mathcal{C}_{1}(T,0)}\frac{\vert e^{-1/z}\vert}{\vert\im(z)\vert^{2}}\, \Abs{\odif z}=\frac{2}{\beta^{2}}\int_{0}^{\infty}\frac{e^{-\frac{\alpha}{t}}}{t^{2}}\, \Abs{\odif t} =\frac{2}{\alpha\beta^{2}}.
\]
 We want to optimize the last expression subject to $\alpha^{2}+\beta^{2}=1$.
The optimal choice is given by $\alpha=\frac{1}{\sqrt{3}}$, which
yields $\frac{2}{\alpha\beta^{2}}=3\sqrt{3}$.

Putting everything together in (\ref{eq:contourbound2}), we have
\[
\Vert \SymP(t)-P_{\I}(t)\Vert_{\{2,*\}} \leq\frac{3\sqrt{3}}{2\pi}\,\frac{\ve_{\{2,*\}} }{t}.
\]
This completes step (2) and concludes the proof.
\end{Proof}

\subsection{Proofs of lemmas for Theorem \ref{th:dissipative} \label{si:lemmas}}

First we prove \Cref{lem:contour0}:

\LemContour*
\begin{Proof}{lem:contour0}
Let $\tilde{U}$ complete the columns of $U$ to an orthonormal basis of $\R^{n}$.
By considering blocks with respect to the $[U,\tilde{U}]$
basis, we can rewrite: 
\begin{eqn} \label{eq:altersolve}
(z\Id-UA^{-1}U^{\top})^{-1} & \ \leftrightarrow \ \left(z\Id-\left[\begin{array}{cc}
A^{-1} & 0\\
0 & 0
\end{array}\right]\right)^{-1}\\
 & \ = \ \left[\begin{array}{cc}
(z\Id-A^{-1})^{-1} & 0\\
0 & z^{-1}\Id
\end{array}\right]\leftrightarrow U(z\Id-A^{-1})^{-1}U^{\top}+z^{-1} \tilde{U} \tilde{U}^{\top},
\end{eqn}
where we use ``$\leftrightarrow$'' to indicate the suitable changes of basis. 

By Cauchy's
integral theorem for matrices \cite{Higham_2008}, we have that  
\begin{eqn}\label{eq:cauchy1}
e^{-A}=\oint_{\mathcal{C}}e^{-1/z}\left(z\Id-A^{-1}\right)^{-1}\,\odif z.
\end{eqn}
 We also have that 
\begin{eqn}\label{eq:cauchy2}
\oint_{\mathcal{C}}e^{-1/z} \, z^{-1} \, \tilde{U} \tilde{U}^{\top}\,\odif z=0,
\end{eqn}
since $e^{-1/z}z^{-1}$ is holomorphic on the interior of $\mathcal{C}$.

Then by taking the contour integral of both sides of \Cref{eq:altersolve} against $e^{-1/z} \, \odif z$ and using \Cref{eq:cauchy1,eq:cauchy2}, we conclude the proof.
\end{Proof}

Then we prove Lemma \ref{lem:refactor}:

\LemRefactor*
\begin{Proof}{lem:refactor}
We will prove a more general result.
Let $X$ be a tall rectangular matrix of full rank, and let $V=KX(X^{\top}K^{2}X)^{-1/2}$.
Expanding,
\[
V^{\top}K^{-1}V=(X^{\top}K^{2}X)^{-1/2}(X^{\top}KX)(X^{\top}K^{2}X)^{-1/2},
\]
 hence 
\[
(V^{\top}K^{-1}V)^{-1}=(X^{\top}K^{2}X)^{1/2}(X^{\top}KX)^{-1}(X^{\top}K^{2}X)^{1/2},
\]
 and in turn 
\[
V(V^{\top}K^{-1}V)^{-1}V^{\top}=KX(X^{\top}KX)^{-1}X^{\top}K.
\]
The desired statement in \cref{eq:refactor} follows from the choice of $X \Eq \Id_{:,\I}$.
\end{Proof}

Then to prove Lemma \ref{lem:polebound}, we first prove an even simpler
lemma.
\begin{lemma}
Let $B$ be a real symmetric matrix. Then $\Vert(\Id-iB)^{-1}\Vert_2 \leq1$. 
\end{lemma}
\begin{proof}
$\Id-iB$ is not Hermitian, but it is a normal operator (i.e.,
it commutes with its adjoint $\Id+iB$). Therefore $\Id-iB$
can be unitarily diagonalized by the spectral theorem, and it follows
that $\Vert(\Id-iB)^{-1}\Vert_2 $ is the spectral radius of $(\Id-iB)^{-1}$.
Since $B$ is real symmetric, the eigenvalues of $(\Id-iB)^{-1}$
all have the form $(1-i\lambda)^{-1}$, where $\lambda\in\R$ is an
eigenvalue of $B$. For any such $\lambda$, we have $\vert(1-i\lambda)^{-1}\vert=\vert1-i\lambda\vert^{-1}\leq1$.
\end{proof}

Then Lemma \ref{lem:polebound} is restated and proved as follows.

\LemPolebound*
\begin{Proof}{lem:polebound}
Write $z=a+ib$ and unpack: 
\begin{align*}
(z\Id-A)^{-1} & =(ib\Id+a\Id-A)^{-1}\\
 & =\frac{1}{ib}(\Id-ib^{-1}[a\Id-A])^{-1},
\end{align*}
 so the desired conclusion follows from the preceding lemma.
\end{Proof}

\subsection{Proof of Theorem 1B}

\ThNonNystromSing*
\begin{Proof}{th:non-nystrom-sing} 
    Let $\Pi$ be the orthogonal projector onto the null space of $L$, and define $L^{(\gamma)} := L + \gamma \Pi$ for $\gamma > 0$. Then $L^{(\gamma)}$ is positive definite, and we can further define $K^{(\gamma)} := [L^{(\gamma)}]^{-1} = K + \gamma^{-1} \Pi$. Moreover, define ${P}^{(\gamma)} (t) := e^{-L^{(\gamma)} t}$ and 
    $\tilde{P}^{(\gamma)}(t) := V e^{- (V^\t L^{(\gamma)} V) t} V^\t$. Then 
    Theorem \hyperlink{link.Theorem1A}{1A} implies that 
    \begin{eqn}
        \Norms{ \SymP^{(\Kill)}(t) - \tilde{\SymP}^{(\Kill)}(t) } \leq \frac{3 \sqrt{3}}{2 \pi} \Nus^{(\Kill)}
    \end{eqn}
    where $\Nus^{(\Kill)} := \Norms{K^{(\gamma)} - V (V^\t L^{(\gamma)} V)^\Mo V^\t}$. 

    Evidently, for any $t\geq 0$, both ${P}^{(\gamma)} (t) \ra {P} (t)$ and $\tilde{P}^{(\gamma)}(t) \ra \tilde{P}(t)$ as $\gamma \ra 0$. Therefore it suffices to show that 
    \begin{eqn}
        \lim_{\Kill \ra 0} \Nus^{(\Kill)} =  \Nus.
    \end{eqn}

    Now since the columns of $V$ contain the null space of $L$, it follows that  $VV^\t \Pi = \Pi  VV^\t = \Pi$. It also follows that $V^\t \Pi V$ is the orthogonal projector onto the the null space of $V^\t L V$, so in turn  $(V^\t L V)^+ = (V^\t L V + \Kill V^\t \Pi V )^\Mo - \Kill^\Mo V^\t \Pi V$ for any $\Kill > 0$. We can use these facts to compute: 
    \begin{eqn}
        \lim_{\Kill \ra 0} \left[ K^{(\gamma)} - V (V^\t L^{(\gamma)} V)^\Mo V^\t \right] &= \lim_{\Kill \ra 0} \left[   K + \Kill^\Mo \Pi - V ( V^\t [L + \Kill \Pi ] V)^\Mo V^\t \right] \\   
         &= \lim_{\Kill \ra 0} \left[   K - V \left\{  ( V^\t [L + \Kill \Pi ] V)^\Mo - \gamma^{-1} V^\t \Pi V \right\} V^\t \right] \\   
        &= \lim_{\Kill \ra 0} \left[   K - V (V^\t L V)^+ V^\t \right] \\    
        &=  K - V (V^\t L V)^+ V^\t.
    \end{eqn}
    By taking norms we complete the proof.
\end{Proof}


\section{Proofs for \Cref{s:projective}}

\subsection{Closed-form expressions using the fundamental matrix}

In the next proposition, we provide a probabilistic proof of a closed form expression for the committor $\Cu$ induced by states $\I$ in terms of $K$, the pseudoinverse of $L$. An alternative linear-algebraic proof is sketched in the main text after the definition of the killed committor (\Cref{def:killed-committor}).

\begin{proposition}[Committor functions from fundamental matrix] \label{th:committor}
    For a reversible Markov chain associated to Laplacian $L$, fundamental matrix $K \Eq L^+$, and non-empty set $\I \subset \Indices$:
    \begin{eqn}
        \Cu = \DiagI(h) \lrp{\frac{h - K_{:,\I} \Ki \hi}{\hi^\t \Ki \hi} \hi^\t + K_{:,\I}} \Ki \Diag(\hi).
        \label{eq:committor}
    \end{eqn}
    The stationary expectation of the committor is:
    \begin{eqn}
        \pr = \Cu^\t \pi = \frac{\hi^\t \Ki \Diag(\hi)} {\hi^\t \Ki \hi}.
        \label{eq:committor-pi}
    \end{eqn}
\end{proposition}
\begin{proof}
\newcommand{\TimeToI}[1]{H^{\I}_{#1}}
Let $S \Eq \DiagI(h) K \DiagI(h)$.
Since $K$ is symmetric and $K h = \Zero$, $S$ is symmetric and $S \pi = \Zero$ (recall $\pi = h^{2}$).
Let $\TimeToI{j} = \Ex[\tau_{\I} \mid X_0 = j]$ be the average time for the Markov chain starting in state $j$ to first reach any of $\I$ (so that $\TimeToI{j} = 0$ for $j \in \I$), $H_{i,j} = \Ex[\tau_\lrb{j} \mid X_0 = i]$ be the average time for the Markov chain starting in state $i$ to reach state $j$, and $T^{k}_{i,j}$ be the average time spent in state $k$ for the Markov chain in state $i$ run until reaching state $j$. 

For any $j \in \Indices$, we will next consider the path starting at $j$ with stopping time ``the first return to $j$ after hitting any one of $\I$''(a path $j \ra \I \ra j$) in order to establish a linear system yielding $\Cu$.
To proceed, for any $i \in \I$,
application of Proposition~2.3 of \cite{Aldous1995-reversible} yields:
\begin{eqn}
    \label{eq:committor-measure}
    \p_i \pp{\TimeToI{j} + \sum_{k \in \I} \Cu_{j,k} H_{k,j}}
    &= \sum_{k \in \I} \Cu_{j, k} T^i_{k, j}
\end{eqn}
where the left-hand side is the expected duration of such a path (weighted by $\p_i$), and the right-hand side is the expected occupancy time in $i$ of such a path.\footnote{
We can explain the intuition in more detail. 
Consider an infinitely long trajectory of the Markov chain from a stationary start and break it into intervals based on the cycles as defined ($j \ra \I \ra j$).
The average time of such a cycle is $\TimeToI{j}$ plus the sum over $k \in \I$ of $H_{k,j}$, weighted by the probability that $k$ was the first state in $\I$ encountered (i.e., the committor).
Meanwhile, the time spent in $i$ during such a cycle is just the sum over $k \in \I$ of $T^i_{k, j}$ weighted by the committor, because no time is spent in $i$ before reaching $\I$, by definition.
Now, the fraction of time spent in $i$ for an infinitely long trajectory must just be $\pi_i$, so multiplying the average cycle time by $\pi_i$ must be the same as the average time spent in $i$ during such a cycle, giving \cref{eq:committor-measure}.
}
\cref{eq:committor-measure} gives us one equation involving the unknown probabilities $\Cu_{j,:}$ for each $i \in \I$ with $\TimeToI{j}$ as an additional unknown.

Now, substitute $H_{i,j} = S_{j,j} - S_{i,j}$ (\cite{Aldous1995-reversible}, Lemma~2.12) and $T^i_{k, j} = \pi_i (S_{j,j} - S_{k,j} - S_{j,i} + S_{k,i})$ (\cite{Aldous1995-reversible}, Lemma~2.9) to find that:
\begin{eqn}
    \label{eq:committor-column-i}
    \TimeToI{j} + \sum_{k \in \I} \Cu_{j,k} (S_{j,j} - S_{k,j}) &=
    \sum_{k \in \I} \Cu_{j,k} (S_{j,j} - S_{k,j} - S_{j,i} + S_{k,i}). \\
\end{eqn}
By collecting equations \cref{eq:committor-column-i} for each $i \in \I$ and simplifying, we find:
\begin{eqn}
    (\TimeToI{:}) \One^\t &= \Cu S_{\I, \I} - S_{:,\I}.
  \label{eq:committor-system}
\end{eqn}
For each $j \in \Indices$, this provides $\Ia$ equations for the $\Ia+1$ unknowns $\Cu_{j,:}$ and $\TimeToI{j}$.
Another constraint is gained from normalization of probability,
%
  $\Cu \One = \One$,
%
such that solving the system of equations gives:
%
\begin{eqn}
    \Cu = \pp{\frac{\One -S_{:, \I} S_{\I, \I}^{-1} \One}{\One^\t S_{\I, \I}^{-1} \One} \One^\t + S_{:, \I}} S_{\I,\I}^{-1}, \qquad
    \TimeToI{:} = \frac{\One - S_{:, \I} S_{\I,\I}^{-1} \One}{\One^\t S_{\I,\I}^{-1} \One},
    \label{eq:committor-column}
\end{eqn}
which implies \cref{eq:committor} by substitution via the definition of $S$.
Finally, \cref{eq:committor-pi} follows from computation of $\Cu^\t \pi$ with simplifications since $S \pi = \Zero$.
\end{proof}

\begin{lemma}[Limiting norms from fundamental matrix] \label{th:markov-norm-limits}
    For a nonempty set $\I$, Laplacian $L$, and $K = L^+$: 
    \begin{eqn}
        \label{eq:markov-norm-k}
    \Es(\I) &= \snorm*{K - K_{:,\I} \Ki K_{\I,:} + \frac{(h - K_{:,\I} \Ki \hi) (h^\t - \hi^\t \Ki K_{\I,:})}{\hi^\t \Ki \hi} } \\
    \En(\I) &=\Tr\left[K - K_{:,\I} \Ki K_{\I,:} \right] + \frac{1+\hi^\t \Ki (K^{2})_{\I,\I} \Ki \hi}{\hi^\t \Ki \hi}
    \end{eqn}
\end{lemma}
\begin{proof}
    Recall that $\La \Eq L + \Kill \Id$ and $\Ka \Eq \La^\Mo$ and moreover that $K := L^+$. Additionally it is useful to define a matrix $\tilde{K}_\gamma := K_\gamma - \gamma^{-1} hh^\t$. This definition subtracts off the divergent rank-one component of $K_\gamma$ in the limit $\gamma \ra 0$, so that $\tilde{K}_\gamma$ shares the same null space as $K$ and satisfies $\lim_{\gamma \ra 0} \tilde{K}_\gamma = K$. For ease of notation, in the following computations we will suppress the $\gamma$-dependence as $\tilde{K} = \tilde{K}_\gamma$.
    
    

    Then we can view $K_\gamma = \tilde{K} + \gamma^{-1} hh^\t$ as a rank-one update and compute:
    \begin{eqn}
        \Ka - (\Ka)_{:,\I} (\Ka)_{\I,\I}^\Mo (\Ka)_{\I,:} &= \Ka - (\Ka)_{:,\I} \lrp{\Kti - \frac{\Kti \hi \hi^\t \Kti}{\Kill + \hi^\t \Kti \hi}} (\Ka)_{\I,:} \\
        &= \Ka - \lrp{\Kt_{:,\I} \Kti + \frac{(h - \Kt_{:,\I} \Kti \hi) \hi^\t \Kti}{\Kill + \hi^\t \Kti \hi}} (\Ka)_{\I,:} \\
        &= \Kt - \Kt_{:,\I} \Kti \Kt_{\I,:} + \frac{(\Id -  \Kt_{:,\I} \Kti \Id_{:,\I}) h h^\t (\Id - \Id_{:,\I} \Kti \Kt_{\I,:})}{\Kill + \hi^\t \Kti \hi},
    \end{eqn}
    with the first line from the Sherman-Morrison formula and the latter lines by expansion of $\Ka$ and collection of terms.
    Then taking the limit $\Kill \ra 0$ yields:
    \begin{eqn}
        \Eps &= \lim_{\Kill \ra 0} \Norms{\Ka - (\Ka)_{:,\I} (\Ka)_{\I,\I}^\Mo (\Ka)_{\I,:}} \\
        &= \Norms*{K - K_{:,\I} \Ki K_{\I,:} + \frac{(\Id -  K_{:,\I} \Ki \Id_{:,\I}) h h^\t (\Id - \Id_{:,\I} \Ki K_{\I,:})}{\hi^\t \Ki \hi}}
    \end{eqn}
    with the trace formula in \cref{eq:markov-norm-k} following by rearrangement and use of $h^\t h = 1$.
\end{proof}

\subsection{Analysis of killed chain and limit $\gamma \ra 0$}

In this section we detail how analysis of the killed chain leads to \Cref{th:projective-compression}. First, for motivation and future reference, we prove the probabilistic interpretation of the killed committor suggested in the main text. 
\newcommand{\Xa}{X^{\Kill}}
\begin{lemma}[Probabilistic interpretation of killed committor] \label{th:killed-committor-prob}
For $\Kill > 0$, consider the (irreversible) Markov chain on the space $[n]\cup \{x\}$, augmented by a single cemetery state $x$, with rate matrix given by
\begin{eqn}
    R^{(\Kill)} \Eq \begin{bmatrix}
        R - \Kill \Id & \Kill \One \\
        \Zero & 0
    \end{bmatrix}.
\end{eqn}
Let $\Process{\Xa}$ be the associated stochastic process with initialization $\Xa_0 \sim [\pi^\t, 0]^\t$.
Then for any $i \in \Indices$, $j \in \I$, and with $\Cu_{\Kill}$ as in \Cref{th:killed-committor}:
\begin{eqn} \label{eq:killed-committor-prob}
    (\Cu_{\Kill})_{i,j} = \Prob(\Xa_{\tau_{\I \cup \lrb{x}}} = j \mid \Xa_0 = i).
\end{eqn}
\end{lemma}
\begin{proof}
The probability evolution outside of $\I \cup \{ x \}$ is given by $\exp((R_{\Ic,\Ic} - \gamma \Id) t)$, while the rates of entering the selected states from unselected states are $R_{\Ic,\I}$.
Therefore the committor probability from any state $i \notin \I \cup \lrb{x}$ to any state $j \in \I$ is:
\begin{eqn}
    \Prob(\Xa_{\tau_{\I \cup \lrb{x}}} = j \mid \Xa_0 = i) &= \int_0^\infty  \Id_{i,\Ic} \exp((R_{\Ic, \Ic} - \Kill \Id)t) R_{\Ic, j} \odif t \\
    &= \Minus \Id_{i, \Ic} (R_{\Ic, \Ic} - \Kill \Id)^\Mo R_{\Ic, j} \\
    &= \Minus \frac{h_j}{h_i} \Id_{i, \Ic} ( (\La)_{\Ic,\Ic} )^\Mo (\La)_{\Ic, j} \\
    &= \frac{h_j}{h_i} (\Ka)_{i, \I} (\Ka)_{\I,\I}^\Mo \Id_{\I,j} \\
\end{eqn}
where the third line follows from the definition of $\La$, and the last line follows from the general identity $A_{\Ic,\Ic}^{-1} A_{\Ic,\I} = -B_{\Ic,\I} B_{\I,\I}^{-1}$ for arbitrary $A = B^{-1}$, which in turn can be seen from the block matrix inversion formula. 
Additionally, we know automatically that $\Prob(\Xa_{\tau_{\I \cup \lrb{x}}} = j \mid \Xa_0 = i) = \delta_{i,j}$ for any $i \in \I$. Therefore $\Prob(\Xa_{\tau_{\I \cup \lrb{x}}} = j \mid \Xa_0 = i)$ matches the $(i,j)$ entry of $\Cu_\Kill \Eq \Diag(1/h) (\Ka)_{:,\I} (\Ka)_{\I,\I}^\Mo \Diag(\hi)$ (\Cref{def:killed-committor}) for all $i \in [n]$ and $j \in \I$, as was to be shown.
\end{proof}


Next we prove that our killed quantities recover the original quantities in the limit $\gamma \ra 0$.




\ThKilledCommittor*
\begin{Proof}{th:killed-committor}
    Recall that $\La \Eq L + \Kill \Id$ and $\Ka \Eq \La^\Mo$ and moreover that $K := L^+$. As in the proof of \Cref{th:markov-norm-limits}, it is useful to define a matrix $\tilde{K}_\gamma := K_\gamma - \gamma^{-1} hh^\t$, suppress the $\gamma$-dependence from the notation $\tilde{K} = \tilde{K}_\gamma$, and view $K_\gamma = \tilde{K} + \gamma^{-1} hh^\t$ as a rank-one update. Now by \Cref{def:killed-committor} for $\Cu_{\Kill}$: 
\begin{eqn}
    \lim_{\Kill \ra 0} \Cu_{\Kill} &= \lim_{\Kill \ra 0} \DiagI(h) (\tilde{K} + \Kill^\Mo h h^\t)_{:,\I} (\tilde{K} + \Kill^\Mo h h^\t)_{\I,\I}^\Mo \Diag(\hi).
\end{eqn}
Then compute by the Sherman-Morrison formula: 
\begin{eqn}
    \label{eq:committor-killed}
    (\Kt + \Kill^\Mo h h^\t)_{:,\I} (\Kt + \Kill^\Mo h h^\t)_{\I,\I}^\Mo
    &= (\Kt + \Kill^\Mo h h^\t)_{:,\I} \lrp{\Kti - \frac{\Kti \hi \hi^\t \Kti}{\Kill + \hi^\t \Kti \hi}} \\
    &= \Kt_{:,\I} \Kti + \frac{(h - \Kt_{:,\I} \Kti \hi) \hi^\t \Kti}{\Kill + \hi^\t \Kti \hi},
\end{eqn}
with the second line following by expansion and collection of outer products.
Then taking the limit $\Kill \ra 0$ yields $\lim_{\Kill \ra 0} \Cu_{\Kill} = \Cu$ by comparison to \Cref{th:committor}.
    Then $\lim_{\Kill \ra 0} \pr_\Kill = \pr$, $\lim_{\Kill \ra 0} \hr_\Kill = \hr$, and $\lim_{\Kill \ra 0} \Co_\Kill = \Co$ from \Cref{def:killed-committor}, trivially.
\end{Proof}


\ThProjectiveCompression*
\begin{Proof}{th:projective-compression}
    Recall from \cref{eq:killed-projective-compression} that:
    \begin{eqn} \label{eq:PIlim}
        \OrthP_{\I}(t) &\Eq \lim_{\Kill \ra 0} \Va \Expt{\Va^\t \La \Va} \Va^\t \text{ where } \Va \Eq \Orth(\KA_{:,\I})
    \end{eqn}
    From \cref{eq:killed-chp} we have that, for any $\Kill > 0$:
    \begin{eqn}
        \Ca = \KA_{:,\I} \KA_{\I,\I}^\Mo \Diag(\hi / \ha)
    \end{eqn}
    Therefore the range of $\KA_{:,\I}$ and $\Ca$ are the same, and in the limit appearing \cref{eq:PIlim} we can assume that we have taken $V_\gamma = \Orth(C_\gamma)$, where for concreteness $\Orth$ indicates the symmetric orthogonalization as usual. Since $\lim_{\Kill \ra 0} \Ca = \Co$, it follows that $\lim_{\Kill \ra 0} V_\gamma = V = \Orth(C)$ since the symmetric orthogonalization is a continuous operation. Moreover, $\lim_{\Kill \ra 0} \La = L$, so we can directly take the limit in \cref{eq:PIlim} to deduce that $\OrthP_{\I}(t) = V \Expt{V^\t L V} V^\t$. The additional expression in \eqref{eq:orth-prop} is justified by standard linear-algebraic manipulations.
\end{Proof}

\ThOrthoMarkov*
\begin{Proof}{th:ortho-markov}
    Consider $t$ fixed and $\Kill > 0$. 
    Let $\SymP_\Kill(t) = \Expt{\La}$ and $\OrthP_{\I,\Kill}(t)$ be as in \cref{eq:killed-projective-compression}.
    First, since $\SymP_\Kill(t) = \expt{\Kill} \SymP(t)$ and $\OrthP_{\I,\Kill}(t) = \expt{\Kill} \OrthP_{\I}(t)$,
    \begin{eqn}
        \label{eq:ortho-markov-1}
        \Norms{\SymP(t) - \SymP_{\Kill}(t)} = (1 - \expt{\Kill}) \Norms{\SymP(t)} \\
        \Norms{\OrthP_{\I}(t) - \OrthP_{\I,\Kill}(t)} = (1 - \expt{\Kill}) \Norms{\OrthP_{\I}(t)}. 
    \end{eqn}
    Second, by \Cref{th:dissipative} (covering symmetric positive definite matrices), we have
    \begin{eqn}
        \label{eq:ortho-markov-2}
        \Norms{\SymP_\Kill(t) - \OrthP_{\I,\Kill}(t)} \leq \frac{3\sqrt{3}}{2\pi}\,\frac{\Eps^\Kill}{t}.
    \end{eqn}
    So by the triangle inequality and use of \cref{eq:ortho-markov-1,eq:ortho-markov-2},
    \begin{eqn}
        \Norms{\SymP(t) - \OrthP_{\I}(t)} &\leq \Norms{\SymP_\Kill(t) - \OrthP_{\I,\Kill}(t)} + \Norms{\SymP(t) - \SymP_{\Kill}(t)} + \Norms{\OrthP_{\I}(t) - \OrthP_{\I,\Kill}(t)} \\
        &\leq \frac{3 \sqrt{3}}{2 \pi} \frac{\Eps^\Kill}{t} + (1 - \expt{\Kill}) (\Norms{\SymP(t)} + \Norms{\OrthP_{\I}(t)}).
    \end{eqn}
    In the $\Kill \ra 0$ limit, $\Eps^\Kill \ra \Eps$ and the latter term vanishes, yielding \cref{eq:ortho-markov-norm-bound}.
\end{Proof}

\subsection{Relating symmetric and asymmetric approximations}

\ThAutocorrelation*
\begin{Proof}{th:autocorrelation}
Given the definition of autocorrelation \cref{eq:acorr,eq:acorr2}, since the initial state $X_0$ is sampled out of $\pi$, we compute: 
\begin{eqn}
    A_t(f, \FullP) &= f^\t \Diag(\pi) \FullP(t) f = f^\t \Diag(h) \SymP(t) \Diag(h) f, \\
    A_t(F, \FullP) &= F^\t \Diag(\pi) \FullP(t) F = F^\t \Diag(h) \SymP(t) \Diag(h) F.
\end{eqn}
Analogous relations hold for $\FullQ$, $\SymQ$.

From standard properties of the spectral norm:
\begin{eqn}
    \snorm{\SymP(t) - \SymQ(t)} &= \sup_{g \in \R^n} \lrb{ \Abs{g^\t (\SymP(t) - \SymQ(t)) g} \;:\; \norm{g}_2 = 1}
\end{eqn}
and \cref{eq:acorr} follows from defining $f \Eq \DiagI(h) g$.

Next we consider the nuclear norm bound. Note that for $m \leq n $ and a general  matrix $G \in \R^{n \times m}$ satisfying $G^\t G = \Id$, we have (by considering the SVD) that $\Vert G \Vert_2 = 1$, and therefore $\Vert G^\t A G \Vert_* \leq \Vert A \Vert_*$ for any $A \in \R^{n\times n}$ by H\"{o}lder's inequality for the Schatten $p$-norms. Moreover, equality is attained by simply taking $G = \Id$. Therefore, by substituting $A = P(t) - Q(t)$, we deduce that 
\begin{eqn}
    \nnorm{\SymP(t) - \SymQ(t)} &= \sup_{m \leq n, \, G \in \R^{n\times m} } \lrb{ \nnorm{G^\t (\SymP(t) - \SymQ(t)) G} \;:\; G^\t G = \Id}.
\end{eqn}
and \cref{eq:acorr2} follows from defining $F \Eq \DiagI(h) G$.
\end{Proof}

\section{Proofs for \Cref{s:structure-preserving} \label{si:structure-preserving}}

\subsection{Induced chain formulation}

\newcommand{\DiagRat}[2]{\Diag\lrp{\tfrac{#1}{#2}}}

\begin{lemma}[Marking time from a stationary start] \label{th:timescale}
    The time to get to $\I$ from a stationary start satisfies:
    \begin{eqn}
        \Ex[\tau_{\I}] = (\hi^\t \Ki \hi)^\Mo.
    \end{eqn}    
\end{lemma}
\begin{proof}
    Trajectories starting in $\I$ contribute 0 to this expectation, whereas the probability distribution of the chain started in $\Ic$ and run until reaching $\I$ is governed by $R_{\Ic,\Ic}$.
    Therefore we compute:
    \begin{eqn}
        \Ex[\tau_{\I}] = \int_0^\infty \pi_{\Ic}^\t \ert{R_{\Ic,\Ic}} \One_{\Ic} \odif t 
        &= -\pi^\t_{\Ic} R_{\Ic,\Ic}^\Mo \One_{\Ic}
        = h_{\Ic}^\t L_{\Ic,\Ic}^\Mo h_{\Ic}.
    \end{eqn}
    Then we derive that $h_{\Ic}^\t L_{\Ic,\Ic}^\Mo h_{\Ic} = (\hi^\t \Ki \hi)^\Mo$ by Schur complement computations.
\end{proof}
\noindent We henceforth define for convenience:
\begin{eqn} \label{eq:omega}
    \TimeScale \Eq \Ex[\tau_{\I}] = (\hi^\t \Ki \hi)^\Mo.
\end{eqn}

\begin{proposition}[Induced chain in terms of fundamental matrix] \label{th:induced-from-k}
    Given the definitions in \Cref{eq:reduced-h-L-K} and \Cref{eq:simple-induced}, in the unsymmetrized rate matrix formulation, we derive that:
    \begin{eqn} \label{eq:induced-R-from-K}
        \pr &= \TimeScale \Diag(\hi) \Ki \hi, \\
        \Rr &= \TimeScale^\Mo \One \pr^\t - \DiagRat{\hi}{\pr} \Ki \Diag(\hi).
    \end{eqn}
    In the symmetrized formulation, letting $\hr = \pr^\Oh$:
    \begin{eqn}
        \Lr &= \DiagRat{\hi}{\hr} \Ki \DiagRat{\hi}{\hr} - \TimeScale^\Mo \hr \hr^\t, \\
        \Kr &= \DiagRat{\hr}{\hi} K_{\I,\I} \DiagRat{\hr}{\hi} - \TimeScale \hr \hr^\t.
    \end{eqn}
    $\Lr$ and $\Kr$ are symmetric matrices with $\Kr = \Lr^+$.
\end{proposition}
\begin{proof}
    We first verify the formula for $\pr$:
    \begin{eqn}
        \pr &= \Cu^\t \pi = \Diag(\hi) \Ki \lrs{\frac{\hi (h^\t - \hi^\t \Ki K_{\I,:}) }{\hi^\t \Ki \hi} + K_{\I,:}} h
        = \frac{\Diag(\hi) \Ki  \hi}{\hi^\t \Ki \hi}
    \end{eqn}
    where the last equality follows from $K h = 0$ and $h^\t h = 1$.

    We next verify the formula for $\Lr$ following its definition \cref{eq:L-h-def} and $\TimeScale$ as above:
    \begin{eqn}
        \Lr &= \Co^\t L \Co \\
        &= \DiagRat{\hi}{\hr} \Ki \lrs{\frac{\hi (h^\t - \hi^\t \Ki K_{\I,:}) }{\hi^\t \Ki \hi} + K_{\I,:}} L \lrs{\frac{(h - K_{:,\I} \Ki \hi) \hi^\t}{\hi^\t \Ki \hi} + K_{:,\I}} \Ki \DiagRat{\hi}{\hr} \\
        &= \DiagRat{\hi}{\hr} \Ki \DiagRat{\hi}{\hr} - (\hi^\t \Ki \hi) \hr \hr^\t
    \end{eqn}
    where simplifications are yielded by considering that $L h = 0$ and $K L K = K$.
    The formula for $\Rr$ is straightforwardly shown via the relation $\Rr = \Minus \DiagI(h) L \Diag(h)$.

    Finally, we verify the formula for $\Kr$ by showing:
    \begin{eqn}
        \Lr \Kr &= \lrs{\DiagRat{\hi}{\hr} \Ki \DiagRat{\hi}{\hr} - (\hi^\t \Ki \hi) \hr \hr^\t} \lrs{\DiagRat{\hr}{\hi} K_{\I,\I} \DiagRat{\hr}{\hi} - (\hi^\t \Ki \hi)^\Mo \hr \hr^\t} \\
        &= \Id - \hr \hr^\t,
    \end{eqn}
    implying that $\Kr = \Lr^+$ since $\hr$ is the zero eigenvector of $\Lr$.
\end{proof}



\newcommand{\fopt}{f^{(i)}}
\ThDirichletTrace*
\begin{Proof}{th:dirichlet-trace}
    $\Tr[F^\t \Delta F]$ is a sum of independent terms for each vector $f^{(i)} \Eq F_{:,i}$.
    Therefore, the minimization problem over $F$ decouples over the columns into independent minimization problems 
    \begin{eqn}
    \label{eq:indargmin}
        \fopt = \argmin_{f \in \R^n}  \left\{ f^\t \Delta f \, : \, f_{\I} = \Id_{\I,i} \right\}
    \end{eqn}
    for each $i \in \I$.

    Recalling that $\Delta$ is symmetric positive semidefinite, the Euler-Lagrange equation for \Cref{eq:indargmin} is the Laplace equation:
    \begin{eqn}
      (\Delta \fopt)_{\Ic} &= \Zero, \qquad
         \fopt_{\I} &= \Id_{\I,i}.
      \label{eq:laplace-equation}
    \end{eqn}
    By \Cref{eq:flow-matrices}, it follows that $(R f^{(i)})_{\Ic} = \Zero$, and in turn that 
    \begin{eqn} \label{eq:optional-stopping}
      \fopt_k & = 
      \sum_{j \neq k } \Prob(k \ra j) \,  \fopt_j \text{ for all } k \in \Ic,
    \end{eqn}
    where $$\Prob(k \ra j) = \Minus \frac{R_{k,j}}{R_{k,k}}$$ is the probability that the chain $\FullChain$ in state $k$ will next jump to state $j$.
    
    By the optional stopping theorem (e.g., \cite{Durrett2019-probability}, Chapter 4.8), given \cref{eq:optional-stopping} and the prescribed boundary condition of $\fopt_{\I}$, the solution $\fopt$ is unique and satisfies:
    \begin{eqn}
      \fopt_k & = \E \big[ \fopt_{X_{\tau_{\,\I}}}  \, \big\vert \,  X_0 = k \big] \text{ for all } k \notin \I,
    \end{eqn}
    where $\tau_{\,\I}$ is the stopping time defined as in \cref{eq:stoppingtime}.
    Therefore $\fopt_k = \Cu_{k,i}$ with $\Cu_{k,i}$ the committor function from \Cref{def:sym-operators}.
    Straightforward computation yields that $\Cu^\t \Delta \Cu = \Diag(\hr) \Lr \Diag(\hr) = \hat{\Delta}$.
\end{Proof}


\ThInducedBasics*
\begin{Proof}{th:induced-basics}
    First, reversibility is implied since the matrix
    \begin{eqn} \label{eq:diagpiR}
        \Diag(\pr) \Rr = \Cu^\t \Diag(\pi) R \Cu
    \end{eqn}
    is symmetric and negative semidefinite by inspection. 
    (Recall that the original chain being reversible implies that $\Diag(\pi) R$ is symmetric and negative semidefinite.)
    Therefore $\Rr$ has nonpositive eigenvalues.
    
    We next verify that $\pr$ does in fact define the stationary distribution, recalling that $\pr \Eq \Cu^\t \pi$. 
    By the probabilistic formulation of $\Cu$ \cref{eq:committor-prob-def},
    \begin{eqn}
        \pr_i = \sum_{k \in \Indices} \Prob(X_{\tau_{\I}} = i \mid X_0 = k) \pi_k.
    \end{eqn}
    Therefore $\pr_i > 0$ for all $i \in \I$, and, since $\Cu \One = \One$, moreover we have that $\One^\t \pr = 1$.
    Again using $\Cu \One = \One$ and the stationarity of the original rate matrix $R$, we compute:
    \begin{eqn}
        \Rr \One &= \DiagI(\pr) \Cu^\t \Diag(\pi) R \Cu \One = \Cu^\t \Diag(\pi) R \One = \Zero \\
        \Rr^\t \pr &= \Cu^\t R^\t \Diag(\pi) \Cu \DiagI(\pr) \pr = \Cu^\t R^\t \Diag(\pi) \Cu \One = \Cu^\t R^\t \pi = \Zero,
    \end{eqn}
    verifying that $\pr$ is the stationary distribution and moreover that $\Rr \One  = \Zero$.
    
    Next, we establish that the offdiagonal elements are nonnegative, i.e., $\Rr_{i,j} \geq 0$ for all $i \neq j$.
    By \Cref{eq:diagpiR}, it suffices to prove the same about the matrix $\Cu^\t \Diag(\pi) R \Cu$.
    By \cref{eq:laplace-equation}, for any $i \in \I$, $\Delta \Cu_{:,i} = \Minus \Diag(\pi) R \Cu_{:,i}$ is zero on all rows except $\I$. 
   Then 
    \begin{eqn}
        \Cu^\t \Diag(\pi) R \Cu = \Cu^\t \Id_{:, \I} \Id_{\I, :} \Diag(\pi) R \Cu = \Diag(\pi_{\I}) R_{\I,:} \Cu, 
    \end{eqn}
    where we have additionally used the identity that $\Id_{\I,:}\Cu = \Cu_{\I,:} = \Id$.
    Therefore it suffices to show that $R_{i,:} \Cu_{:,j} \geq 0$ for all $i,j \in \I, i \neq j$.
    Using the fact that the rows of $R$ sum to 0, we compute that
    \begin{eqn}
        R_{i,:} \Cu_{:,j} = \sum_{k \neq i} R_{i,k} (\Cu_{k,j} - \Cu_{i,j}) = \sum_{k \neq i} R_{i,k} \Cu_{k,j}
    \end{eqn}
    where the last equality follows from $\Cu_{i,j} = 0$, since $\Prob(X_{\tau_{\I}} = i \mid X_0 = i) = 1$.
    The final right-hand side is nonnegative since (1) $R_{i,k} \geq 0$ for any $i \neq k$ by the definition of the original chain's rate matrix and (2) $\Cu$ is entrywise nonnegative via its probabilistic definition.
    
    To guarantee irreducibility, consider that $\Cu_{k,i}$ is positive for all states $k \in \Indices$ which connect to $i$ via a path not involving $\Iminus{i}$.
    Therefore $\Rr_{i,j}$ is nonzero for any $i \neq j$ where $i$ and $j$ are connected in the adjacency graph of $R$ by a path not involving $\I \setminus \lrb{i,j}$.
    Thus $i$ and $j$ are connected in $\Rr$ if there is any path connecting them in the adjacency graph of $R$.
    Considering that the adjacency graph of $R$ is connected (since $R$ is irreducible), then it follows that the adjacency graph of $\Rr$ must be connected. 
    Therefore $\Rr$ is irreducible.

    Given that $\Rr$ is irreducible, the stationary distribution is unique.
    Together with the negative semidefiniteness of $\Diag(\pr) \Rr$, this implies that the diagonal of $\Diag(\pr) \Rr$ (and therefore $\Rr$) is strictly negative.
\end{Proof}

\ThInducedInterpretation*
\begin{Proof}{th:induced-interpretation}
    For claim 1, we compute:
    \begin{eqn}
        \Prob(\tau_{\Iminus{i}} > t \mid X_0 = i) = \Id_{i,:} \Ert{R_{\Ic \cup \lrb{i},\Ic \cup \lrb{i}}} \One.
    \end{eqn}
    It simplifies the linear algebra to express this quantity as the limit:
    \begin{eqn}
        \Prob(\tau_{\Iminus{i}} > t \mid X_0 = i) = \lim_{\beta \ra 0} \Id_{i,:} \Ert{(R - \beta \One \pi^\t)_{\Ic \cup \lrb{i},\Ic \cup \lrb{i}}} \One.
    \end{eqn}
    Then by \Cref{lem:expalpha}, we express the expected time via a limit $\alpha \ra \infty$ of a perturbed probability evolution where the chain is killed instantly on reaching $\Iminus{i}$.
    Given this approach, the probability of the chain not having reached $\Iminus{i}$ at time $t$ is: 
    \begin{eqn}
        \Prob(\tau_{\Iminus{i}} > t \mid X_0 = i) = \lim_{\Shift \ra 0} \lim_{\alpha \ra \infty} \Id_{i,:} \Ert{(R - \Shift \One \pi^\t - \alpha \Id_{:,\Iminus{i}} \Id_{:,\Iminus{i}}^\t)} \One.
    \end{eqn}
    So, integrating to find the expected lifetime of such a process:
    \begin{eqn}
        \Ex(\tau_{\Iminus{i}} \mid X_0 = i) &= \lim_{\Shift \ra 0} \lim_{\alpha \ra \infty} \int_0^\infty \Id_{i,:} \Ert{(R - \Shift \One \pi^\t - \alpha \Id_{:,\Iminus{i}} \Id_{:,\Iminus{i}}^\t)} \One \odif t.
    \end{eqn}
    Then we compute the right-hand integral for finite $\Shift > 0$ and $\alpha > 0$ to find that:
    \begin{eqn}
        \Ex(\tau_{\Iminus{i}} \mid X_0 = i) &= \lim_{\Shift \ra 0} \lim_{\alpha \ra \infty} \Minus \Id_{i,:} (R - \Shift \One \pi^\t - \alpha \Id_{:,\Iminus{i}} \Id_{:,\Iminus{i}}^\t)^\Mo \One \\
        &= \lim_{\Shift \ra 0} \lim_{\alpha \ra \infty} h_i^\Mo \Id_{i,:} (L + \Shift h h^\t + \alpha \Id_{:,\Iminus{i}} \Id_{:,\Iminus{i}}^\t)^\Mo h
    \end{eqn}
    where the last line follows from the definition of $L$, $h$, and $\pi$.
    Next, let $\Kb \Eq (L + \Shift h h^\t)^\Mo = K + \Shift^\Mo h h^\t$. 
    Then, by use of the Woodbury matrix identity, we find that:
    \newcommand{\Kbi}{((\Kb)_{\I,\I}^\Mo)}
    \begin{eqn} \label{eq:Rii-shift-limit}
        \Ex(\tau_{\Iminus{i}} \mid X_0 = i) &= \lim_{\Shift \ra 0} \frac{h^\t (\Kb)_{:,\I} \Kbi_{:,i} }{h_i \Kbi_{i,i}}.
    \end{eqn}
    Invoking the Sherman-Morrison formula and taking the limit $\Shift \ra 0$, we compute that:
    \begin{eqn}
        \Ex(\tau_{\Iminus{i}} \mid X_0 = i) &= h_i^\Mo \lrp{ \Ki \hi - \frac{(\hi^\t \Ki \hi) (\Ki)_{i,i}}{(\Ki)_{i,:} \hi} }^\Mo \\
        &= \Mo / \Rr_{i,i}
    \end{eqn}
    with the last line following by comparison to \cref{eq:induced-R-from-K} in \Cref{th:induced-from-k}.
    
    For claim 2, we use the same limiting approach, but now we measure the accumulated flux of the perturbed chain into states $\Iminus{i}$ by the following integral:
    \begin{eqn}
        \Prob(X_{\tau_{\Iminus{i}}} = j \mid X_0 = i) = \lim_{\Shift \ra 0} \lim_{\alpha \ra \infty} \int_0^\infty \Id_{i,:} \Ert{(R - \Shift \One \pi^\t - \alpha \Id_{:,\Iminus{i}} \Id_{:,\Iminus{i}}^\t)} (R - \Shift \One \pi^\t)_{:,j} \odif t
    \end{eqn}
    Again computing the integral, we compute that:
    \begin{eqn}
        \Prob(X_{\tau_{\Iminus{i}}} = j \mid X_0 = i) &= \lim_{\Shift \ra 0} \lim_{\alpha \ra \infty} 
            \Minus \Id_{i,:} \lrp{R - \Shift \One \pi^\t - \alpha \Id_{:,\Iminus{i}} \Id_{:,\Iminus{i}}^\t}^\Mo (R - \Shift \One \pi^\t)_{:,j} \\
            &= \lim_{\Shift \ra 0} \lim_{\alpha \ra \infty} \Minus \frac{h_j}{h_1} \Id_{i,:} \lrp{(L + \Shift h h^\t) + \alpha \Id_{:,\Iminus{i}} \Id_{:,\Iminus{i}}^\t}^\Mo (L + \Shift h h^\t)_{:,j}.
    \end{eqn}
    By further use of the Woodbury matrix identity, we find that:
    \begin{eqn}
        \Prob(X_{\tau_{\Iminus{i}}} = j \mid X_0 = i) &= \lim_{\Shift \ra 0} \Minus \frac{h_j \Kbi_{j,i}}{h_i \Kbi_{i,i}}
    \end{eqn}
    so that by invoking \cref{eq:Rii-shift-limit},
    \begin{eqn}
        \Prob(X_{\tau_{\Iminus{i}}} = j \mid X_0 = i) &= \Minus \Rr_{i,i}^\Mo \lim_{\Shift \ra 0} \Minus \frac{\Diag(\hi) \Kbi_{:,i}}{h^\t (\Kb)_{:,\I} \Kbi_{:,i}}.
    \end{eqn}
    Then, evaluating the limit gives:
    \begin{eqn}
        \Prob(X_{\tau_{\Iminus{i}}} = j \mid X_0 = i) &= \Minus \Rr_{i,i}^\Mo \cdot h_j \lrp{ \Kbi_{j,:} \hi - \frac{\hi^\t (\Kb)_{\I,\I}^\Mo \hi}{\Kbi_{i,:} \hi} \Kbi_{i,j} }
        &=  \Minus \frac{\Rr_{i,j}}{\Rr_{i,i}}
    \end{eqn}
    with the last line again by comparing to entry $\Rr_{i,j}$ from \cref{eq:induced-R-from-K} in \Cref{th:induced-from-k}.

    Finally, for claim 3, we much more simply compute that:
    \begin{eqn}
        \pr_i = \sum_{j} \pi_j \Cu_{j,i} = \sum_{j} \pi_j \Prob(X_{\tau_{\I}} = i \mid X_0 = j) = \Prob(X_{\tau_{\I}} = i).
    \end{eqn}
\end{Proof}

\begin{corollary}[Preservation of mean first passage times] \label{th:hitting-times-preserved}
For states $i, j \in \I$, consider the mean first passage time for $\FullChain$ (respectively $\InducedChain$) started in $i$ to reach $j$ and denote it $H_{i,j}$ (respectively $\hat{H}_{i,j}$).
        Then $H_{i,j} = \hat{H}_{i,j}$.
\end{corollary}
\begin{proof}
    Let $S \Eq \DiagI(h) K \DiagI(h)$ and $\Sr \Eq \DiagI(\hr) \Kr \DiagI(\hr)$, where $\Kr \Eq \Lr^+$.
    Then from Lemma~2.12 of \cite{Aldous1995-reversible},
    \begin{eqn}
        H_{i,j} = S_{j,j} - S_{i,j}, \qquad \hat{H}_{i,j} = \Sr_{j,j} - \Sr_{i,j}.
    \end{eqn}
    Recalling our definition $\TimeScale := (\hi^\t \Ki \hi)^\Mo$ \Cref{eq:omega} and substituting the formula for $\Kr$ from \Cref{th:induced-from-k}:
    \begin{eqn}
        \Sr &= \DiagI(\hr) \lrs{\Diag(\hr / \hi) K_{\I,\I} \Diag(\hr / \hi) - \TimeScale \hr \hr^\t} \DiagI(\hr) \\
        &= \DiagI(\hi) K_{\I,\I} \DiagI(\hi) - \TimeScale \One \One^\t \\
        &= S_{\I,\I} - \TimeScale \One \One^\t.
    \end{eqn}
    This implies $\hat{H}_{i,j} = H_{i,j}$, since the constant matrix $\TimeScale \One \One^\t$ vanishes when subtracting $\Sr_{j,j} - \Sr_{i,j}$.
\end{proof}

\ThInducedFlow*
\begin{Proof}{th:induced-flow}
    For the first statement, we rearrange terms:
    \begin{eqn} \label{eq:flow-1}
        \lim_{t \ra 0} \frac{1}{2t} \Ex \lrs{ (\Id_{X_t,:} - \Id_{X_0,:})^\t (\Id_{X_t,:} - \Id_{X_0,:})  } &= 
        \lim_{t \ra 0} \frac{1}{2t} \Ex \lrs{ \Id_{X_0,:}^\t (\Id_{X_0,:} - \Id_{X_t,:})
        + \Id_{X_t,:}^\t (\Id_{X_t,:} - \Id_{X_0,:}) } \\
        &= \lim_{t \ra 0} \frac{1}{t} \Ex \lrs{ \Id_{X_0,:}^\t (\Id_{X_0,:} - \Id_{X_t,:}) }
    \end{eqn}
    where the second line follows from reversibility.
    Then:
    \begin{eqn} \label{eq:flow-2}
        \lim_{t \ra 0} \frac{1}{t} \Ex \lrs{ \Id_{X_0,i} (\Id_{X_0,j} - \Id_{X_t,j}) }
        &= \Prob(X_0 = i) \cdot \lim_{t \ra 0} \frac{1}{t} \Ex \lrs{ \Id_{X_0,j} - \Id_{X_t,j} \mid X_0 = i}  \\
        &= \Prob(X_0 = i) \cdot (\Minus R_{i,j}) = \Minus \pi_i R_{i,j}.
    \end{eqn}
    The second statement follows by the exact same procedure. 
    For the third statement, by insertion of identity and use of \cref{eq:flow-1,eq:flow-2}:
    \begin{eqn}
        \lim_{t \ra 0} \frac{1}{2t} \Ex \lrs{ (\Cu_{X_t,:} - \Cu_{X_0,:})^\t (\Cu_{X_t,:} - \Cu_{X_0,:})  }
        &= \lim_{t \ra 0} \frac{1}{2t} \Ex \lrs{ \Cu^\t (\Id_{X_t,:} - \Id_{X_0,:})^\t (\Id_{X_t,:} - \Id_{X_0,:}) \Cu  } \\
        &= \Cu^\t \lim_{t \ra 0} \frac{1}{2t} \Ex \lrs{ (\Id_{X_t,:} - \Id_{X_0,:})^\t (\Id_{X_t,:} - \Id_{X_0,:}) } \Cu \\
        &= \Minus \Cu^\t \Diag(\pi) R \Cu
    \end{eqn}
\end{Proof}

\subsection{Features of structure-preserving compression \label{a:sp-features}}

\ThInducedIsLaplacian*
\begin{Proof}{th:induced-is-laplacian}
    This a corollary of \Cref{th:induced-basics}.
    First, observe that 
    \begin{eqn}
        \Lr = \Minus \Diag(\hr) \Rr \DiagI(\hr) = \Minus\DiagI(\hr) \Cu^\t \Diag(\pi) R \Cu \DiagI(\hr).
    \end{eqn}
    Since $\Diag(\pi) R$ is symmetric by the reversibility of the original chain, $\Lr$ is therefore symmetric.
    Moreover, since $\hr$ is entrywise positive, it follows from \Cref{th:induced-basics} that $\Lr_{i,j} \leq 0$ for all $i \neq j$ and $\Lr_{i,i} > 0$ for all $i$.

    Next, since $\Lr$ is obtained from $\Minus\Rr$ by a similarity transform, it has the same eigenvalues as $\Minus\Rr$.
    This shows that $\Lr$ is positive semidefinite with a single zero eigenvalue.
    
    Finally, observe that $\snorm{\hr} = 1$ since $\One^\t \hr^2 = \One^\t \pr = 1$.
    Therefore we verify that the zero eigenvalue corresponds to eigenvector $\hr$ by computing that $\Lr \hr = \Minus\Diag(\hr) \Rr \DiagI(\hr) \hr = \Minus\Diag(\hr) \Rr \One = \Zero$.
\end{Proof}


\ThSpNonnegativity*
\begin{Proof}{th:sp-nonnegativity}
    To prove entrywise nonnegativity, we expand the definition of $\NonSymObliqueP_{\I}(t)$:
    \begin{eqn}
        \NonSymObliqueP_{\I}(t) = \Cu \ert{\Rr}  \DiagI(\pr) \Cu^\t \Diag(\pi).
    \end{eqn}
    $\Cu$ is entrywise nonnegative by its probabilistic construction, while $\p$ and $\pr$ are entrywise positive (the latter by its definition in \cref{eq:stationary-chp}).
    Since $\Rr$ is nonnegative outside of its diagonal (by \Cref{th:induced-basics}), $\ert{\Rr}$ is entrywise nonnegative.
    Therefore $\NonSymObliqueP_{\I}(t)$ is entrywise nonnegative, and $\ObliqueP_{\I}(t)$ is as well, since the latter is obtained as a similarity transformation of the former via $\Diag(h) \cdots \DiagI(h)$.

    Next, since $\ert{\Rr}$ is the solution of the probability evolution of a Markov chain with rate matrix $\Rr$, it follows that $\ert{\Rr} \One = \One$.
    Therefore, since $\pr \Eq \Cu^\t \pi$ and $\Cu \One = \One$:
    \begin{eqn}
        \Cu \ert{\Rr}  \DiagI(\pr) \Cu^\t \Diag(\pi) \One = \Cu \ert{\Rr} \One = \Cu \One = \One,
    \end{eqn}
    showing that $\NonSymObliqueP_{\I}(t)$ is row stochastic.

    Finally, by the convergence theorem for irreducible Markov chains, $\lim_{t \ra \infty} \ert{\Rr} = \One \pr^\t$.
    Since $\expt{\Lr} = \Diag(\hr) \ert{\Rr} \DiagI(\hr)$ by the definition of $\Lr$, this shows that 
    $\lim_{t \ra \infty} \expt{\Lr} = \hr \hr^\t$.
    Therefore $\lim_{t \ra \infty} \ObliqueP_{\I}(t) = \Co \hr \hr^\t \Co^\t = h h^\t$ by direct computation.
    Since $\Lr$ is symmetric, $\NonSymObliqueP_{\I}(t)$ is symmetric by inspection of \cref{eq:oblique-p}.
\end{Proof}

\ThSpIntegral*
\begin{Proof}{th:sp-integral}
    Using $\Ca \Eq \KA_{:, \I} \KA_{\I,\I}^\Mo \Diag\lrp{\tfrac{\hi}{\ha}}$ (\Cref{def:killed-committor}), we compute:
    \begin{eqn}
        \int_0^\infty \ObliqueP_{\I,\Kill}(t) \; \odif t 
        &= \int_0^\infty \Ca \expt{\Ca^\t \La \Ca} \Ca^\t \; \odif t \\
        &= \Ca (\Ca^\t \La \Ca)^\Mo \Ca^\t \\
        &= \KA_{:, \I} \KA_{\I,\I}^\Mo \Diag\lrp{\tfrac{\hi}{\ha}} \lrs{\Diag\lrp{\tfrac{\hi}{\ha}} \KA_{\I,\I}^\Mo \Diag\lrp{\tfrac{\hi}{\ha}}}^\Mo \Diag\lrp{\tfrac{\hi}{\ha}} \KA_{\I,\I}^\Mo \KA_{\I,:} \\
        &= \KA_{:, \I} \KA_{\I,\I}^\Mo \KA_{\I,:}.
    \end{eqn}
\end{Proof}

\subsection{Marked process formulation \label{a:marked-formulation}}

\ThMarkedChainDynamics*
\begin{Proof}{th:marked-chain-dynamics}
    We will explain each subblock of $\Rex$, though the formulation may best be verified by inspection.
    \begin{enumerate}
        \item \underline{Upper-left block (marked state $\ra$ marked state)}.
        The rates between marked states $(i,i) \ra (i',i')$ are inherited from the rates $i \ra i'$ in the original chain $\FullChain$. 
        Other than transitions to other marked states, a marked state $(i,i)$ can only transition to an unmarked state $(i,j)$, $j \in \Ic$, with the same marking. Thus the outgoing rate (reflected in the diagonal of $\Rex_{:\Ia,:\Ia}$) is unchanged as well.
        \item \underline{Lower-right block (unmarked state $\ra$ unmarked state)}.
        The rates between unmarked states $(i,j) \ra (i',j')$ are inherited from the rates $j \ra j'$ in the original chain $\FullChain$ when $i=i'$, otherwise they are zero. Indeed, 
        $\Id_{\I,\I} \otimes R_{\Ic,\Ic}$ is the block-diagonal repetition of $R_{\Ic,\Ic}$.
        \item \underline{Upper-right block (marked state $\ra$ unmarked state)}.
        The only unmarked states to which a marked state $(i,i)$ can transition are those of the form $(i,j)$, $j\in \Ic$, sharing the same marking, with rates inherited from the rates $i \ra j$ in the original chain $\FullChain$.
        The face-splitting product selects out these feasible transitions and endows them with the inherited rates from $R$.
        \item \underline{Lower-left block (unmarked state $\ra$ marked state)}.
        Any unmarked state $(i,j)$ is allowed to transition to any marked state $(k,k)$, with a rate inherited from the rate $j \ra k$ in the original chain $\FullChain$.
        Block repetition of the rates $R_{\Ic,\I}$ yields $\One_{\I} \otimes R_{\Ic,\I}$.
    \end{enumerate}
    
    Now we turn to the stationary distribution $\pe$. By ergodicity, the formula for $\pe$ can be computed by considering the law of the marked chain in the infinite-time limit, from arbitrary initialization. In particular, in this limit, the top block must coincide with the limiting law of the original chain, restricted to $\I$, i.e., $\pi_{\I}$.
    
    Meanwhile, in this limit the probability that the chain is in another state $j \in \Ic$ while last having visited $i \in \I$ is $\pi_j \Cu_{j,i}$.
    To see this, note that $\lim_{t \ra \infty} \Prob(X_t = j) = \pi_j$, while $\lim_{t \ra \infty} \Prob(X_s = i : s = \max(s': s' \leq t, X_{s'} \in \I) \mid X_t = j) = \Cu_{j,i}$ by considering the chain run in reverse.
    This determines the lower block of $\pe$.
\end{Proof}


\PropMarkeig*
\begin{Proof}{prop:markeig}
We will first exhibit $n$ right eigenvectors of $\Rex$. To do so,
let $(\lambda,z)$ denote an eigenpair of the original rate matrix,
and consider the blocks $z_{\I}\in\R^{\vert\I\vert}$, $z_{\Ic}\in\R^{n-\vert\I\vert}$.
Then we claim that 
\begin{equation}
\left[\begin{array}{c}
z_{\I}\\
\mathbf{1}_{\I}\otimes z_{\Ic}
\end{array}\right]\label{eq:righteig}
\end{equation}
 defines a right eigenvector of $\Rex$ with eigenvalue $\lambda$.
This can be verified directly using the block presentation of $\Rex$
\Cref{eq:expanded-R}.

Next, for every eigenpair $(\lambda,v)$ of of $R_{\Ic,\Ic}$, we
exhibit $\vert\I\vert-1$ linearly independent left eigenvectors of
$\Rex$ with the same eigenvalue $\lambda$. To do so, let $u\in\R^{\vert\I\vert}$
such that $\mathbf{1}_{\I}^{\top}u=0$. We claim that in this case,
\begin{equation}
\left[\begin{array}{c}
0\\
u\otimes v
\end{array}\right]\label{eq:lefteig}
\end{equation}
 defines a left eigenvector of $\Rex$ with eigenvalue $\lambda$.
This can also be verified directly using the block presentation of
$\Rex$ \Cref{eq:expanded-R}. Since we can choose $\vert\I\vert-1$ linearly independent
vectors $u$ satisfying $\mathbf{1}_{\I}^{\top}u=0$, we have achieved
our goal.

Note that 
\[
n+(\vert\I\vert-1)(n-\vert\I\vert)=m:=\vert\I\vert+\vert\I\vert(n-\vert\I\vert),
\]
 the dimension of the matrix $\Rex$. So formally, it seems that we
have exhausted all the eigenvalues. But if some eigenvalues of $R$
and $R_{\Ic,\Ic}$ coincide with one another, it is not clear at this
point that they are repeated with the appropriate multiplicity as
eigenvalues of $\Rex$.

We can rule out such difficulties emphatically without resorting to
soft arguments. To do so, let $X\in\R^{m\times n}$ collect our linearly
independent right eigenvectors (\ref{eq:righteig}) as its columns,
and let $Y\in\R^{m\times(m-n)}$ collect our linearly independent
right eigenvectors (\ref{eq:lefteig}) as its columns. Therefore 
\begin{equation}
\Rex X=X\Lambda_{\mathrm{R}},\quad\Rex^{\top}Y=Y\Lambda_{\mathrm{L}},\label{eq:leftright}
\end{equation}
 where $\Lambda_{\mathrm{R}}$ and $\Lambda_{\mathrm{L}}$ are diagonal
matrices collecting the right and left eigenvalues, respectively.
Importantly, note via (\ref{eq:righteig}) and (\ref{eq:lefteig})
that our right eigenvectors are orthogonal to our left eigenvectors,
i.e., $Y^{\top}X=0$.

Then define a change-of-basis matrix 
\[
G=\left(\begin{array}{cc}
X & Y(Y^{\top}Y)^{-1}\end{array}\right),
\]
 and observe that 
\[
G^{-1}=\left(\begin{array}{c}
(X^{\top}X)^{-1}X^{\top}\\
Y^{\top}
\end{array}\right).
\]
One can verify by direct computation that $G^{-1}G=\Id$.

Then we can compute 
\[
G^{-1}\Rex G=\left(\begin{array}{cc}
(X^{\top}X)^{-1}X^{\top}\Rex X & *\\
0 & Y^{\top}\Rex Y(Y^{\top}Y)^{-1}
\end{array}\right)=\left(\begin{array}{cc}
\Lambda_{\mathrm{R}} & *\\
0 & \Lambda_{\mathrm{L}}
\end{array}\right),
\]
 where the upper-right block is unimportant to us and the second equality
follows from (\ref{eq:lefteig}). Then we have triangularized $\Rex$,
revealing the eigenvalue structure that we desired to prove.
\end{Proof}

\subsection{Properties of projections of the marked process \label{a:marked-properties}}

Here we provide proofs of the fundamental linear-algebraic properties satisfied by $\Co$ (the committor), $\Qo$ (the projection to position space), and $\So$ (the projection to marking space).
We must also consider the analogous killed operators $\Ca$ (\Cref{def:killed-committor}) and $\Qa$, $\Sa$ defined as follows:

\begin{definition}[Formulation of marked chain operators including killing] \label{def:killed-projections}
For a killing rate $\Kill > 0$ and $\Ca$ as in \Cref{def:killed-committor}, define:
\begin{eqn}
    \pea & \  \Eq \  \begin{bmatrix} \pi_{\I} \\ \mathrm{vec}(\Diag(\pi_{\Ic}) (\Cu_\Kill)_{\Ic, :}) \end{bmatrix}, \qquad
    &\hea & \  \Eq \  \pea^{1/2}, \\
    \Qa & \  \Eq \  \Diag(\hea) \Qu \DiagI(h), \qquad
    &\Sa & \  \Eq \ \Diag(\hea) \Su \DiagI(\ha).
\end{eqn}
Under the same indexing convention as in \Cref{eq:pealt}, since $(\tilde{C}_\gamma)_{\I,:} = \Id_{\I,\I}$, we can write 
\begin{eqn} \label{eq:pealt2}
(\pea)_{(i,j)} = (\tilde{C}_{\gamma})_{j,i} \, \pi_j,
\end{eqn}
extending the above specification by zeros to the impossible states.
\end{definition}




First, while $\Co$ is not semi-orthogonal (i.e. $\Co^\t \Co \neq \Id$), it is straightforward to show that $\Co$ is a contraction as follows:

\begin{lemma}[Committor as a contraction] \label{th:committor-norm}
    For $\gamma > 0$ let $\Ca$ be as in \Cref{def:killed-committor} and $\Co = \lim_{\Kill \ra 0} \Ca$ as in \Cref{th:killed-committor}.
    Then $\Ca$ is a contraction in the sense that 
    \begin{eqn}
        \snorm{\Ca} \leq 1
    \end{eqn}
    for $\Kill > 0$. Moreover, 
    \begin{eqn}
        \snorm{\Co} = 1, \quad
        \Co^\t h = \hr, \quad
        \Co \hr = h.
    \end{eqn}
\end{lemma}
\begin{proof}
    First note that $\Cu$ and $\Cu_\Kill$ are entrywise nonnegative by their probabilistic definitions (\Cref{def:sym-operators} and \Cref{eq:killed-committor-prob}). 
    Then consider the matrices 
    \begin{eqn}
        \Co^\t \Co &= \DiagI(\hr) \Cu^\t \Diag(\pi) \Cu \DiagI(\hr) \\
        \Ca^\t \Ca &= \DiagI(\ha) \Cu^\t_\Kill \Diag(\pi) \Cu_\Kill \DiagI(\ha).
    \end{eqn}
    $\Co^\t \Co$ and $\Ca^\t \Ca$ are symmetric positive semidefinite by construction.
    
    Observe that $\Co^\t \Co$ is similar to the matrix $M \Eq \DiagI(\pr) \Cu^\t \Diag(\pi) \Cu$, while $\Ca^\t \Ca$ is similar to the matrix $M_\Kill \Eq \DiagI(\pa) \Cu_\Kill^\t \Diag(\pi) \Cu_\Kill$.
    Both $M$ and $M_\Kill$ have nonnegative entries by construction.
    Moreover $M$ is row stochastic since:
    \begin{eqn}
        M \One = \DiagI(\pr) \Cu^\t \Diag(\pi) \Cu \One = \DiagI(\pr) \Cu^\t \pi = \One
    \end{eqn}
    where we have used the fact that $\Cu \One = 1$, again following from the probabilistic definition of $\Cu$, and recall that $\pr \Eq \Cu^\t \pi$.
    Since $M$ is entrywise nonnegative and row-stochastic, its operator norm with respect to the infinity vector norm is 1, hence the eigenvalues of $M$ are bounded by 1.
    Since $\Co^\t \Co$ is similar to $M$, it too has a maximum eigenvalue of 1, implying that $\snorm{\Co} \leq 1$.

    The same logic applies to $\Ca^\t \Ca$, except that we use the softer condition following from \Cref{th:killed-committor-prob} that $\Cu_\Kill \One \leq \One$ so that
    \begin{eqn}
        M_\Kill \One = \DiagI(\pr_\Kill) \Cu_\Kill^\t \Diag(\pi) \Cu_\Kill \One \leq \DiagI(\pr_\Kill) \Cu_\Kill^\t \Diag(\pi) \One = \DiagI(\pa) \Cu_\Kill^\t \pi = \One
    \end{eqn}
    in an entrywise sense.
    Therefore $M_\Kill$ is row substochastic, implying that $\snorm{\Ca} \leq 1$.

    The remaining statements can be verified by computing:
    \begin{eqn}
        \Co^\t h &= \DiagI(\hr) \Cu^\t \pi = \hr, \\
        \Co \hr &=  \Diag(h) \Cu \DiagI(\hr) \hr = \Diag(h) \Cu \One = h,
    \end{eqn}
    again using  $\pr = \Cu^\t \pi$ and $\Cu \One = \One$.
    Since $\snorm{\hr}=1$ and $\snorm{h}=1$ by construction, this shows that $\snorm{C}=1$.
\end{proof}


Next, both $\Qo$ and $\So$ are entrywise non-negative matrices with orthonormal columns.
On the other hand, in the killed case for $\gamma > 0$, $\Sa$ has orthonormal columns, while $\Qa$ only satisfies $\snorm{\Qa} = 1$ but does not have orthonormal columns. 
These results are stated in the next two lemmas.

\begin{lemma}[Basic properties of $\Qo$ and $\Qa$] \label{th:q-contraction}
    Given $\Qo$ as specified in \Cref{eq:marked-projections}: 
    \begin{eqn}
        \Qo^\t \Qo = \Id, \qquad \Qo \Qo^\t \he = \he. 
    \end{eqn}
    Meanwhile, for any $\Kill > 0$, with $\Qa$ as defined in \Cref{def:killed-projections}:
    \begin{eqn}
        \snorm{\Qa} = 1.
    \end{eqn}
\end{lemma}
\begin{proof}
    Recall (cf. \Cref{eq:tildeproj} and the ensuing discussion) that $\Qu$ is the indicator matrix which yields the position space representation of a state in the augmented space, i.e., $\Qu_{(k,l),i} = \delta_{i,l}$ for all $(k,l)$ in the augmented space and $i \in [n]$. For this discussion, it is convenient to index the augmented space by $(k,l) \in \I \times [n]$, and (abusing notation slightly) to extend matrices and vectors by zeros for the omitted ordered pair indices $(k,l)$ for which $k,l \in \I$ but $k\neq l$.
    
    Note that marginalizing the stationary distribution of the marked chain over the marking space implies that $\pi = \Qu^\t \pe$, or under our indexing convention,  $\pi_i = \sum_{(k,l)} \Qu_{(k,l),i} \pe_{(k,l)} = \sum_{k \in \I} \pe_{(k,i)}$. (Algebraically, this identity can be confirmed from \Cref{eq:pealt} using the fact that $\tilde{C} \One = \One$.)
    
    Moreover, compute: 
    \begin{eqn}
        \Qo_{:,i}^\t \Qo_{:,j} &= h_i^\Mo h_j^\Mo \Qu_{:,i}^\t \Diag(\pe) \Qu_{:,j} = 
        h_i^\Mo h_j^\Mo
        \sum_{(k,l)} \Qu_{(k,l),i} \, \pe_{(k,l)} \, \Qu_{(k,l),j} \\
        &=
        h_i^\Mo h_j^\Mo \delta_{i,j} \sum_{k \in \I} \pe_{(k,i)} = h_i^\Mo h_j^\Mo \pi_i \delta_{i,j} = \delta_{i,j},
    \end{eqn}
    i.e., $Q^\t Q = \Id$ as was to be shown.
    
    To show the second conclusion:
    \begin{eqn}
        \Qo \Qo^\t \he &= \Diag(\he) \Qu \DiagI(\pi) \Qu^\t \Diag(\he) \he \\
        &= \Diag(\he) \Qu \DiagI(\pi) \Qu^\t \pe \\
        &= \Diag(\he) \Qu \One \\
        &= \he 
    \end{eqn}
    where the third line is implied by $\pi = \Qu^\t \pe$ and the fourth by $\Qu \One = \One$ (since each row of $\Qu$ consists of a single 1, corresponding to the position of the state).

    It remains to show that $\Vert Q_\gamma \Vert_2 = 1$. Note that by the probabilistic interpretation of $\tilde{C}_\gamma$ (\Cref{th:killed-committor-prob}), $\tilde{C}_\gamma \One \leq \One$ (entrywise).
    Thus, by the specification \Cref{eq:pealt2} of $\pea$: 
    \begin{eqn} \label{eq:killed-sums}
        \sum_{k \in \I} (\pea)_{(k,i)} = \sum_{k \in \I} \pi_i (\Cu_\Kill)_{i,k} = \pi_i \Cu_{i,:} \One \leq \pi_i.
    \end{eqn}
    Note that equality holds when $i \in \I$. Therefore, computation of $\Qa^\t \Qa$ yields:
    \begin{eqn}
        (\Qa)_{:,i}^\t (\Qa)_{:,j} &= h_i^\Mo h_j^\Mo \Qu_{:,i}^\t \Diag(\pea) \Qu_{:,j} = 
        h_i^\Mo h_j^\Mo
        \sum_{(k,l)} \Qu_{(k,l),i} \, (\pea)_{(k,l)} \, \Qu_{(k,l),j} \\
        &=
        h_i^\Mo h_j^\Mo \delta_{i,j} \sum_{k \in \I} (\pea)_{(k,i)} \leq h_i^\Mo h_j^\Mo \pi_i \delta_{i,j} = \delta_{i,j},
    \end{eqn}
    using \Cref{eq:killed-sums} for the middle inequality.
    Therefore $\Qa^\t \Qa$ is a diagonal matrix upper-bounded by $\Id$ entrywise.
    Furthermore, the diagonal entries in the $\I$ block are 1, yielding the strict equality of $\snorm{\Qa} = 1$.
\end{proof}

\begin{lemma}[Basic properties of $\So$ and $\Sa$] \label{th:s-orthogonality}
    For any $\Kill \geq 0$:
    \begin{eqn}
        \So^\t \So &= \Id, \qquad \snorm{\So} = 1 \\
        \qquad \Sa^\t \Sa &= \Id, \qquad \snorm{\Sa} = 1
    \end{eqn}
    and
    \begin{eqn}
        \So \So^\t \he = \he .
    \end{eqn}
\end{lemma}
\begin{proof}
    Recall (cf. \Cref{eq:tildeproj} and the ensuing discussion) that $\Su$ is the indicator matrix which yields the marking space representation of a state in the augmented space, i.e., that 
    $\Su_{(l,k),i} = \delta_{i,l}$ for all $(l,k)$ in the augmented space, $i \in \I$. We use the same indexing convention for the augmented space as in the proof of \Cref{th:q-contraction}.

    Direct computation (using the specification \Cref{eq:pealt} for $\pe$) yields that $\Su^\t \pe = \Cu^\t \pi = \pr$, the latter equality following from the definition of $\pr$.
    (In our indexing convention, one may equivalently write that $\sum_{k \in \Indices} \pe_{(i,k)} = \pr_i$.)
    Moreover, compute: 
    \begin{eqn}
        \So_{:,i}^\t \So_{:,j} &= \hr_i^\Mo \hr_j^\Mo \Su_{:,i}^\t \Diag(\pe) \Su_{:,j} 
        = \hr_i^\Mo \hr_j^\Mo \sum_{(l,k)} \Su_{(l,k),i} \,  \pe_{(l,k)} \,  \Su_{(l,k),j} \\
        &= \hr_i^\Mo \hr_j^\Mo \delta_{i,j} \sum_{k \in \Indices} \pe_{(i,k)}  = \delta_{i,j} \hr_i^\Mo \hr_j^\Mo \pr_i = \delta_{i,j}.
    \end{eqn}
    For the killed case, analogously, direct computation (using the specification \Cref{eq:pealt2} for $\pea$) yields that $\Su^\t \pea = \Cu_{\Kill}^\t \pi = \pa$, the latter by the definition of $\pa$ \Cref{eq:killed-chp}. 
    Then:
    \begin{eqn}
        (\Sa)_{:,i}^\t (\Sa)_{:,j} &= (\ha)_i^\Mo (\ha)_j^\Mo \Su_{:,i}^\t \Diag(\pea) \Su_{:,j} 
        = (\ha)_i^\Mo (\ha)_j^\Mo \sum_{(l,k)} \Su_{(l,k),i} \,  (\pea)_{(l,k)} \,  \Su_{(l,k),j} \\
        &= (\ha)_i^\Mo (\ha)_j^\Mo \delta_{i,j} \sum_{k \in \Indices} (\pea)_{(i,k)} = \delta_{i,j} (\ha)_i^\Mo (\ha)_j^\Mo (\pa)_i = \delta_{i,j}.
    \end{eqn}
    
    To show the second conclusion:
    \begin{eqn}
        \So \So^\t \he &= \Diag(\he) \Su \DiagI(\pr) \Su^\t \Diag(\he) \he \\
        &= \Diag(\he) \Su \DiagI(\pr) \Su^\t \pe \\
        &= \Diag(\he) \Su \One \\
        &= \he
    \end{eqn}
    where the third line is implied by $\pr = \Su^\t \pe$ and the fourth by $\Su \One = \One$ (since each row of $\Su$ consists of a single 1, corresponding to the marking of the state).
\end{proof}


The committor $\Co$ may be related to $\Qo$ and $\So$ as follows: 
\begin{lemma}[Relating marked chain projection to the committor] \label{th:q-s-equals-c}
    With $\Qo$, $\So$, and $\Co$ as in \Cref{eq:marked-projections} and \Cref{th:committor}:
    \begin{eqn}
        \Qo^\t \So = \Co.
    \end{eqn}
    In the killed case, for $\Kill > 0$ and $\Qa$, $\Sa$, and $\Ca$ as in \Cref{def:killed-projections} and \Cref{def:killed-committor}:
    \begin{eqn}
        \Qa^\t \Sa = \Ca.
    \end{eqn}
\end{lemma}
\begin{proof}
    First compute: 
    \begin{eqn}
        \Qo^\t \So = \Diag(h) \Qu^\t \Diag(\pe) \Su \Diag(1/\hr).
    \end{eqn}
    Then, by expanding the blockwise definitions of $\Qu$ and $\Su$:
    \begin{eqn}
        \Qu^\t \Diag(\pe) \Su &= \Id_{:,\I} \Diag(\pi_{\I}) + 
        \Id_{:,\Ic} (\One_{\I} \otimes \Id_{\Ic,\Ic})^\t \Diag(\mathrm{vec}(\Diag(\pi_{\Ic}) \Cu_{\Ic,:})) (\Id_{\I,\I} \otimes \One_{\Ic}) \\
        &= \Id_{:,\I} \Diag(\pi_{\I}) + \Id_{:,\Ic} \Diag(\pi_{\Ic}) \Cu_{\Ic,:} \\
        &= \Diag(\pi) \Cu.
    \end{eqn}
    Finally, substituting back in yields 
    \begin{eqn}
        \Qo^\t \So = \Diag(1/h) \Diag(\pi) \Cu \Diag(1/\hr) = \Co
    \end{eqn}
    using the definitions of $h$, $\pi$, and $\Co$.

    The exact same logic applies above for the killed case, with $\pea$, $\hea$, and $\Cu_{\Kill}$ replacing $\pe$, $\he$, and $\Cu$, respectively.
\end{proof}

\begin{lemma}[Original Laplacian in terms of the marked chain] \label{th:qlq}
    With $\Qo$ as in \cref{eq:marked-projections} and $\Le$ as in \Cref{def:marked-laplacian}:
    \begin{eqn}
        \Qo^\t \Le \Qo = L.
    \end{eqn}
    In fact $\Qo$ defines an invariant subspace of $\Le$ in that $\Le \Qo = \Qo L$.
\end{lemma}
\begin{proof}
    Note that it suffices to show that $\Le \Qo = \Qo L$, which implies the first statement via left-multiplication by $\Qo^\t$ (since $\Qo$ has orthonormal columns). In turn, by unraveling the similarity transformations of \Cref{def:marked-laplacian} and \Cref{eq:marked-projections}  used to define $\Le$ and $\Qo$ in terms of $\Rex$ and $\Qu$, respectively, we see that it suffices to show that $\Rex \Qu = \Qu R$. Then by substituting the blockwise expressions \Cref{eq:expanded-R} and \Cref{eq:tildeproj} for $\Rex$ and $\Qu$, we verify by direct computation that 
    \begin{eqn}
        \Rex\tilde{Q}=\left[\begin{array}{cc}
R_{\I,\I} & R_{\I,\Ic}\\
\mathbf{1}_{\I}\otimes R_{\Ic,\I} & \mathbf{1}_{\I}\otimes R_{\Ic,\Ic}
\end{array}\right]=\tilde{Q} R.
    \end{eqn}
    \end{proof}

\begin{lemma}[Induced Laplacian in terms of the marked chain] \label{th:sls}
    With $\So$ as in \cref{eq:marked-projections}, $\Le$ as in \Cref{def:marked-laplacian}, and $\Lr$ as in \cref{eq:oblique-p}:
    \begin{eqn}
        \So^\t \Le \So = \Lr.
    \end{eqn}
    Furthermore, for any $\Kill > 0$, with $\Sa$ as in \Cref{def:killed-projections} and $\Lea$ as in \Cref{eq:Lea}:
    \begin{eqn}
        \Sa^\t \Lea \Sa = \Lr_\Kill.
    \end{eqn}
\end{lemma}
\begin{proof}
We tackle the first statement first. By unraveling the diagonal transformations
\Cref{def:marked-laplacian} and \Cref{eq:marked-projections}
used to define $\mathring{L}$ and $W$ in terms of $\mathring{R}$
and $\tilde{W}$, respectively, we see that it suffices to show that
\[
\tilde{W}^{\top}\mathrm{diag}(\mathring{\pi})\mathring{R}\tilde{W}=\mathrm{diag}(\hat{\pi})\hat{R}.
\]
 Recall from \Cref{eq:simple-induced} that 
\[
\hat{R}=\mathrm{diag}^{-1}(\hat{\pi})\tilde{C}^{\top}\mathrm{diag}(\pi)R\tilde{C},
\]
 so in turn it suffices to show that 
\begin{equation}
\tilde{W}^{\top}\mathrm{diag}(\mathring{\pi})\mathring{R}\tilde{W}=\tilde{C}^{\top}\mathrm{diag}(\pi)R\tilde{C}.\label{eq:LHSRHS}
\end{equation}

Let us simplify the right-hand side of \Cref{eq:LHSRHS}.
By the Laplace equation \Cref{eq:laplace-equation}, we know that $(R\tilde{C})_{\Ic,:}=0$.
(Note that this can also be verified by direct computation, as we
explicitly demonstrate below in the killed $\gamma>0$ case.) Moreover,
$(\tilde{C})_{\I,:}=\Id_{\I,\I}$. Therefore 
\begin{equation}
\tilde{C}^{\top}\mathrm{diag}(\pi)R\tilde{C}=\left[\begin{array}{c}
\Id_{\I,\I}\\
0
\end{array}\right]^{\top}\mathrm{diag}(\pi)\left[\begin{array}{c}
(R\tilde{C})_{\I,:}\\
0
\end{array}\right]=\mathrm{diag}(\pi_{\I})R_{\I,:}\tilde{C}.\label{eq:RHS}
\end{equation}
 Now by reversibility, $\mathrm{diag}(\pi)R$ is symmetric, so the
left-hand side of \Cref{eq:RHS} is symmetric, hence
so is the right-hand side, and in summary, for the proof it suffices
to show that 
\begin{equation}
\tilde{W}^{\top}\mathrm{diag}(\mathring{\pi})\mathring{R}\tilde{W}=\left[\mathrm{diag}(\pi_{\I})R_{\I,:}\tilde{C}\right]^{\top}.\label{eq:LHSRHS2}
\end{equation}

Now we turn to the left-hand side of \Cref{eq:LHSRHS}.
It is useful to adopt the same indexing convention as in the proof
of \Cref{th:q-contraction}. Within this convention, recall from \eqref{eq:pealt} that $\mathring{\pi}_{(i,j)}=\tilde{C}_{j,i}\pi_{j}$.
Moreover, for any ordered pair $(i,j)$ that is \emph{possible}
(i.e., not omitted from the augmented space), we have that 
\begin{equation}
\mathring{R}_{(i,j),(i',j')}=\begin{cases}
\delta_{i',j'}R_{j,j'}, & j'\in\I\\
\delta_{i,i'}R_{j,j'}, & j'\in\Ic.
\end{cases}\label{eq:ReEntrywise}
\end{equation}
 The first line follows from the fact that during a transition to
a position state $j'\in\I$, we must pick up the new marking $j'$.
The second line follows from the fact that during a transition to
a position state $j'\in\Ic$, we must retain our old marking.

Now since the expression $\mathring{\pi}_{(i,j)}=\tilde{C}_{j,i}\pi_{j}$
extends by zeros to all impossible states, in turn the expression
\begin{equation}
\left(\mathrm{diag}(\mathring{\pi})R\right)_{(i,j),(i',j')}=\begin{cases}
\tilde{C}_{j,i}\pi_{j}\delta_{i',j'}R_{j,j'}, & j'\in\I\\
\tilde{C}_{j,i}\pi_{j}\delta_{i,i'}R_{j,j'}, & j'\in\Ic,
\end{cases}\label{eq:piRcon}
\end{equation}
 extends $\mathrm{diag}(\mathring{\pi})R$ by zeros and satisfies
our indexing convention.

Then we expand the left-hand side of \Cref{eq:LHSRHS2}
and insert $\tilde{W}_{(i,j),k}=\delta_{i,k}$ to compute, for $k,k'\in\I$:
\begin{align*}
\left(\tilde{W}^{\top}\mathrm{diag}(\mathring{\pi})\mathring{R}\tilde{W}\right)_{k,k'} & =\sum_{i,i'\in\I,\,j,j'\in[n]}\tilde{W}_{(i,j),k}\left(\mathrm{diag}(\mathring{\pi})R\right)_{(i,j),(i',j')}\tilde{W}_{(i',j'),k'}\\
 & =\sum_{j,j'\in[n]}\left(\mathrm{diag}(\mathring{\pi})R\right)_{(k,j),(k',j')}\\
 & =\sum_{j\in[n],\,j'\in\I}\left(\mathrm{diag}(\mathring{\pi})R\right)_{(k,j),(k',j')}+\sum_{j\in[n],\,j'\in\Ic}\left(\mathrm{diag}(\mathring{\pi})R\right)_{(k,j),(k',j')}\\
 & \overset{(*)}{=}\sum_{j\in[n],\,j'\in\I}\tilde{C}_{j,k}\pi_{j}\delta_{k',j'}R_{j,j'}+\sum_{j\in[n],\,j'\in\Ic}\tilde{C}_{j,k}\pi_{j}\delta_{k,k'}R_{j,j'}\\
 & =\sum_{j\in[n]}\tilde{C}_{j,k}\pi_{j}R_{j,k'}+\delta_{k,k'}\sum_{j\in[n],\,j'\in\Ic}\tilde{C}_{j,k}\pi_{j}R_{j,j'}\\
 & \overset{(\star)}{=}\sum_{j\in[n]}\pi_{k'}R_{k',j}\tilde{C}_{j,k}+\delta_{k,k'}\sum_{j\in[n],\,j'\in\Ic}\pi_{j'}R_{j',j}\tilde{C}_{j,k}\\
 & =\pi_{k'}(R\tilde{C})_{k',k}+\delta_{k,k'}\sum_{j'\in\Ic}\pi_{j'}(R\tilde{C})_{j',k},
\end{align*}
where in the $(*)$ equation we have substituted \Cref{eq:piRcon} and in the $(\star)$ equation we have used reversibility: $\pi_i R_{i,j} = \pi_j R_{j,i}$ for all $i,j$.
 Now in the last expression, $(R\tilde{C})_{j',k}=0$ for $j'\in\Ic$,
again by the Laplace equation \Cref{eq:laplace-equation}. So we conclude that 
\[
\left(\tilde{W}^{\top}\mathrm{diag}(\mathring{\pi})\mathring{R}\tilde{W}\right)_{k,k'}=\pi_{k'}(R\tilde{C})_{k',k}
\]
 for all $k,k'\in\I$, which means precisely that \Cref{eq:LHSRHS2}
holds, as desired.

This concludes the proof of the first statement. For the second statement,
we want to show that $W_{\gamma}^{\top}\mathring{L}_{\gamma}W_{\gamma}=\hat{L}_{\gamma}$,
where $\gamma>0$. Recall that $\hat{L}_{\gamma}=C_{\gamma}^{\top}L_{\gamma}C_{\gamma}$,
so by unraveling the diagonal transformations (cf. \Cref{def:killed-committor},
\Cref{eq:Lea}, and \Cref{eq:marked-projections})
used to define $C_{\gamma}$, $\mathring{L}_{\gamma}$, and $W_{\gamma}$
in terms of $\tilde{C}_{\gamma}$, $\mathring{R}_\gamma$, and $\tilde{W}$,
we see that it suffices to show that 
\begin{equation}
\tilde{W}^{\top}\mathrm{diag}(\mathring{\pi}_{\gamma})\mathring{R}_{\gamma}\tilde{W}=\tilde{C}_{\gamma}^{\top}\mathrm{diag}(\pi)R_{\gamma}\tilde{C}_{\gamma},\label{eq:LHSRHSgamma}
\end{equation}
similar to \Cref{eq:LHSRHS}. The proof of \Cref{eq:LHSRHSgamma}
is the same as that of \Cref{eq:LHSRHS}, given the
ingredients: 
\begin{enumerate}
\item $\mathrm{diag}(\pi)R_{\gamma}$ is symmetric, following directly from
the same fact for $\mathrm{diag}(\pi)R$, since $R_\gamma = R - \gamma \Id$.
\item The entrywise formula $\left(\mathrm{diag}(\mathring{\pi}_{\gamma})R_{\gamma}\right)_{(i,j),(i',j')}$
is the same as \Cref{eq:piRcon}, upon substitution
of $\tilde{C}_{\gamma}$ for $\tilde{C}$ (based on the entrywise
expression \Cref{eq:pealt2} for $\mathring{\pi}_{\gamma}$) and $R_{\gamma}=R-\gamma\Id$
for $R$ (based on the fact that this reproduces the effect of subtracting
$\gamma\Id$ from $\mathring{R}$ in the blockwise formula \Cref{eq:expanded-R} for
$\mathring{R}$).
\item The ``killed Laplace equation'' is satisfied by the killed commitor on $\Ic$, i.e., $[R_{\gamma}\tilde{C}_{\gamma}]_{\Ic,:}=0$. We verify this last point below.
\end{enumerate}
Then it only remains to show $[R_{\gamma}\tilde{C}_{\gamma}]_{\Ic,:}=0$.
To see this, recall the formula (\Cref{def:killed-committor})
\[
\tilde{C}_{\gamma}= \mathrm{diag}^{-1}(h)(K_{\gamma})_{:,\I} \ \cdots,
\]
 where ``$\,\cdots\,$'' indicates an omitted matrix factor. 
 Then since $K_{\gamma}=-\mathrm{diag}(h)R_{\gamma}^{-1}\mathrm{diag}^{-1}(h)$,
it follows that 
\[
[R_{\gamma}\tilde{C}_{\gamma}]_{\Ic,:}=-(R_{\gamma})_{\Ic,:}(R_{\gamma}^{-1})_{:,\I} \ \cdots
\]
But $(R_{\gamma})_{\Ic,:}(R_{\gamma}^{-1})_{:,\I}=\Id_{\Ic,\I}=0$,
which completes the proof.
\end{proof}


Now we are in position to write the probability evolution of $\FullChain$ and $\InducedChain$ in terms of $\MarkedChain$:
\ThInducedFromMarked*
\begin{Proof}{th:induced-from-marked}
    The first statement is evident by considering that $\Qo$ has orthonormal columns (\Cref{th:q-contraction}) and defines an invariant subspace for $\Le$ with $\Qo^\t \Le \Qo = L$ (\Cref{th:qlq}).
    
    For the second statement, $\So$ has orthonormal columns, so $\Id - \So \So^\t$ defines an orthogonal projection. Using the block structure defined by $\So$ and its complement, \Cref{lem:expalpha} directly implies that for any $t>0$:
    \begin{eqn}
        \lim_{\alpha \ra \infty} \Expt{[\Le + \alpha (\Id - \So \So^\t)]} = \So \Expt{\So^\t \Le \So} \So^\t.
    \end{eqn}
    Next, $\Qo^\t \So = \Co$ by \Cref{th:q-s-equals-c}, and $\So^\t \Le \So = \Co^\t L \Co$ by \Cref{th:sls}, yielding the desired result.
\end{Proof}

\subsection{Contour integral derivations \label{a:sp-contour}}

Equipped with the preceding results, we now prove how Cauchy's integral formula can be used with the marked chain projections to rewrite the structure-preserving approximation:

\ThObliqueCauchy*
\begin{Proof}{th:oblique-cauchy}
Fix $\gamma, t>0$ for the duration of the proof. Expand 
\begin{eqn} \label{eq:unpack}
    \ObliqueP_{\I,\Kill}(t) = \Ca e^{- \Lrk t} \Ca^\t = Q_\gamma^\t \left( W_\gamma \, e^{- \Lrk t} \,  W_\gamma^\t \right) Q_\gamma,
\end{eqn}
using the fact that $\Ca =  Q_\gamma^\t W_\gamma$ (\Cref{th:q-s-equals-c}). Since $W_\gamma$ has orthonormal columns (\Cref{th:s-orthogonality}), we can directly apply \Cref{lem:contour0} with $A = \Lrk t$ and $U = W_\gamma$ to obtain 
\begin{eqn}
    W_\gamma \, e^{- \Lrk t} \, W_\gamma^\t = \frac{1}{2\pi i} \oint_{\Contour} e^{\Minus 1/z} \lrp{z \Id -  \Sa (t \Lrk)^{-1} \Sa^\t }^\Mo \odif z.
\end{eqn}
Together with \Cref{eq:unpack} this implies:
\begin{eqn}
    \ObliqueP_{\I,\Kill}(t) = \frac{1}{2\pi i} \oint_{\Contour} e^{\Minus 1/z} \Qa^\t \lrp{z \Id -  \Sa (t \Lrk)^{-1} \Sa^\t }^\Mo \Qa \odif z.
\end{eqn}
    Examining the integrand: 
    \begin{eqn}
        \Qa^\t \lrp{z \Id - \Sa (t \Lrk)^\Mo \Sa^\t}^\Mo \Qa
        &= z^\Mo \Qa^\t \lrp{\Id - \Sa (\Id - z t \Lrk )^\Mo \Sa^\t} \Qa \\
        &= z^\Mo \lrp{\Qa^\t \Qa - \Ca (\Id - z t \Lrk )^\Mo \Ca^\t} \\
        &= \mathcal{T}^{\,\Kill}(z) + z^\Mo \lrp{\Id - \Ca (\Id - z t \Lrk )^\Mo \Ca^\t} \\
        &= \mathcal{T}^{\,\Kill}(z) + \lrp{z \Id - \Ca \left[ t \Lrk - z^\Mo (\Id - \Ca^\t \Ca) \right]^\Mo \Ca^\t}^\Mo \\
        &= \mathcal{T}^{\,\Kill}(z) + \ObliqueRes^\Kill(z),
    \end{eqn}
    where the first line follows from the Woodbury matrix identity, the second from \Cref{th:q-s-equals-c}, and the fourth from the Woodbury identity (most easily verified in reverse).
    The other lines follow from the definitions of $\mathcal{T}^{\,\Kill}(z)$ and $\ObliqueRes^\Kill(z)$ \Cref{eq:oblique-res-def}.
\end{Proof}


We now turn to bounding the difference in resolvents of the structure-preserving and projective compressions, $\ObliqueRes$ and $\OrthRes$.

\begin{lemma}\label{th:resolvent-sqrt}
    Let $B$ be a real symmetric positive definite matrix and $z \in \mathbb{C}$, $\re(z) > 0$. 
    Then
    \begin{eqn}
    \snorm*{\frac{(B^\Mo - z B)^\Mo}{z^{1/2}}} \leq 2^{\Mo/2} \lrp{1 - \frac{\re(z)}{\Abs{z}}}^\Mh \Abs{z}^\Mo. \\
    \end{eqn}
\end{lemma}
\begin{proof}
    Let $\lambda_i > 0$ denote the eigenvalues of $B$, which are in particular real. Now $B^\Mo - z B$ is a normal matrix, so it can be orthogonally diagonalized by the spectral theorem, and therefore 
    \begin{eqn}
        \norm{(B^\Mo - z B)^\Mo}_2 = \max_i \; \Abs{\lambda_i^\Mo - z \lambda_i}^\Mo .
    \end{eqn}
Let $z = a + b i$ for real $a,b$ with $a > 0$. Then 
\begin{eqn}
    \frac{\Abs{\lambda_i^\Mo - z \lambda_i}^\Mo}{\Abs{z}^{1/2}} = \lrp{a^2 + b^2}^{-1/4} \lrp{b^2 \lambda_i^2 + (\lambda_i^\Mo - a \lambda_i)^2}^\Mh \eqqcolon f(a,b,\lambda_i).
\end{eqn}
Consider $a,b$ fixed. Then $f(a,b,\lambda_i)$ is upper bounded as follows:
\begin{eqn}
    f(a,b,\lambda_i) \leq \max_{\lambda > 0} \{ f(a,b,\lambda) \}  = 2^\Mh \lrp{1 - \frac{a}{\sqrt{a^2 + b^2}}}^\Mh (a^2 + b^2)^\Mh
\end{eqn}
with the maximum achieved at $\lambda = (a^2 + b^2)^{-1/4} = \Abs{z}^{-1/2}$.
\end{proof}


\begin{lemma}\label{th:obliqueness-equivalence}
     Let $\La$ be a killed Laplacian with committor $\Ca$ as in \Cref{def:killed-committor}, $z \in \mathbb{C}$ with $\re(z) > 0$, 
    $\OrthRes^\Kill$ as in \cref{eq:orth-res-def}, and
    $\ObliqueRes^\Kill$ as in \Cref{th:oblique-cauchy}. Then:
\begin{eqn}
    \label{eq:oblique-resolvent-diff}
    \ObliqueRes^\Kill(z) - \OrthRes^\Kill(z) = \Ca \,
    \SqrtResolvent_z [\Ca^\t \La \Ca]
    \, \Phi_\Kill \,
  \SqrtResolvent_z [\Va^\t \La \Va] \,
 \Va^\t / t,
\end{eqn}
where  $\Va = \Ca (\Ca^\t \Ca)^\Mh$ (cf. \Cref{def:killed-committor}), and moreover we define
\begin{eqn}
    \SqrtResolvent_z[A] \Eq \frac{(A^\Mh - z A^{1/2})^\Mo}{z^{1/2}}
\end{eqn}
and 
\begin{eqn}
    \Phi_\Kill \Eq (\Ca^\t \La \Ca)^\Mh \big[ (\Ca^\t \Ca)^\Mh - (\Ca^\t \Ca)^{1/2} \big] (\Va^\t \La \Va)^\Mh.
\end{eqn}
\end{lemma}
\begin{proof}
    Applying the Woodbury matrix identity to both $\OrthRes^\Kill(z)$ and $\ObliqueRes^\Kill(z)$, we deduce that 
    \begin{eqn}
        \OrthRes^\Kill(z) &= z^\Mo \lrp{\Id - \Ca (\Ca^\t \Ca - z \Lrk t)^\Mo \Ca^\t} \\
        \ObliqueRes^\Kill(z) &= z^\Mo \lrp{\Id - \Ca (\Id - z \Lrk t)^\Mo \Ca^\t}
    \end{eqn}
    so that upon rearrangement,
    \begin{eqn}
        \ObliqueRes^\Kill(z) - \OrthRes^\Kill(z) \ \ = \ \ & z^\Mo \Ca \left[ (\Ca^\t \Ca - z \Lrk t)^\Mo - (\Id - z \Lrk t)^\Mo \right] \Ca^\t \\
        = \ \ & z^\Mo \Ca (\Ca^\t \Ca - z \Ca^\t \La \Ca t)^\Mo (\Id - \Ca^\t \Ca) (\Id - z \Ca^\t \La \Ca t)^\Mo \Ca^\t \\
        = \ \ & z^\Mo \Ca 
            \left[ (\Ca^\t \La \Ca t)^\Mh - z (\Ca^\t \La \Ca t)^{1/2} \right]^\Mo \\
            & \qquad (\Ca^\t \La \Ca)^\Mh
         \big[ (\Ca^\t \Ca)^\Mh - (\Ca^\t \Ca)^{1/2} \big]
         (\Va^\t \La \Va)^\Mh \\
         & \qquad  \left[ (\Va^\t \La \Va t)^\Mh - z (\Va^\t \La \Va t)^{1/2} \right]^\Mo 
         \Va^\t / t\\
         = \ \ & \Ca  \, 
            \SqrtResolvent_z [\Ca^\t \La \Ca] \\
            & \qquad (\Ca^\t \La \Ca)^\Mh
         \big[ (\Ca^\t \Ca)^\Mh - (\Ca^\t \Ca)^{1/2} \big] 
         (\Va^\t \La \Va)^\Mh \\
         & \qquad  \SqrtResolvent_z [\Va^\t \La \Va ] \, 
         \Va^\t / t.
         \label{eq:resolvent-difference-forms}
    \end{eqn}
\end{proof}


Finally, we are able to directly invoke $\Obs^\Kill$ in the bounding of $\ObliqueRes^\Kill(z) - \OrthRes^\Kill(z)$ via:
\begin{lemma}[Obliqueness norm identity] \label{th:obliqueness-norm}
    Consider a nonempty subset $\I$, $\La$ with $\Kill > 0$, $\Ca$ as in \Cref{th:killed-committor}, and $\Va = \Ca (\Ca^\t \Ca)^\Mh$ (as before).
    Then with $\Ob_\Kill$ and $\Obs^\Kill$ as in \Cref{def:obliqueness} and $\Phi_\Kill$ as in \Cref{th:obliqueness-equivalence}:
    \begin{eqn}
        \Norms{\Phi_\Kill} = \Obs^\Kill.
    \end{eqn}
\end{lemma}
\begin{proof}
    Straightforward computation yields:
    \begin{eqn}
        \Phi_\Kill \Phi_\Kill^\t
        & \ =  \ \ (\Ca^\t \La \Ca)^\Mh \big[ (\Ca^\t \Ca)^\Mh - (\Ca^\t \Ca)^{1/2} \big] (\Va^\t \La \Va)^\Mo
         \\ 
         & \qquad \qquad \big[ (\Ca^\t \Ca)^\Mh - (\Ca^\t \Ca)^{1/2} \big] (\Ca^\t \La \Ca)^\Mh \\
         & \ =  \ \   (\Ca^\t \La \Ca)^\Mh (\Id - \Ca^\t \Ca) (\Ca^\t \La \Ca)^\Mo
         (\Id - \Ca^\t \Ca) (\Ca^\t \La \Ca)^\Mh \\
         & \ =  \ \  \left[ (\Ca^\t \La \Ca)^\Mh (\Id - \Ca^\t \Ca) (\Ca^\t \La \Ca)^\Mh \right]^2 = \Ob_\Kill^2.
    \end{eqn}
    Since $\Ob_\Kill$ is symmetric positive semidefinite, it follows that 
    \begin{eqn}
        \snorm{\Phi_\Kill} &= \snorm{\Phi_\Kill \Phi_\Kill^\t}^{1/2} = \snorm{\Ob_\Kill} \\
        \nnorm{\Phi_\Kill} &= \nnorm{(\Phi_\Kill \Phi_\Kill^\t)^{1/2}} = \nnorm{\Ob_\Kill},
    \end{eqn}
    by expanding identities for the nuclear and spectral norms.
\end{proof}


Putting our bounds together and using \Cref{th:resolvent-sqrt} and \Cref{def:obliqueness}, we can bound the resolvent difference as follows:
\ThResolventObliqueness*
\begin{Proof}{th:resolvent-obliqueness}
    Applying H\"{o}lder's inequality for the Schatten $p$-norms repeatedly to \cref{eq:oblique-resolvent-diff} yields:
    \begin{eqn}
        \Norms{\ObliqueRes^\Kill(z) - \OrthRes^\Kill(z)} & \ \leq  \ 
        \snorm{\Ca}
        \cdot \snorm{\SqrtResolvent_z[\Ca^\t \La \Ca]}  \cdot 
        \Norms{\Phi_\Kill}  \cdot 
        \snorm{\SqrtResolvent_z[\Va^\t \La \Va]}
        \cdot \snorm{\Va} / t
    \end{eqn}
    following the notation of \Cref{th:obliqueness-equivalence}. 
    
    We deal with the factors in their order of appearance as follows: 
    \begin{eqn}
        \snorm{\Ca} &\leq 1 
            & \hspace{-0.1in}\text{by \Cref{th:committor-norm}} \\
        \snorm{\SqrtResolvent_z [\Ca^\t \La \Ca]} &\leq 2^\Mh \lrp{1 - \re(z) / \Abs{z}}^\Mh \Abs{z}^\Mo 
            & \hspace{-0.1in}\text{by \Cref{th:resolvent-sqrt}} \\
        \Norms{\Phi_\Kill}  &= \Obs^\Kill
            & \hspace{-0.1in}\text{by \Cref{th:obliqueness-norm}} \\
        \snorm{\SqrtResolvent_z [\Va^\t \La \Va]} &\leq 2^\Mh \lrp{1 - \re(z) / \Abs{z}}^\Mh \Abs{z}^\Mo  
            & \hspace{-0.1in}\text{by \Cref{th:resolvent-sqrt}} \\
        \snorm{\Va} &= 1 
            & \hspace{-0.1in}\text{by \Cref{th:committor-norm}.}
    \end{eqn} 
    Then \cref{eq:resolvent-diff-bound} follows directly.
\end{Proof}



\begin{figure}
    \centering
    \includegraphics[width=0.9\textwidth]{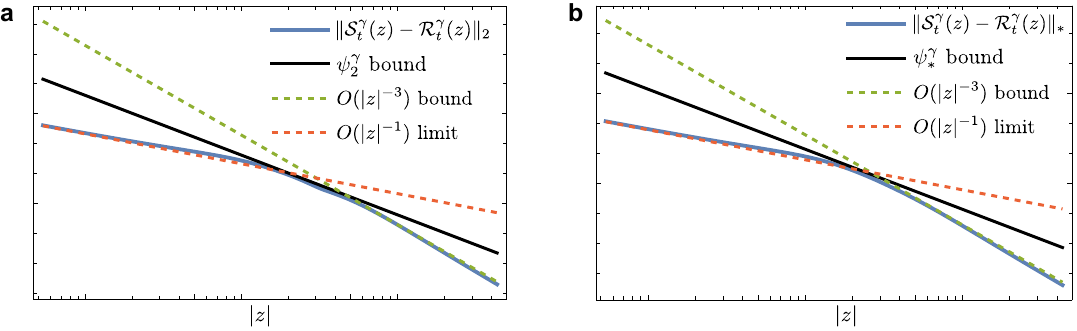}
    \caption{
        Resolvent error along ray for a randomly generated killed Laplacian $\La$ and fixed $t = 1$.
        (a) Log-log plot of $\snorm{\ObliqueRes^\Kill(z) - \OrthRes^\Kill(z)}$ as a function of $\Abs{z}$ for $z \propto \tfrac{1}{2} + \tfrac{\sqrt 3}{2} i$, exhibiting a transition from an $O(|z|)^\Mo$ regime (small $|z|$) 
        to an $O(|z|)^{\Minus 3}$ regime (large $|z|$). 
        \Cref{th:resolvent-obliqueness} gives a bound of $\large(1 - \frac{\re(z)}{\Abs{z}} \large)^\Mo \frac{\Os^\Kill}{2 t \Abs{z}^2}$ (black) over the entire ray.
        (b) Same as (a) with the nuclear norm replacing the spectral norm throughout.
        \label{fig:oblique-contour-schematic}
    }
\end{figure}

Similar to the proof of \Cref{th:dissipative}, ultimately the key difficulty in bounding the error of the structure-preserving compression is to bound $\Norms{\ObliqueRes^\Kill(z) - \OrthRes^\Kill(z)}$ on the diagonal contour segments comprising $\mathcal{C}_{1}(T,0)$ defined in \Cref{eq:contourdef0}. 
Along such a segment, \Cref{fig:oblique-contour-schematic} depicts $\Norms{\ObliqueRes^\Kill(z) - \OrthRes^\Kill(z)}$ in a typical scenario.
Close to the origin, the norm is $O( \Abs{z}^\Mo )$, while in the large $\Abs{z}$ limit, it is possible to prove an $O(\Abs{z}^{\Minus 3})$ bound. Since it is not useful to us in our analysis, we will omit the proof.
Indeed, it is apparently not possible to make use of sharp asymptotics in the small-$\vert z \vert$ or large-$\vert z\vert$ regime in terms of quantities that can be made small via column selection like the \Nystrom{} error.
Instead we use an $O(\Abs{z}^{\Minus 2})$ bound which holds globally over the contour segment with a preconstant controlled by $\Obs^\gamma$ via \Cref{th:resolvent-obliqueness}.

\ThObliqueSumBound*
\begin{Proof}{th:oblique-sum-bound}
     First fix $t>0$. In this proof we will let $c > 0$ denote a suitably large constant, which may change or enlarge based on context. Importantly $c$ must be independent of $\gamma$ as well as any contour integral argument $z$, but it can otherwise depend on $t$, $L$ (and in particular the dimension), and $\I$. Likewise, we will use big-$O$ notation in the same sense (i.e., in which the preconstants are independent of $\gamma$ and $z$ and moreover the norm is arbitrary).
     These soft arguments are only used to perform contour manipulations and do not affect the quantitative error bounds.
     
     Then by \Cref{th:oblique-cauchy} and \Cref{eq:orthogonal-cauchy}, we can write 
    \begin{eqn}\label{eq:bigsum}
        \ObliqueP_{\I,\Kill}(t) - \OrthP_{\I,\Kill}(t)  
        = \underbrace{ \frac{1}{2\pi i} \oint_{\mathcal{C}} e^{\Minus 1/z} [\mathcal{S}_t^\gamma (z) - \mathcal{R}_t^\gamma (z)] \, \odif z}_{ =: \, I_1 }
        + \underbrace{\frac{1}{2\pi i} \oint_{\mathcal{C}} e^{\Minus 1/z} \, \mathcal{T}^{\,\gamma} (z) \, \odif z }_{ =: \, I_2 }
    \end{eqn}
    for any contour $\mathcal{C}$ in the strict right half-plane enclosing the eigenvalues of both $(t \Lrk)^{-1}$ and $(t L_\gamma)^{-1}$. Note that as $\gamma \ra 0$, these eigenvalues extend further away from the origin, and we will need a bound on their growth. First, since $L_\gamma \succeq \gamma$, we know that the eigenvalues of $(t L_\gamma)^{-1}$ are bounded by $(t \gamma)^{-1}$. However, $\Lrk \neq \Lr + \gamma \Id$, and we need a more delicate argument to get a lower bound on its eigenvalues. Indeed, $\Lrk = C_\gamma^\t L_\gamma C_\gamma = \Lr + \gamma C_\gamma^\t C_\gamma \succeq \gamma C_\gamma^\t C_\gamma$. Now by the arguments of \Cref{th:killed-committor}, we know that $C_\gamma = C + O(\gamma)$. Moreover, $C$ has full rank because $\tilde{C}_{\I,\I} = \Id$ (by its probabilistic definition), so $\lambda_{\min} (C^\t C) > 0$. Therefore $\Lrk \succeq \gamma \lambda_{\min} (C^\t C) + O(\gamma^2)$. It follows that all of the eigenvalues of $(t \Lrk)^{-1}$ and $(t L_\gamma)^{-1}$ are $O(\gamma^{-1})$.
    
    Then we can choose the same contour $\mathcal{C} = \mathcal{C}(T,\delta)$ as defined in \Cref{eq:contourdef0} from the proof of \Cref{th:dissipative}, where $\delta > 0$ sufficiently small and $T = \gamma^{-2}$. By the preceding arguments, this contour is valid for $\gamma >0$ sufficiently small. We will choose the constants $\alpha,\beta > 0$ in \Cref{eq:contourdef0} in order to optimize our bound, and ultimately these constants will be different than in the proof of \Cref{th:dissipative}. (Alternatively, we can view the values $\alpha = 1/2$ and $\beta = \sqrt{3}/2$ that we uncover as being fixed from the start.) By the same arguments as in the proof of \Cref{th:dissipative}, we can take the limit $\delta \ra 0^+$ to adopt the contour $\mathcal{C} = \mathcal{C}(T,0)$, cf. \Cref{fig:contour}. Then let us view $I_1 = I_1 (\gamma)$ and $I_2 = I_2 (\gamma)$ defined in \Cref{eq:bigsum} as depending on $\gamma$ only, where the $\gamma$-dependent contour is fixed by our choice $T = T(\gamma) = \gamma^{-2}$.

    Now by the expression \Cref{eq:oblique-res-def} for $\mathcal{T}^{\,\gamma} (z)$, we know that $\Norms{ e^{\Minus 1/z} \, \mathcal{T}^{\,\gamma} (z) } \leq c ( 1+\vert z \vert )^{-1} \Norms{ Q_\gamma^\t Q_\gamma - \Id }$ on $\mathcal{C} (T,0)$, so by integration we deduce that 
    \begin{eqn}\label{eq:I2bound}
        \Norms{I_2 (\gamma)} \leq c \log(T) \Norms{ Q_\gamma^\t Q_\gamma - \Id }.
    \end{eqn}
    By inspection of \Cref{def:killed-projections} for $Q_\gamma$, we have that $Q_\gamma = Q + O(\gamma)$. Since $Q$ has orthonormal columns, it follows that $\Norms{ Q_\gamma^\t Q_\gamma - \Id } = O(\gamma)$, and therefore 
    \begin{eqn}\label{eq:I2bound2}
        \Norms{I_2 (\gamma)} = O(\gamma \log(1/\gamma)),
    \end{eqn}
    so in particular $ \Norms{I_2 (\gamma)} \ra 0 $ as $\gamma \ra 0$.

    Thus we can ignore $I_2$ asymptotically, and we turn our attention to $I_1$. First we focus on the contribution due to the vertical crossbar $\mathcal{C}_2 (T)$ as defined in the proof of \Cref{th:dissipative}, cf. \Cref{eq:contourdef0}.
    Just as in that proof (cf. \Cref {eq:verticalcontourAB}), we divide the vertical crossbar into two pieces, near and far from the imaginary axis: 
    \begin{eqn}
    \mathcal{C}_{2A}(T):=\{T+is\,:\,\vert s\vert\leq \sqrt{T}\},\quad\mathcal{C}_{2B}(T):=\{T+is\,:\,\vert s\vert\geq\sqrt{T}\}.
    \end{eqn}
    Now the poles of $\ObliqueRes^\Kill(z)$ and $\OrthRes^\Kill(z)$ lie in $[0,c \sqrt{T}]$, so  $\snorm{\ObliqueRes^\Kill(z)},\snorm{\OrthRes^\Kill(z)} = O(T^\Mo)$ on $\Contour_{2A}$. Since the $\Contour_{2A}$ contour is of length $O( \sqrt{T} )$, it follows that the  $\Contour_{2A}$ contribution is $O(1/\sqrt{T}) = O(\gamma)$, which vanishes as $\gamma \ra 0$.

    We will use \Cref{th:resolvent-obliqueness} to manage the $\mathcal{C}_{2B}$ contribution. First note that for $z \in \mathcal{C}_{2B} (T)$, we can write $z = T + i \, \im(z)$ and in turn compute 
    \begin{eqn}
    \left(1-\frac{\re(z)}{\vert z\vert}\right)^{-1}=\frac{1+(\im(z)/T)^{2}+\sqrt{1+(\im(z)/T)^{2}}}{(\im(z)/T)^{2}}.
    \end{eqn}
    We also have $\im(z)/T \leq \beta/\alpha$ for $z \in \mathcal{C}_{2B} (T)$, so we deduce 
    \begin{eqn}
        \left(1-\frac{\re(z)}{\vert z\vert}\right)^{-1} = O(T^2 / \vert \im (z) \vert^2).
    \end{eqn}
    Moreover $1/\vert z\vert^2 = O(T^{-2})$ for $z \in \mathcal{C}_{2B} (T)$, and by \Cref{th:recurrent-obliqueness} we know that $\Obs^\Kill$ are convergent as $\gamma \ra 0$. Therefore from \Cref{th:resolvent-obliqueness}, it follows that $\Norms{\ObliqueRes^\Kill(z) - \OrthRes^\Kill(z)} = O( \vert \im (z) \vert^{-2})$ on $\mathcal{C}_2 (T) $. Then by the same reasoning \Cref{eq:C2Bint} as in the proof of \Cref{th:dissipative}, the $\mathcal{C}_{2B}$ contribution is $O(1/\sqrt{T})$ and in particular vanishes in our limit $\gamma \ra 0$.
    

    Finally, consider the contribution from $\mathcal{C}_1 (T)$ as defined in the proof of \Cref{th:dissipative}, cf. \Cref{eq:contourdef0}. It is useful to let $\phi := \alpha + i \beta$ be the complex unit defining the angle of the segments forming $\mathcal{C}_1 (T)$, where we assume $\alpha^2 + \beta^2 = 1$ without loss of generality.

    By \Cref{th:resolvent-obliqueness}, for any $z \in \Contour_{1} (T)$ we have 
        \begin{eqn}
            \Norms{\ObliqueRes^\Kill(z) - \OrthRes^\Kill(z)} \leq 
            \frac{1}{2}\lrp{1 - \re(\phi)}^\Mo \frac{\Obs^\Kill}{t \cdot \Abs{z}^2},
        \end{eqn}
    and we can  compute: 
    \begin{eqn}
        \int_{\Contour_{1}(T)} \Abs{z}^{\Minus 2} \Abs{e^{\Mo/z}} \Abs{\odif z} 
        \leq  2 \int_0^\infty x^{\Minus 2} \Abs{e^{\Mo/ (\phi x)}} \odif x
        = \frac{2}{\re(\phi)},
    \end{eqn}
    so we deduce that
    \begin{eqn}
        \int_{\Contour_{1}(T)} e^{-1/z} \Norms{\ObliqueRes^\Kill(z) - \OrthRes^\Kill(z)} \, \Abs{\odif z}
        \leq 
        \left[ \re(\phi) (1-\re(\phi) ) \right]^{-1}  \frac{\Obs^\Kill}{t}.
    \end{eqn}
    To fix our contour, we want to determine $\phi$ on the unit circle in the first quadrant of the complex plane that maximizes $\re(\phi) (1-\re(\phi))$. Since $\re(\phi)$ can attain any value in $[0,1]$, it is equivalent to maximize $x(1-x)$ over $x \in [0,1]$. The optimizer is $\re(\phi) = x = 1/2$, which yields $\phi = \frac{1}{2} + \frac{\sqrt{3}}{2} i$. Fixing this choice of $\phi$, and in turn our contour, we obtain at last 
    \begin{eqn}
        \int_{\Contour_{1}(T)} e^{-1/z} \Norms{\ObliqueRes^\Kill(z) - \OrthRes^\Kill(z)} \, \Abs{\odif z}
        \leq 
          4 \frac{\Obs^\Kill}{t}.
    \end{eqn}
    
    Then by taking norms on both sides of \eqref{eq:bigsum} and taking the limit as $\gamma \ra 0$ where we recall the specification of our $\gamma$-dependent contour, we deduce the desired bound \Cref{eq:oblique-bound}. 
    
    
    Finally, \cref{eq:composed-bound} follows from the triangle inequality, using Theorem~\ref{th:ortho-markov} together with \cref{eq:oblique-bound}. This completes the proof.
\end{Proof}

\subsection{Probabilistic derivations related to obliqueness}


We begin with the following general lemma which allows us to write the obliqueness norms $\Obs^\Kill$ in terms of $\Ka$.
\begin{lemma}
    \label{th:norm-switch}
    Let $A$ and $B$ be real symmetric positive definite matrices and $C$ be real. Then
\begin{eqn}
    \Norms{C B (B A B)^\Mh} = \Norms{C A^\Mh}.
\end{eqn}
\end{lemma}
\begin{proof}
    For any real $M$, $\snorm{M} = \snorm{M M^\t}^{1/2}$ and $\nnorm{M} = \nnorm{(M M^\t)^{1/2}}$.
    With $M = C B (B A B)^\Mh$,
\begin{eqn}
    M M^\t = [C B (B A B)^\Mh] [C B (B A B)^\Mh]^\t = C B (B A B)^\Mo B C^\t = C A^\Mo C^\t
\end{eqn}
so that $\snorm{M M^\t}^{1/2} = \snorm{C A^\Mo C^\t}^{1/2}  = \snorm{C A^\Mh}$ and $\nnorm{(M M^\t)^{1/2}} = \nnorm{(C A^\Mo C^\t)^{1/2}} = \nnorm{C A^\Mh}$.
\end{proof}


Using this lemma and the definition of $\Ca$ (\Cref{def:killed-committor}), we obtain the following:
\begin{lemma}[Obliqueness in terms of $\Ka$] \label{th:obliqueness-using-k}
    For the killed chain $\La$, $\Ka = \La^\Mo$, $\pa$ as in \Cref{def:killed-committor}, and $\Obs^\Kill$ as in \Cref{def:obliqueness}:
    \begin{eqn}
        \label{eq:obliqueness-using-k}
        \Obs^\Kill = \Norms*{\KA_{\I,\I}^\Mh \KA_{\I,:} \left[ \Id_{:,\I} \Diag \lrp{\frac{\pa}{\pi_{\I}}} \Id_{\I,:} - \Id \right] \KA_{:,\I} \KA_{\I,\I}^\Mh}.
    \end{eqn}
\end{lemma}
\begin{proof}
From direct use of \Cref{def:killed-committor}, we compute that:
\begin{eqn}
    \Ca^\t \La \Ca &= \Diag(\hi / \ha) [\KA_{\I,\I}]^\Mo \Diag(\hi / \ha), \\
    \Ca^\t \Ca &= \Diag(\hi / \ha) \KA_{\I,\I}^\Mo (\Ka^2)_{\I,\I} \KA_{\I,\I}^\Mo \Diag(\hi / \ha).
\end{eqn}
Then by substitution, with $\Ob_\Kill$ as defined as in \Cref{def:obliqueness},
\begin{eqn}
    \Ob_\Kill \ &= \ (\Ca^\t \La \Ca)^\Mh (\Id - \Ca^\t \Ca) (\Ca^\t \La \Ca)^\Mh \\
    &= \ \left[ \DiagRat{\hi}{\ha}  \KA_{\I,\I}^\Mo \DiagRat{\hi}{\ha} \right] ^\Mh \DiagRat{\hi}{\ha} \\ 
    &\hspace{3em} \KA_{\I,\I}^\Mo  \left[ \KA_{\I,\I} \DiagRat{\ha^2}{\hi^2} \KA_{\I,\I} - (\Ka^2)_{\I,\I} \right] \KA_{\I,\I}^\Mo  \\
 &\hspace{3em} \DiagRat{\hi}{\ha} \left[ \DiagRat{\hi}{\ha} \KA_{\I,\I}^\Mo \DiagRat{\hi}{\ha} \right]^\Mh.
\end{eqn}
Finally, by application of \Cref{th:norm-switch} once on the left and once on the right (with $A=\KA_{\I,\I}^\Mo$, $B=\DiagRat{\hi}{\ha}$, $C$ the rest), we compute that:
\begin{eqn}
    \Norms{\Ob_\Kill} &= \Big\Vert \left[  \KA_{\I,\I}^\Mo \right] ^\Mh \KA_{\I,\I}^\Mo  \left[ \KA_{\I,\I} \DiagRat{\ha^2}{\hi^2} \KA_{\I,\I} - (\Ka^2)_{\I,\I} \right] \\ 
    & \hspace{3em} \KA_{\I,\I}^\Mo \DiagRat{\hi}{\ha} \left[ \DiagRat{\hi}{\ha} \KA_{\I,\I}^\Mo \DiagRat{\hi}{\ha} \right]^\Mh \Big\Vert_{\lrb{2,*}} \\
    &= \Big\Vert \left[  \KA_{\I,\I}^\Mo \right] ^\Mh \KA_{\I,\I}^\Mo  \left[ \KA_{\I,\I} \DiagRat{\ha^2}{\hi^2} \KA_{\I,\I} - (\Ka^2)_{\I,\I} \right] \KA_{\I,\I}^\Mo \left[ \KA_{\I,\I}^\Mo \right]^\Mh \Big\Vert_{\lrb{2,*}}
\end{eqn}
and therefore upon simplification:
\begin{eqn}
    \Obs^\Kill = \Norms*{\KA_{\I,\I}^\Mh  \lrp{\KA_{\I,\I} \Diag\lrp{\tfrac{\ha^2}{\hi^2}} \KA_{\I,\I} - (\Ka^2)_{\I,\I}} \KA_{\I,\I}^\Mh}.
\end{eqn}
Then \cref{eq:obliqueness-using-k} follows by rearrangement and the facts that $\pi=h^2$, $\pa=\ha^2$.
\end{proof}


Next we demonstrate how the obliqueness norms converge and may be computed in the $\Kill \ra 0$ in closed form.
\begin{proposition}[Obliqueness in the $\Kill \ra 0$ limit] \label{th:recurrent-obliqueness}
    With $\Obs^\Kill$ as defined in \Cref{def:obliqueness}, $\On^\Kill$ converges in the $\Kill \ra 0$ limit as:
    \begin{eqn}  \label{eq:obliqueness-nuclear-limit}
        \lim_{\Kill \ra 0} \On^\Kill &= \TimeScale \Tr\left[K_{\I,\I} \Diag\lrp{\tfrac{\Ki \hi}{\hi}}\right] 
        +\TimeScale \hi^\t \Ki (K^2)_{\I,\I} \Ki \hi
        -  \Tr[(K^2)_{\I,\I} \Ki] 
    \end{eqn}
    and $\Os^\Kill$ converges as
    \begin{eqn} \label{eq:obliqueness-spectral-limit}
        \lim_{\Kill \ra 0} \Os^\Kill &= \rho\lrp{\TimeScale \lrp{ K_{\I,\I} \Diag\lrp{\tfrac{\Ki \hi}{\hi}} + (K^2)_{\I,\I} \Ki \hi \hi^\t \Ki } - (K^2)_{\I,\I} \Ki},
    \end{eqn}
    where $\TimeScale = 1 / \hi^\t \Ki \hi$ as in \Cref{eq:omega} and $\rho(\, \cdot \, )$ is the spectral radius.
\end{proposition}
\begin{proof}
Define the following (asymmetric) obliqueness matrix:
\begin{eqn}
    \ObAsym_\Kill &\Eq \left[ \KA_{\I,\I} \Diag \lrp{\tfrac{\pr}{\pi_{\I}}} \KA_{\I,\I} - (\Ka^2)_{\I,\I} \right] \KA_{\I,\I}^\Mo
    \label{eq:obliqueness-asym}
\end{eqn}
Since $\ObAsym_\Kill$ is a similarity transform of $\Ob_\Kill$, it follows that
 $
    \On^\Kill 
    = \Tr[\ObAsym_\Kill]
$ and $\Os^\Kill = \rho(\ObAsym_\Kill)$.
Therefore we proceed to computing $\ObAsym \Eq \lim_{\Kill \ra 0} \ObAsym_\Kill$.
If this limit exists, then the limits of $\On^\Kill$ and $\Os^\Kill$ do as well, since the trace and spectral radius are continuous functions.



As in the proof of \Cref{th:markov-norm-limits}, it is useful to define a matrix $\tilde{K}_\gamma := K_\gamma - \gamma^{-1} hh^\t$, suppress the $\gamma$-dependence from the notation $\tilde{K} = \tilde{K}_\gamma$, and view $K_\gamma = \tilde{K} + \gamma^{-1} hh^\t$ as a rank-one update. 
Then using the Sherman-Morrison formula, the right-hand side of \Cref{eq:obliqueness-asym} may be simplified to show that:
\begin{eqn}
    \ObAsym \Eq \lim_{\Kill \ra 0} \ObAsym_\Kill &= \TimeScale \lrp{ K_{\I,\I} \Diag\lrp{\tfrac{\Ki \hi}{\hi}} + (K^2)_{\I,\I} \Ki \hi \hi^\t \Ki } - (K^2)_{\I,\I} \Ki.
    \label{eq:obliqueness-asym-limit}
\end{eqn}
Taking the spectral radius of $\ObAsym$ yields \Cref{eq:obliqueness-spectral-limit} immediately, while taking its trace and rearranging yields \cref{eq:obliqueness-nuclear-limit}.
\end{proof}


\ThObliquenessFromHitting*
\begin{Proof}{th:obliqueness-from-hitting}
    First, let $S \Eq \DiagI(h) K \DiagI(h)$. 
    Then $S$ is symmetric with $S \pi = \Zero$.
    From Lemma~2.12 of \cite{Aldous1995-reversible}, $H_{j,i} = S_{i,i} - S_{j,i}$ for all $i,j$, so we may write:
    \begin{eqn}\label{eq:hitting-time}
        H = \One \Diag(S)^\t - S.
    \end{eqn}
    Then we compute the diagonal entries of $H_{:,\I}^\t \Diag(\pi) \Cu$.
    For $i \in \I$,
    \begin{eqn}
        H_{:,i}^\t \Diag(\pi) \Cu_{:,i} &= (S_{i,i} \One - S_{:,i})^\t \Diag(\pi) \DiagI(h) \lrp{\frac{(h - K_{:,\I} \Ki \hi^\t) \hi^\t}{\hi^\t \Ki \hi} + K_{:,\I}} (\Ki)_{:,i} \, h_i  \\
        &= (K_{i,i} h_i^\Mo h^\t - K_{i,:}) \lrp{\frac{(h - K_{:,\I} \Ki \hi^\t) \hi^\t}{\hi^\t \Ki \hi} + K_{:,\I}} (\Ki)_{:,i} \\
        &=  \lrp{K_{i,i} h_i^\Mo + (K^2)_{i,\I} \Ki \hi} \frac{(\Ki)_{i,:} \hi}{\hi^\t \Ki \hi} - (K^2)_{i,\I} (\Ki)_{:,i} \\
        &= \omega \lrp{K_{i,i} (\Ki)_{i,:} \hi h_i^\Mo  + (K^2)_{i,\I} \Ki \hi \hi^\t (\Ki)_{:,i}}   - (K^2)_{i,\I} (\Ki)_{:,i}
    \end{eqn}
    by use of the symmetry of $K$, the fact that $K h = \Zero$, and the definition \Cref{eq:omega} of $\TimeScale$.
    Summing and comparing to \Cref{th:recurrent-obliqueness} yields the desired result.
\end{Proof}



\section{Matrix perturbation lemma}\label{s:matpert}

In this section we prove an elementary lemma pertaining to the perturbation of a block of a matrix with a shift by a large multiple of the identity. Although the result is intuitive, we are not aware of a standard reference.

\begin{lemma}[Exponential of block perturbation] \label{lem:expalpha}
For $\alpha\geq0$, consider the matrix 
\[
A_{\alpha}=\left(\begin{array}{cc}
A_{11} & A_{12}\\
A_{21} & A_{22}+\alpha\Id
\end{array}\right).
\]
 Then 
\[
\lim_{\alpha\ra\infty}e^{-A_{\alpha}}=\left(\begin{array}{cc}
e^{-A_{11}} & 0\\
0 & 0
\end{array}\right).
\]
 
\end{lemma}

\begin{proof}
Let $z_{0}=\left(\begin{array}{c}
x_{0}\\
y_{0}
\end{array}\right)$ be arbitrary. It suffices to show that 
\[
\lim_{\alpha\ra\infty}e^{-A_{\alpha}}z_{0}=\left(\begin{array}{c}
e^{-A_{11}}x_{0}\\
0
\end{array}\right).
\]
 We can express 
\[
e^{-A_{\alpha}}z_{0}=z_{\alpha}(1),
\]
 where $z_{\alpha}(t)=\left(\begin{array}{c}
x_{\alpha}(t)\\
y_{\alpha}(t)
\end{array}\right)$ is the solution of the ODE: 
\[
z_{\alpha}'(t)=-A_{\alpha}z_{\alpha}(t),\qquad z(0)=z_{0}.
\]

First we want to show an \emph{a priori} bound of the form $\Vert z_{\alpha}(t)\Vert\leq C$
that is uniform in $\alpha\geq0$, $t\in[0,1]$. (In the following
discussion, for simplicity, we use $\Vert\,\cdot\,\Vert$ to denote
the vector 2-norm and the induced spectral norm.) To achieve this
\emph{a priori} bound, we write ODEs for $\Vert x_{\alpha}(t)\Vert^{2}$
and $\Vert y_{\alpha}(t)\Vert^{2}$ and bound their right-hand sides:
\begin{align*}
\frac{d}{dt}\Vert x_{\alpha}(t)\Vert^{2} & \ =\ -2x_{\alpha}(t)\cdot[A_{11}x_{\alpha}(t)+A_{12}y_{\alpha}(t)]\\
 & \ \leq\ 2\Vert A_{11}\Vert\,\Vert x_{\alpha}(t)\Vert^{2}+2\Vert A_{12}\Vert\,\Vert x_{\alpha}(t)\Vert\,\Vert y_{\alpha}(t)\Vert\\
 & \ \leq\ 2\Vert A_{11}\Vert\,\Vert x_{\alpha}(t)\Vert^{2}+\Vert A_{12}\Vert\,(\Vert x_{\alpha}(t)\Vert^{2}+\Vert y_{\alpha}(t)\Vert^{2})\\
 & \ =\ 2\Vert A_{11}\Vert\,\Vert x_{\alpha}(t)\Vert^{2}+\Vert A_{12}\Vert\,\Vert z_{\alpha}(t)\Vert^{2},
\end{align*}
 and similarly 
\begin{align*}
\frac{d}{dt}\Vert y_{\alpha}(t)\Vert^{2} & \ =\ -2\alpha\Vert y_{\alpha}(t)\Vert^{2}-2y_{\alpha}(t)\cdot[A_{21}x_{\alpha}(t)+A_{22}y_{\alpha}(t)]\\
 & \ \leq\ -2y_{\alpha}(t)\cdot[A_{21}x_{\alpha}(t)+A_{22}y_{\alpha}(t)]\\
 & \ \leq\ 2\Vert A_{22}\Vert\,\Vert y_{\alpha}(t)\Vert^{2}+\Vert A_{21}\Vert\,\Vert z_{\alpha}(t)\Vert^{2}.
\end{align*}
 Summing the bounds yields 
\[
\frac{d}{dt}\Vert z_{\alpha}(t)\Vert^{2}\leq C'\,\Vert z_{\alpha}(t)\Vert^{2}
\]
 for a suitably defined constant $C'>0$, independent of $\alpha,t$.
Then Gr\"{o}nwall's inequality yields the \emph{a priori} bound
$\Vert z_{\alpha}(t)\Vert\leq C$ for suitable $C>0$, as we aimed
to show.

Now examine the dynamics $y_{\alpha}(t)$ individually, 
\[
y_{\alpha}'(t)=-\alpha y_{\alpha}(t)+\underbrace{\left[-A_{21}x_{\alpha}(t)-A_{22}y_{\alpha}(t)\right]}_{=:\,b_{\alpha}(t)}.
\]
 We can view the bracketed term as an inhomogeneous contribution in
Duhamel's principle to deduce that 
\[
y_{\alpha}(t)=e^{-\alpha t}y_{0}+\int_{0}^{t}e^{-\alpha(t-s)}b_{\alpha}(s)\,\odif s.
\]
 Then taking norms we see that 
\[
\Vert y_{\alpha}(t)\Vert\leq e^{-\alpha t}\Vert y_{0}\Vert+C''\int_{0}^{t}e^{-\alpha(t-s)}\,\odif s
\]
for suitable $C''$, via our uniform bound on $\Vert z_{\alpha}(t)\Vert$,
which implies a uniform bound on $\Vert b_{\alpha}(t)\Vert$. We can
compute $\int_{0}^{t}e^{-\alpha(t-s)}\,\odif s=\alpha^{-1}[1-e^{-\alpha t}]\leq\alpha^{-1}$
to deduce that 
\begin{equation}
\Vert y_{\alpha}(t)\Vert\leq e^{-\alpha t}\Vert y_{0}\Vert+C''\alpha^{-1}\leq C'''(e^{-\alpha t}+\alpha^{-1})\label{eq:ybound}
\end{equation}
 for suitable $C'''$.

Then examine the dynamics for $x_{\alpha}(t)$ individually: 
\[
x_{\alpha}'(t)=-A_{11}x_{\alpha}(t)-A_{12}y_{\alpha}(t),
\]
 and again viewing the second term as an inhomogeneous contribution
using Duhamel's principle we deduce that 
\[
x_{\alpha}(t)=e^{-A_{11}t}x_{0}-\int_{0}^{t}e^{-(t-s)A_{11}}A_{12}y_{\alpha}(s)\,\odif s.
\]
 Now $\Vert e^{-\tau A_{11}}A_{12}\Vert$ is uniformly bounded over
$\tau\in[0,1]$, and recalling (\ref{eq:ybound}) it follows that
there exists $C''''>0$ such that 
\begin{align*}
\Vert x_{\alpha}(1)-e^{-A_{11}}x_{0}\Vert & \leq C''''\int_{0}^{1}(e^{-\alpha s}+\alpha^{-1})\,\odif s=O(\alpha^{-1}).
\end{align*}
 Hence $\lim_{\alpha\ra\infty}x_{\alpha}(1)=e^{-A_{11}}x_{0}$. Moreover,
(\ref{eq:ybound}) implies that $\lim_{\alpha\ra\infty}y_{\alpha}(1)=0$.
In other words, $\lim_{\alpha\ra\infty}z_{\alpha}(1)=\left(\begin{array}{c}
e^{-A_{11}}x_{0}\\
0
\end{array}\right)$, as was to be shown.
\end{proof}

\section{Proofs for \Cref{s:algs} \label{s:cssp-proofs}}

\newcommand{\DPPerror}{\mathcal{E}}
\newcommand{\DPPfunction}{\mathcal{E}^\mathrm{DPP}}
For an $n\times n$ symmetric  positive semidefinite matrix $M$, we define the expected nuclear approximation error of $k$-determinantal point process (DPP) sampling (see, e.g., \cite{Fornace2024-column}) as:
\begin{eqn}
    \label{eq:dpp-trace}
    \DPPfunction_k(M) \Eq \Ex_{\I \sim k\text{-DPP}(M)} \lrs{\Tr(M - M_{:,\I} M_{\I,\I}^\Mo M_{\I, :})}.
\end{eqn}

We will use the classical fact (\cite{Guruswami2012-optimal}, see also Proposition~5.1 of \cite{Fornace2024-column}) that 
    \begin{eqn}\label{eq:dppclassical}
        \mathcal{E}_{k}^{\mathrm{DPP}}(M)=(k+1)\frac{e_{k+1}(\sigma(M))}{e_{k}(\sigma(M))},
    \end{eqn}
where $\sigma(M) \in \R^n$ denotes the spectrum of $M$ and $e_{k}$ denotes the elementary symmetric polynomial of order $k$.


We will also use the celebrated bound on the DPP expectation \cite{Chen2022-randomly,Guruswami2012-optimal,Belabbas2009-spectral,Derezinski2021-determinantal}:
\begin{eqn} \label{eq:dpp-nystrom}
    \DPPfunction_k(M) \leq \lrp{1 + \frac{r}{k-r+1}} (\Tr[M] - \Tr^{(r)}[M]), 
\end{eqn}
in which $\Tr^{(r)}[M]$ denote the sum of the largest $r$ eigenvalues.


    



Then with a view toward proving Theorem \hyperlink{Theorem4B}{4B}, we prove the following lemma: 
\begin{lemma}\label{lem:dppstuff}
For any positive integer $k$, there exists a subset $\I$ of size $k$ such that
\begin{eqn}
    {\ve}_{*} (\I) \leq \left( 1 + \frac{r+1}{k-r} \right) (\Tr[K]- \Tr^{(k)}[K]).
\end{eqn}
\end{lemma}
\begin{proof}
   Consider performing $k$-DPP sampling on $\tilde{K}_\gamma := (L+\gamma hh^{\top})^{-1}=K+\gamma^{-1}hh^{\top}$. 
    For symmetric positive semidefinite matrices $M$, $\DPPfunction_k(M)$ is Schur concave in the eigenvalues of $M$ (see, e.g., \cite{Schur1923-uber}, Theorem~10 of \cite{Fornace2024-column}).
    Therefore replacing any subset of the eigenvalues of $\tilde{K}_\gamma$ with repeated copies of its average value can only increase the value of  $\DPPfunction_k$.
    It follows that $\DPPfunction_k(\tilde{K}_\gamma) \leq \DPPfunction_k(\Diag(\hat{\lambda}_\gamma))$, where 
    \begin{eqn}
        \hat{\lambda}_\gamma \Eq 
        [\hat{\lambda}, \gamma^{-1}] :=
        \Bigg[ \underbrace{\frac{t}{r}, \dots, \frac{t}{r}}_{r \text{ times}},  \, 
        \underbrace{\frac{T-t}{n-r-1}, \dots, \frac{T-t}{n-r-1}}_{n-r-1 \text{ times}}, \, \gamma^\Mo \Bigg]
    \end{eqn}
    with $t := \Tr^{(r)} [K]$  and $T:= \Tr[K]$.

    Now by \Cref{eq:dppclassical}, we know that 
    \[
    \mathcal{E}_{k}^{\mathrm{DPP}}(\mathrm{diag}(\hat{\lambda}_{\gamma}))=(k+1)\frac{e_{k+1}(\hat{\lambda}\cup\{\gamma^{-1}\})}{e_{k}(\hat{\lambda}\cup\{\gamma^{-1}\})},
    \]
     and in the small $\gamma$ limit, only products including $\gamma^{-1}$
    dominate the symmetric polynomials, so 
    \begin{align*}
    \mathcal{E}_{k}^{\mathrm{DPP}}(\mathrm{diag}(\hat{\lambda}_{\gamma})) & =(k+1)\frac{e_{k}(\hat{\lambda})}{e_{k-1}(\hat{\lambda})} + O(\gamma) \\
     & =\frac{k+1}{k}\mathcal{E}_{k-1}^{\mathrm{DPP}}(\mathrm{diag}(\hat{\lambda})) + O(\gamma) \\
     & \leq\frac{k+1}{k}\left(1 + \frac{r}{k-r}\right)(T-t) + O(\gamma) \\ 
     & = \left( 1 + \frac{r+1}{k-r} \right) (T-t) + O(\gamma).
    \end{align*}
    where in the second equation we have used \Cref{eq:dppclassical}, in the third line we have used \Cref{eq:dpp-nystrom}, and the last line follows by elementary manipulations. In summary, we have established that 
    \begin{eqn}\label{eq:dppwah}
        \DPPfunction_k(\tilde{K}_\gamma) \leq   \left( 1 + \frac{r+1}{k-r} \right) (T-t) + O(\gamma).
    \end{eqn}

    Now the proof technique of \Cref{th:markov-norm-limits} establishes that in the limit of $\gamma \ra 0$, the quantity $\ve_{*}^{\gamma}(\I)$
(cf. \Cref{def:markov-epsilon}) agrees with 
\[
\tilde{\ve}_{*}^{\gamma}(\I):=\Tr\left[\tilde{K}_{\gamma}-(\tilde{K}_{\gamma})_{:,\I}(\tilde{K}_{\gamma})_{\I,\I}^{-1}(\tilde{K}_{\gamma})_{\I,:}\right],
\]
 as well as the limit $\ve_{*}(\I)$, both up to error $O(\gamma)$. 

    For every $\gamma > 0$, there must exist a $k$-subset $\I^\gamma$ that performs at least as well as the DPP expectation, i.e., satisfying $ \tilde{\ve}_{*}^{\gamma}(\I^\gamma) \leq \DPPfunction_k(\tilde{K}_\gamma)$. But also $\tilde{\ve}_{*}^{\gamma}(\I^\gamma) = {\ve}_{*}(\I^\gamma) + O(\gamma)$, so it follows from \Cref{eq:dppwah} that 
    \begin{eqn}
        {\ve}_{*}(\I^\gamma) \leq \left( 1 + \frac{r+1}{k-r} \right) (T-t) + O(\gamma).
    \end{eqn}
    Since there are only finitely many subsets, it follows that there exists some $k$-subset $\I$ such that $\tilde{\ve}_{*} (\I) \leq \left( 1 + \frac{r+1}{k-r} \right) (T-t)$, which completes the proof.
\end{proof}

Now we can restate and prove Theorem \hyperlink{Theorem4B}{4B}:
\ThSpectralCSSP*
\begin{Proof}{th:spectral-cssp}
The proof follows directly from Theorem \hyperlink{Theorem4A}{4A}, together with \Cref{lem:dppstuff}.
\end{Proof}
\noindent While the choice of $s$ in Theorem \hyperlink{Theorem4B}{4B}  may be optimized to give an asymptotic $(r, \epsilon)$ bound in the sense of \cite{Chen2022-randomly,Derezinski2021-determinantal} for nuclear maximization (following the strategy in \cite{Fornace2024-column}), we omit such derivations for simplicity's sake.

\end{document}

%% file: math.tex
\usepackage[parfill]{parskip}
\usepackage[utf8]{inputenc}    
\usepackage{amsmath,amssymb,amsfonts}
\usepackage{graphicx}
\usepackage{hyperref}
\usepackage{etoolbox}
\usepackage{titling}
\usepackage[normalem]{ulem}
\usepackage{algpseudocodex}
\usepackage{algorithm}
\usepackage{mathtools}
\usepackage{scalerel}
\usepackage{derivative}
\usepackage{abstract}
\usepackage{braket}
\usepackage{xfrac}

\newcommand{\Nystrom}{Nystr\"om}

\DeclareMathOperator{\p}{\pi}

\DeclareMathOperator{\Diag}{diag}
\newcommand{\DiagI}{\Diag^\Mo}

\DeclareMathOperator{\I}{\mathcal{I}}

\renewcommand{\t}{\top}

\newcommand{\Id}{\mathbf{I}}

\newcommand{\Ex}{\mathbb{E}}

\newcommand{\Tr}{\mathrm{Tr}}

\newcommand{\lrp}[1]{{\left( #1 \right) }}
\newcommand{\lrs}[1]{{\left[ #1 \right] }}
\newcommand{\lrb}[1]{{\left \{ #1 \right \} }}

\newcommand{\One}{\mathbf{1}}
\newcommand{\Zero}{\mathbf{0}}

\newcommand{\Eq}{\coloneqq}

\newcommand{\ve}{\varepsilon}

\newcommand{\Oh}{{1/2}}

\newenvironment{eqn}{\begin{equation}\begin{aligned}}{\end{aligned}\end{equation}\ignorespacesafterend}

\DeclareMathOperator*{\argmax}{arg\,max}
\DeclareMathOperator*{\argmin}{arg\,min}

\DeclarePairedDelimiterX{\norm}[1]{\lVert}{\rVert}{#1}
\DeclarePairedDelimiter\Abs{\lvert}{\rvert}%

\newcommand{\Minus}{-}

\newcommand{\Mo}{{\Minus 1}}

\newcommand{\Mh}{{\Minus 1/2}}

\usepackage{amsthm}

\usepackage{framed}
\usepackage{amsthm,thmtools}
\renewenvironment{leftbar}{%
  \AtBeginEnvironment{itemize}{\parskip=0pt\parsep=0pt\partopsep=0pt}
  \MakeFramed {\advance\hsize-\width \FrameRestore}}%
 {\endMakeFramed}

\declaretheoremstyle[headfont=\bfseries,%
spaceabove=2pt,spacebelow=2pt,%
notefont=\bfseries,%
notebraces={}{},%
headpunct=,%
bodyfont=,%
headformat=\color{black}\NAME~\NUMBER:~{\NOTE}\hfill\smallskip\linebreak,%
preheadhook=\begin{leftbar},%
postfoothook=\end{leftbar},%
]{customDefinition}

\declaretheoremstyle[headfont=\bfseries,%
spaceabove=2pt,spacebelow=2pt,%
notefont=\bfseries,%
notebraces={}{},%
headpunct=,%
bodyfont=,%
headformat=\color{black}\NAME \NUMBER:~{\NOTE}\hfill\smallskip\linebreak,%
preheadhook=\begin{leftbar},%
postfoothook=\end{leftbar},%
]{customDefinition2}

\newbool{PROVING}
\newcommand{\prooflink}[1]{\ifbool{PROVING}{from \Cref{#1:section}}{proof in \Cref{#1:proof}}}
\boolfalse{PROVING}

\makeatletter
\def\th@plain{%
  \thm@notefont{}
  \itshape 
}
\def\th@definition{%
  \thm@notefont{}
  \normalfont 
}
\makeatother

\usepackage{thmtools,thm-restate}
\usepackage[unq]{unique}

\declaretheorem[style=customDefinition]{definition}

\declaretheorem[style=customDefinition, numberwithin=section]{lemma}
\declaretheorem[style=customDefinition, numberwithin=section, sibling=lemma]{proposition}
\declaretheorem[style=customDefinition, numberwithin=section, sibling=lemma]{corollary}

\declaretheorem[style=customDefinition2, numbered=no]{Theorem 1A}
\declaretheorem[style=customDefinition2, numbered=no]{Theorem 1B}
\declaretheorem[style=customDefinition2, numbered=no]{Theorem 3*}
\declaretheorem[style=customDefinition2, numbered=no]{Theorem 4A}
\declaretheorem[style=customDefinition2, numbered=no]{Theorem 4B}

\newenvironment{Proof}[1]{\begin{proof} \label{#1:proof}}{\end{proof}}

\newbool{STATING}
\newbool{NOLINK}
\boolfalse{STATING}
\boolfalse{NOLINK}

\newcommand{\ProofLink}[1]{\ifbool{NOLINK}{}{\normalfont \ifbool{STATING}{(proof in \Cref{#1:proof})}{(from \Cref{#1:section})}}}

\boolfalse{STATING}

\makeatletter
\def\th@plain{%
  \thm@notefont{}
  \itshape 
}
\def\th@definition{%
  \thm@notefont{}
  \normalfont 
}
\makeatother

\makeatletter
\g@addto@macro\normalsize{%
  \setlength\abovedisplayskip{7pt}
  \setlength\belowdisplayskip{7pt}
  \setlength\abovedisplayshortskip{2pt}
  \setlength\belowdisplayshortskip{2pt}
}
\makeatother

%% file: common.tex
\usepackage[parfill]{parskip}
\usepackage[utf8]{inputenc}    
\usepackage{amsmath,amssymb,amsfonts}
\usepackage{graphicx}
\usepackage{hyperref}
\usepackage{etoolbox}
\usepackage{titling}
\usepackage[normalem]{ulem}
\usepackage{algpseudocodex}
\usepackage{algorithm}
\usepackage{blindtext}
\usepackage{mathtools}
\usepackage{todonotes}
\usepackage{mdframed}
\usepackage{scalerel}
\usepackage{derivative}
\usepackage{pdfpages}
\usepackage{setspace}
\usepackage{wrapfig}
\usepackage{microtype}
\usepackage{booktabs}
\usepackage[margin=1cm]{caption}

\usepackage{enumitem}
\setlist{nosep}

\usepackage{abstract}

\hypersetup{
    colorlinks,
    citecolor=blue,
    filecolor=red,
    linkcolor=black,
    urlcolor=blue
}

\usepackage[capitalise]{cleveref}
\crefname{algorithm}{Algorithm}{Algorithms}
\crefname{figure}{Figure}{Figures}
\crefname{section}{Section}{Sections}
\crefname{appendix}{Appendix}{Appendices}
\crefformat{equation}{(#2#1#3)}
\crefrangeformat{equation}{(#3#1#4) to~(#5#2#6)}
\crefmultiformat{equation}{(#2#1#3)}%
{ and~(#2#1#3)}{, (#2#1#3)}{ and~(#2#1#3)}

%% file: symbols.tex
\newcommand{\Opt}{\mathcal{O}}
\newcommand{\Rand}{\zeta}
\newcommand{\Orand}{(1+\zeta)}

\newcommand{\Comp}[1]{\overline{#1}}

\newcommand{\h}{h}

\DeclareMathOperator{\R}{\mathbb{R}}

\DeclareMathOperator{\Prob}{\mathbb{P}}

\newcommand{\ra}{\rightarrow}
\newcommand{\Indices}{[n]}
\newcommand{\Ia}{\Abs{\I}}

\DeclareMathOperator{\E}{\mathbb{E}}

\newcommand{\pp}[1]{\left( #1 \right)}

\let\p\relax
\DeclareMathOperator{\p}{\pi}

  {\end{mdframed}}

\newcommand{\Expt}[1]{\exp \lrp{\Minus #1 t}}
\newcommand{\expt}[1]{e^{\Minus #1 t}}

\newcommand{\Ert}[1]{\exp \lrp{#1 t}}
\newcommand{\ert}[1]{e^{#1 t}}

\newcommand{\Orth}{\mathrm{orth}}

\newcommand{\im}{\mathrm{Im}}
\newcommand{\re}{\mathrm{Re}}

\DeclarePairedDelimiterX{\snorm}[1]{\lVert}{\rVert_{2}}{#1}
\DeclarePairedDelimiterX{\nnorm}[1]{\lVert}{\rVert_{*}}{#1}
\DeclarePairedDelimiterX{\Norms}[1]{\lVert}{\rVert_{\{2,*\}}}{#1}

\newcommand{\En}{\varepsilon_*}
\newcommand{\Es}{\varepsilon_2}

\newcommand{\Nus}{\nu_{\{2,*\}}}
\newcommand{\Eps}{\varepsilon_{\{2,*\}}}

\newcommand{\Ob}{\Psi}
\newcommand{\ObAsym}{\UnSym{\Psi}}
\newcommand{\On}{\psi_*}
\newcommand{\Os}{\psi_2}
\newcommand{\Obs}{\psi_{\{2,*\}}}

\newcommand{\Contour}{\mathcal{C}}

\newcommand{\Kill}{\gamma}
\newcommand{\La}{L_\Kill}

\newcommand{\Ka}{K_\Kill}

\newcommand{\KA}{(\Ka)}

\newcommand{\Ca}{\Co_\Kill}
\newcommand{\Va}{V_\Kill}
\newcommand{\ha}{\hr_\Kill}
\newcommand{\pa}{\pr_\Kill}

\newcommand{\Kt}{\tilde{K}}
\newcommand{\Kti}{\tilde{K}_{\I,\I}^\Mo}

\newcommand{\Shift}{\beta}

\newcommand{\Kb}{K_\Shift}

\newcommand{\InducedChain}{\mathcal{M}_{\I}}
\newcommand{\FullChain}{\mathcal{M}}
\newcommand{\MarkedChain}{\mathcal{M}_{\I}^{\circ}}

\newcommand{\Le}{\mathring{L}} 
\newcommand{\Lea}{\mathring{L}_\Kill} 
\newcommand{\Rex}{\mathring{R}} 
\newcommand{\Ke}{\mathring{K}} 
\newcommand{\he}{\mathring{h}} 
\newcommand{\pe}{\mathring{\pi}} 
\newcommand{\hea}{\mathring{h}_\Kill} 
\newcommand{\pea}{\mathring{\pi}_\Kill} 

\newcommand{\Ic}{\Comp{\I}}
\newcommand{\Iminus}[1]{\I \setminus \{#1\}}

\newcommand{\hr}{\hat{h}}
\newcommand{\Lr}{\hat{L}}
\newcommand{\Lrk}{\hat{L}_\Kill}
\newcommand{\Kr}{\hat{K}}
\newcommand{\Sr}{\hat{S}}
\newcommand{\pr}{\hat{\pi}}
\newcommand{\Rr}{\hat{R}}

\newcommand{\TimeScale}{\omega}

\newcommand{\Ct}{\Cu^\sharp}

\newcommand{\Sym}[1]{#1}
\newcommand{\UnSym}[1]{\tilde{#1}}

\newcommand{\Su}{\UnSym{W}} 
\newcommand{\So}{\Sym{W}} 
\newcommand{\Sa}{\So_\Kill} 

\newcommand{\Qu}{\UnSym{Q}} 
\newcommand{\Qo}{\Sym{Q}} 
\newcommand{\Qa}{\Qo_\Kill} 

\newcommand{\ObliqueRes}{\mathcal{S}_{t}}
\newcommand{\OrthRes}{\mathcal{R}_{t}}
\newcommand{\FullRes}{\mathcal{Q}_{t}}

\newcommand{\Cu}{\UnSym{C}}
\newcommand{\Co}{\Sym{C}}

\newcommand{\OrthP}{\Sym{P}}
\newcommand{\NonSymOrthP}{\UnSym{P}}

\newcommand{\ObliqueP}{\Sym{P}^\mathrm{sp}}
\newcommand{\NonSymObliqueP}{\UnSym{P}^\mathrm{sp}}

\newcommand{\SymP}{\Sym{P}}
\newcommand{\FullP}{\UnSym{P}}

\newcommand{\SymQ}{\Sym{Q}}
\newcommand{\FullQ}{\UnSym{Q}}

\newcommand{\SymPhi}{\Sym{\phi}}
\newcommand{\UnSymPhi}{\UnSym{\phi}}

\newcommand{\Ki}{K_{\I,\I}^\Mo}
\newcommand{\Kn}{K_{:,\I} K_{\I,\I}^\Mo K_{\I,:}}
\newcommand{\hi}{h_{\I}}

\newcommand{\NoninterpolativeP}{\tilde{P}}

\newcommand{\Process}[1]{\{#1_t\}_{t \in \mathbb{R}^+}}
\newcommand{\XI}{Z}
\newcommand{\XM}{Y}

%% file: auto-generated.bib
@book{Higham_2008, title={Functions of Matrices: Theory and Computation}, ISBN={9780898716467}, url={http://epubs.siam.org/doi/book/10.1137/1.9780898717778}, DOI={10.1137/1.9780898717778}, publisher={Society for Industrial and Applied Mathematics}, author={Higham, Nicholas J.}, year={2008}, month=jan, language={en} }

@book{LevinPeresWilmer2006,
  AUTHOR = {Levin, David A. and Peres, Yuval and Wilmer, Elizabeth L.},
  PUBLISHER = {American Mathematical Society},
  DATE = {2009},
  TITLE = {Markov Chains and Mixing Times},
}

@BOOK{Ross1983,
  AUTHOR = {Ross, Sheldon M.},
  PUBLISHER = {Wiley},
  DATE = {1983},
  TITLE = {Stochastic processes},
}

@misc{Fornace2024-column,
      title={Column and row subset selection using nuclear scores: algorithms and theory for Nystr\"{o}m approximation, CUR decomposition, and graph Laplacian reduction}, 
      author={Mark Fornace and Michael Lindsey},
      year={2024},
      eprint={2407.01698},
      archivePrefix={arXiv},
      primaryClass={math.NA},
      url={https://arxiv.org/abs/2407.01698}, 
}

@ARTICLE{Fornace2020-unified,
  AUTHOR = {Fornace, Mark E and Porubsky, Nicholas J and Pierce, Niles A},
  DATE = {2020},
  ISSUE = {10},
  JOURNALTITLE = {ACS Synth. Biol.},
  NOTE = {PMID: 32910644},
  PAGES = {2665--2678},
  TITLE = {A Unified Dynamic Programming Framework for the Analysis of Interacting Nucleic Acid Strands: Enhanced Models, Scalability, and Speed},
  VOLUME = {9},
}

@INPROCEEDINGS{Schaeffer2015-stochastic,
  AUTHOR = {Schaeffer, Joseph M and Thachuk, Chris and Winfree, Erik},
  INSTITUTION = {Springer},
  BOOKTITLE = {International Workshop on DNA-Based Computers},
  DATE = {2015},
  PAGES = {194--211},
  TITLE = {Stochastic simulation of the kinetics of multiple interacting nucleic acid strands},
}

@ARTICLE{Flamm2000-rna,
  AUTHOR = {Flamm, Christoph and Fontana, Walter and Hofacker, Ivo L and Schuster, Peter},
  PUBLISHER = {Cambridge University Press},
  DATE = {2000},
  ISSUE = {3},
  JOURNALTITLE = {RNA},
  PAGES = {325--338},
  TITLE = {{RNA} folding at elementary step resolution},
  VOLUME = {6},
}

@BOOK{Aldous1995-reversible,
  AUTHOR = {Aldous, David and Fill, James},
  PUBLISHER = {Berkeley},
  DATE = {1995},
  TITLE = {Reversible {M}arkov chains and random walks on graphs},
}

@ARTICLE{Chung2000-discrete,
  AUTHOR = {Chung, Fan and Yau, S-T},
  LANGUAGE = {en},
  PUBLISHER = {Elsevier BV},
  DATE = {2000-07},
  ISSUE = {1-2},
  JOURNALTITLE = {J. Comb. Theory Ser. A.},
  PAGES = {191--214},
  TITLE = {Discrete {Green's} functions},
  VOLUME = {91},
}

@BOOK{Kemeny1976-markov,
  AUTHOR = {Kemeny, John G and Snell, J Laurie},
  PUBLISHER = {Springer-Verlag, New York},
  DATE = {1976},
  TITLE = {{M}arkov Chains},
}

@BOOK{Doyle1984-random,
  AUTHOR = {Doyle, Peter G and Snell, J Laurie},
  PUBLISHER = {American Mathematical Soc.},
  DATE = {1984},
  TITLE = {Random walks and electric networks},
  VOLUME = {22},
}

@ARTICLE{Slyusar1999-family,
  AUTHOR = {Slyusar, V I},
  DATE = {1999},
  ISSUE = {3},
  JOURNALTITLE = {Cybernetics and systems analysis},
  PAGES = {379--384},
  TITLE = {A family of face products of matrices and its properties},
  VOLUME = {35},
}

@ARTICLE{Meyer2021-hutch++,
  AUTHOR = {Meyer, Raphael A and Musco, Cameron and Musco, Christopher and Woodruff, David P},
  LANGUAGE = {en},
  DATE = {2021-01},
  JOURNALTITLE = {Proc SIAM Symp Simplicity Algorithms},
  PAGES = {142--155},
  TITLE = {Hutch++: Optimal Stochastic Trace Estimation},
  VOLUME = {2021},
}

@ARTICLE{Chen2020-rchol,
author = {Chen, Chao and Liang, Tianyu and Biros, George},
title = {RCHOL: Randomized Cholesky Factorization for Solving SDD Linear Systems},
journal = {SIAM Journal on Scientific Computing},
volume = {43},
number = {6},
pages = {C411-C438},
year = {2021},
doi = {10.1137/20M1380624},
}

@ARTICLE{Kyng2016-approximate,
  AUTHOR = {Kyng, Rasmus and Sachdeva, Sushant},
  DATE = {2016-05-08},
  JOURNALTITLE = {FOCS},
  PAGES = {573--582},
  TITLE = {Approximate {Gaussian} elimination for {Laplacians} - fast, sparse, and simple},
}

@ARTICLE{Gao2023-robust,
      title={Robust and Practical Solution of Laplacian Equations by Approximate Elimination}, 
      author={Yuan Gao and Rasmus Kyng and Daniel A. Spielman},
      year={2023},
      eprint={2303.00709},
      archivePrefix={arXiv},
      primaryClass={math.NA},
      url={https://arxiv.org/abs/2303.00709}, 
}

@ARTICLE{Chen2022-randomly,
author = {Chen, Yifan and Epperly, Ethan N. and Tropp, Joel A. and Webber, Robert J.},
title = {Randomly pivoted Cholesky: Practical approximation of a kernel matrix with few entry evaluations},
journal = {Communications on Pure and Applied Mathematics},
volume = {78},
number = {5},
pages = {995-1041},
doi = {https://doi.org/10.1002/cpa.22234},
year = {2025}
}

@ARTICLE{Shitov2021-column,
  AUTHOR = {Shitov, Yaroslav},
  LANGUAGE = {en},
  PUBLISHER = {Elsevier BV},
  DATE = {2021-02},
  JOURNALTITLE = {Linear Algebra Appl.},
  PAGES = {52--58},
  TITLE = {Column subset selection is {NP}-complete},
  VOLUME = {610},
}

@ARTICLE{Fornace2022-nupack,
  AUTHOR = {Fornace, Mark E and Huang, Jining and Newman, Cody T and Porubsky, Nicholas J and Pierce, Marshall B and Pierce, Niles A},
  LANGUAGE = {en},
  PUBLISHER = {chemrxiv.org},
  DATE = {2022-11-11},
  JOURNALTITLE = {ChemRxiv},
  TITLE = {{NUPACK}: Analysis and Design of Nucleic Acid Structures, Devices, and Systems},
  URLDATE = {2024-04-22},
}

@ARTICLE{Rozemberczki2021-multi-scale,
  AUTHOR = {Rozemberczki, Benedek and Allen, Carl and Sarkar, Rik},
  LANGUAGE = {en},
  PUBLISHER = {Oxford University Press (OUP)},
  DATE = {2021-05-05},
  ISSUE = {2},
  JOURNALTITLE = {J. Complex Netw.},
  PAGES = {cnab014},
  TITLE = {Multi-Scale attributed node embedding},
  VOLUME = {9},
}

@ONLINE{Leskovec2014-snap,
  AUTHOR = {Leskovec, Jure and Krevl, Andrej},
  DATE = {2014},
  TITLE = {{SNAP} datasets: Stanford large network dataset collection},
  URLDATE = {2025-06-24},
}

@ARTICLE{Sanderson2019-practical,
  AUTHOR = {Sanderson, Conrad and Curtin, Ryan},
  LANGUAGE = {en},
  PUBLISHER = {Multidisciplinary Digital Publishing Institute},
  DATE = {2019-07-19},
  ISSUE = {3},
  JOURNALTITLE = {Math. Comput. Appl.},
  PAGES = {70},
  TITLE = {Practical Sparse Matrices in {C++} with Hybrid Storage and Template-Based Expression Optimisation},
  URLDATE = {2024-06-14},
  VOLUME = {24},
}

@ARTICLE{Sanderson2016-armadillo,
  AUTHOR = {Sanderson, Conrad and Curtin, Ryan},
  DATE = {2016},
  ISSUE = {2},
  JOURNALTITLE = {Journal of Open Source Software},
  PAGES = {26},
  TITLE = {Armadillo: a template-based {C++} library for linear algebra},
  VOLUME = {1},
}

@SOFTWARE{Fornace2024-nuclear-score-maximization,
  AUTHOR = {Fornace, Mark},
  DATE = {2024},
  TITLE = {nuclear-score-maximization},
  TYPE = {software},
  VERSION = {1.0},
}

@BOOK{Durrett2019-probability,
  AUTHOR = {Durrett, Rick},
  LANGUAGE = {en},
  PUBLISHER = {Cambridge University Press},
  DATE = {2019-04-18},
  PAGETOTAL = {433},
  TITLE = {Probability: Theory and Examples},
}

@INBOOK{Guruswami2012-optimal,
  AUTHOR = {Guruswami, Venkatesan and Sinop, Ali Kemal},
  PUBLISHER = {Society for Industrial and Applied Mathematics},
  BOOKTITLE = {Proceedings of the 2012 Annual ACM-SIAM Symposium on Discrete Algorithms (SODA)},
  DATE = {2012-01-17},
  PAGES = {1207--1214},
  SERIES = {Proceedings},
  TITLE = {Optimal Column-Based Low-Rank Matrix Reconstruction},
}

@ARTICLE{Belabbas2009-spectral,
  AUTHOR = {Belabbas, Mohamed-Ali and Wolfe, Patrick J},
  LANGUAGE = {en},
  DATE = {2009-01-13},
  ISSUE = {2},
  JOURNALTITLE = {Proc. Natl. Acad. Sci. U. S. A.},
  PAGES = {369--374},
  TITLE = {Spectral methods in machine learning and new strategies for very large datasets},
  VOLUME = {106},
}

@ARTICLE{Schur1923-uber,
  AUTHOR = {Schur, Issai},
  DATE = {1923},
  ISSUE = {9-20},
  JOURNALTITLE = {Sitzungsberichte der Berliner Mathematischen Gesellschaft},
  PAGES = {51},
  TITLE = {Uber eine Klasse von Mittelbildungen mit Anwendungen auf die Determinantentheorie},
  VOLUME = {22},
}

@BOOK{Rogers2000-diffusions,
  AUTHOR = {Rogers, L Chris G and Williams, David},
  PUBLISHER = {Cambridge university press},
  DATE = {2000},
  TITLE = {Diffusions, {M}arkov processes, and martingales: Volume 1, foundations},
  VOLUME = {1},
}

@BOOK{Frenkel2023-understanding,
  AUTHOR = {Frenkel, Daan and Smit, Berend},
  LANGUAGE = {en},
  PUBLISHER = {Elsevier},
  DATE = {2023-07-13},
  PAGETOTAL = {679},
  TITLE = {Understanding Molecular Simulation: From Algorithms to Applications},
}

@ARTICLE{Szekely2014-stochastic,
  AUTHOR = {Székely, Jr, Tamás and Burrage, Kevin},
  LANGUAGE = {en},
  PUBLISHER = {Elsevier},
  DATE = {2014-11},
  ISSUE = {20-21},
  JOURNALTITLE = {Comput. Struct. Biotechnol. J.},
  PAGES = {14--25},
  TITLE = {Stochastic simulation in systems biology},
  VOLUME = {12},
}

@ARTICLE{Durr2011-lattice,
  AUTHOR = {Dürr, S and Fodor, Z and Hoelbling, C and Katz, S D and Krieg, S and Kurth, T and Lellouch, L and Lippert, T and Szabó, K K and Vulvert, G and {Budapest-Marseille-Wuppertal collaboration}},
  PUBLISHER = {Springer},
  DATE = {2011-08-30},
  ISSUE = {8},
  JOURNALTITLE = {J. High Energy Phys.},
  PAGES = {148},
  TITLE = {Lattice {QCD} at the physical point: simulation and analysis details},
  VOLUME = {2011},
}

@ARTICLE{Pande2010-everything,
  AUTHOR = {Pande, Vijay S and Beauchamp, Kyle and Bowman, Gregory R},
  LANGUAGE = {en},
  DATE = {2010-09},
  ISSUE = {1},
  JOURNALTITLE = {Methods},
  PAGES = {99--105},
  TITLE = {Everything you wanted to know about Markov State Models but were afraid to ask},
  VOLUME = {52},
}

@ARTICLE{Husic2018-markov,
  AUTHOR = {Husic, Brooke E and Pande, Vijay S},
  LANGUAGE = {en},
  PUBLISHER = {ACS Publications},
  DATE = {2018-02-21},
  ISSUE = {7},
  JOURNALTITLE = {J. Am. Chem. Soc.},
  PAGES = {2386--2396},
  TITLE = {Markov State Models: From an Art to a Science},
  VOLUME = {140},
}

@ARTICLE{Park2006-validation,
  AUTHOR = {Park, Sanghyun and Pande, Vijay S},
  LANGUAGE = {en},
  DATE = {2006-02-07},
  ISSUE = {5},
  JOURNALTITLE = {J. Chem. Phys.},
  PAGES = {054118},
  TITLE = {Validation of Markov state models using Shannon's entropy},
  VOLUME = {124},
}

@INBOOK{Junghare2023-markov,
  AUTHOR = {Junghare, Vivek and Bhattacharya, Sourya and Ansari, Khalid and Hazra, Saugata},
  EDITOR = {Saudagar, Prakash and Tripathi, Timir},
  LOCATION = {Singapore},
  PUBLISHER = {Springer Nature Singapore},
  BOOKTITLE = {Protein Folding Dynamics and Stability: Experimental and Computational Methods},
  DATE = {2023},
  PAGES = {147--164},
  TITLE = {Markov State Models of Molecular Simulations to Study Protein Folding and Dynamics},
}

@ARTICLE{Chodera2014-markov,
  AUTHOR = {Chodera, John D and Noé, Frank},
  LANGUAGE = {en},
  PUBLISHER = {Elsevier},
  DATE = {2014-04},
  JOURNALTITLE = {Curr. Opin. Struct. Biol.},
  PAGES = {135--144},
  TITLE = {Markov state models of biomolecular conformational dynamics},
  VOLUME = {25},
}

@ARTICLE{Shukla2015-markov,
  AUTHOR = {Shukla, Diwakar and Hernández, Carlos X and Weber, Jeffrey K and Pande, Vijay S},
  LANGUAGE = {en},
  PUBLISHER = {ACS Publications},
  DATE = {2015-02-17},
  ISSUE = {2},
  JOURNALTITLE = {Acc. Chem. Res.},
  PAGES = {414--422},
  TITLE = {Markov state models provide insights into dynamic modulation of protein function},
  VOLUME = {48},
}

@ARTICLE{Schwantes2013-improvements,
  AUTHOR = {Schwantes, Christian R and Pande, Vijay S},
  LANGUAGE = {en},
  PUBLISHER = {ACS Publications},
  DATE = {2013-04-09},
  ISSUE = {4},
  JOURNALTITLE = {J. Chem. Theory Comput.},
  PAGES = {2000--2009},
  TITLE = {Improvements in Markov State Model Construction Reveal Many Non-Native Interactions in the Folding of {NTL9}},
  VOLUME = {9},
}

@ARTICLE{McGibbon2015-variational,
  AUTHOR = {McGibbon, Robert T and Pande, Vijay S},
  LANGUAGE = {en},
  PUBLISHER = {pubs.aip.org},
  DATE = {2015-03-28},
  ISSUE = {12},
  JOURNALTITLE = {J. Chem. Phys.},
  PAGES = {124105},
  TITLE = {Variational cross-validation of slow dynamical modes in molecular kinetics},
  VOLUME = {142},
}

@ARTICLE{Chatterjee2015-uncertainty,
  AUTHOR = {Chatterjee, Abhijit and Bhattacharya, Swati},
  LANGUAGE = {en},
  DATE = {2015-09-21},
  ISSUE = {11},
  JOURNALTITLE = {J. Chem. Phys.},
  PAGES = {114109},
  TITLE = {Uncertainty in a Markov state model with missing states and rates: Application to a room temperature kinetic model obtained using high temperature molecular dynamics},
  VOLUME = {143},
}

@ARTICLE{Thiede2019-galerkin,
  AUTHOR = {Thiede, Erik H and Giannakis, Dimitrios and Dinner, Aaron R and Weare, Jonathan},
  LANGUAGE = {en},
  DATE = {2019-06-28},
  ISSUE = {24},
  JOURNALTITLE = {J. Chem. Phys.},
  PAGES = {244111},
  TITLE = {Galerkin approximation of dynamical quantities using trajectory data},
  VOLUME = {150},
}

@INBOOK{Deshpande2006-adaptive,
  AUTHOR = {Deshpande, Amit and Vempala, Santosh},
  LOCATION = {Berlin, Heidelberg},
  PUBLISHER = {Springer Berlin Heidelberg},
  BOOKTITLE = {Approximation, Randomization, and Combinatorial Optimization. Algorithms and Techniques},
  DATE = {2006},
  PAGES = {292--303},
  SERIES = {Lecture notes in computer science},
  TITLE = {Adaptive sampling and fast low-rank matrix approximation},
}

@ARTICLE{Gu2004-strong,
  AUTHOR = {Gu, M and Miranian, L},
  DATE = {2004},
  JOURNALTITLE = {Plan. Perspect.},
  PAGES = {92},
  TITLE = {Strong rank revealing {Cholesky} factorization},
  VOLUME = {76},
}

@ARTICLE{Steinerberger2024-randomly,
  AUTHOR = {Steinerberger, Stefan},
  DATE = {2024-04-17},
  EPRINTCLASS = {math.NA},
  EPRINTTYPE = {arXiv},
  JOURNALTITLE = {arXiv [math.NA]},
  TITLE = {Randomly Pivoted Partial {Cholesky}: Random How?},
}

@ARTICLE{Civril2009-selecting,
  AUTHOR = {Çivril, Ali and Magdon-Ismail, Malik},
  DATE = {2009-11-06},
  ISSUE = {47},
  JOURNALTITLE = {Theor. Comput. Sci.},
  PAGES = {4801--4811},
  TITLE = {On selecting a maximum volume sub-matrix of a matrix and related problems},
  VOLUME = {410},
}

@ARTICLE{Williams2000-using,
  AUTHOR = {Williams, Christopher K I and Seeger, M},
  DATE = {2000},
  JOURNALTITLE = {Adv. Neural Inf. Process. Syst.},
  PAGES = {682--688},
  TITLE = {Using the {Nyström} method to speed up kernel machines},
}

@ARTICLE{Derezinski2021-determinantal,
  AUTHOR = {Derezinski, M and Mahoney, M W},
  DATE = {2021},
  ISSUE = {1},
  JOURNALTITLE = {Notices of the American Mathematical Society},
  PAGES = {34--45},
  TITLE = {Determinantal point processes in randomized numerical linear algebra},
  URLDATE = {2024-01-29},
  VOLUME = {68},
}

@BOOK{Jacobsen2005-point,
  AUTHOR = {Jacobsen, Martin},
  LANGUAGE = {en},
  LOCATION = {Secaucus, NJ},
  PUBLISHER = {Birkhauser Boston},
  DATE = {2005-12-15},
  EDITION = {2006},
  PAGETOTAL = {328},
  SERIES = {Probability and Its Applications},
  TITLE = {Point process theory and applications: Marked point and piecewise deterministic processes},

}

@ARTICLE{Drineas2005-nystrom,
  title        = {On the {Nyström} method for approximating a {Gram} matrix for
                  improved kernel-based learning},
  author       = {Drineas, P and Mahoney, Michael W},
  journaltitle = {J. Mach. Learn. Res.},
  publisher    = {jmlr.org},
  volume       = {6},
  pages        = {2153--2175},
  date         = {2005-06-27}
}

@BOOK{Park2017-fundamentals,
  title     = {Fundamentals of probability and stochastic processes with
               applications to communications},
  author    = {Park, Kun Il},
  publisher = {Springer International Publishing},
  location  = {Cham, Switzerland},
  edition   = {1},
  date      = {2017-12-04},
  pagetotal = {275},
  language  = {en}
}

@ARTICLE{Yao2013-hierarchical,
    author = {Yao, Yuan and Cui, Raymond Z. and Bowman, Gregory R. and Silva, Daniel-Adriano and Sun, Jian and Huang, Xuhui},
    title = {Hierarchical {N}ystr\"{o}m methods for constructing Markov state models for conformational dynamics},
    journal = {The Journal of Chemical Physics},
    volume = {138},
    number = {17},
    pages = {174106},
    year = {2013},
    month = {05},
    issn = {0021-9606},
    doi = {10.1063/1.4802007},
}

@ARTICLE{Lin2018-mathematical,
  title        = {A mathematical theory of optimal milestoning (with a detour
                  via exact milestoning)},
  author       = {Lin, Ling and Lu, Jianfeng and Vanden-Eijnden, Eric},
  journaltitle = {Commun. Pure Appl. Math.},
  publisher    = {Wiley},
  volume       = {71},
  issue        = {6},
  pages        = {1149--1177},
  date         = {2018-06},
  language     = {en}
}

@ARTICLE{Vanden-Eijnden2008-assumptions,
  title        = {On the assumptions underlying milestoning},
  author       = {Vanden-Eijnden, Eric and Venturoli, Maddalena and Ciccotti,
                  Giovanni and Elber, Ron},
  journaltitle = {J. Chem. Phys.},
  publisher    = {AIP Publishing},
  volume       = {129},
  issue        = {17},
  pages        = {174102},
  date         = {2008-11-07},
  language     = {en}
}

@ARTICLE{Elber2021-modeling,
  title        = {Modeling molecular kinetics with Milestoning},
  author       = {Elber, Ron and Fathizadeh, Arman and Ma, Piao and Wang, Hao},
  journaltitle = {Wiley Interdiscip. Rev. Comput. Mol. Sci.},
  publisher    = {Wiley},
  volume       = {11},
  issue        = {4},
  date         = {2021-07},
  language     = {en}
}

@INPROCEEDINGS{Deshpande2010-efficient,
  title      = {Efficient volume sampling for row/column subset selection},
  author     = {Deshpande, Amit and Rademacher, Luis},
  booktitle  = {2010 IEEE 51st Annual Symposium on Foundations of Computer
                Science},
  publisher  = {IEEE},
  eventtitle = {2010 IEEE 51st Annual Symposium on Foundations of Computer
                Science (FOCS)},
  venue      = {Las Vegas, NV, USA},
  pages      = {329--338},
  date       = {2010-10}
}

@INBOOK{Boutsidis2009-improved,
  title     = {An Improved Approximation Algorithm for the Column Subset
               Selection Problem},
  author    = {Boutsidis, Christos and Mahoney, Michael W and Drineas, Petros},
  booktitle = {Proceedings of the 2009 Annual ACM-SIAM Symposium on Discrete
               Algorithms (SODA)},
  publisher = {Society for Industrial and Applied Mathematics},
  pages     = {968--977},
  date      = {2009-01-04},
  series    = {Proceedings}
}

@THESIS{Fornace2022-computational,
  title     = {Computational methods for simulating and parameterizing nucleic
               acid secondary structure thermodynamics and kinetics},
  author    = {Fornace, Mark Evan},
  publisher = {California Institute of Technology},
  date      = {2022}
}

@INBOOK{Valmari2010-simple,
  title     = {Simple {O}(m log n) time Markov chain lumping},
  author    = {Valmari, Antti and Franceschinis, Giuliana},
  booktitle = {Tools and Algorithms for the Construction and Analysis of Systems},
  publisher = {Springer Berlin Heidelberg},
  location  = {Berlin, Heidelberg},
  pages     = {38--52},
  date      = {2010},
  series    = {Lecture notes in computer science}
}

@ARTICLE{Derisavi2003-optimal,
  title        = {Optimal state-space lumping in Markov chains},
  author       = {Derisavi, Salem and Hermanns, Holger and Sanders, William H},
  journaltitle = {Inf. Process. Lett.},
  publisher    = {Elsevier BV},
  volume       = {87},
  issue        = {6},
  pages        = {309--315},
  date         = {2003-09},
  language     = {en}
}

@ARTICLE{Webber2021-error,
  title        = {Error bounds for dynamical spectral estimation},
  author       = {Webber, Robert J and Thiede, Erik H and Dow, Douglas and
                  Dinner, Aaron R and Weare, Jonathan},
  journaltitle = {SIAM Journal on Mathematics of Data Science},
  publisher    = {Society for Industrial and Applied Mathematics},
  volume       = {3},
  issue        = {1},
  pages        = {225--252},
  date         = {2021-01-01}
}

@ARTICLE{Berezhkovskii2019-committors,
  title        = {Committors, first-passage times, fluxes, Markov states,
                  milestones, and all that},
  author       = {Berezhkovskii, A and Szabó, A},
  journaltitle = {J. Chem. Phys.},
  publisher    = {pubs.aip.org},
  volume       = {150},
  issue        = {5},
  pages        = {054106},
  date         = {2019-02-06}
}

@ARTICLE{Kells2020-correlation,
  title        = {Correlation functions, mean first passage times, and the
                  Kemeny constant},
  author       = {Kells, Adam and Koskin, Vladimir and Rosta, Edina and
                  Annibale, Alessia},
  journaltitle = {J. Chem. Phys.},
  publisher    = {AIP Publishing},
  volume       = {152},
  issue        = {10},
  pages        = {104108},
  date         = {2020-03-14},
  language     = {en}
}

@ARTICLE{Lorpaiboon2020-integrated,
  title        = {Integrated variational approach to conformational dynamics: a
                  robust strategy for identifying eigenfunctions of dynamical
                  operators},
  author       = {Lorpaiboon, Chatipat and Thiede, Erik Henning and Webber,
                  Robert J and Weare, Jonathan and Dinner, Aaron R},
  journaltitle = {J. Phys. Chem. B},
  volume       = {124},
  issue        = {42},
  pages        = {9354--9364},
  date         = {2020-10-22},
  language     = {en}
}

@ARTICLE{Faradjian2004-computing,
  title        = {Computing time scales from reaction coordinates by milestoning},
  author       = {Faradjian, Anton and Elber, R},
  journaltitle = {J. Chem. Phys.},
  publisher    = {pubs.aip.org},
  volume       = {120},
  issue        = {23},
  pages        = {10880--10889},
  date         = {2004-06-15}
}

@ARTICLE{Sarich2010-approximation,
  title        = {On the approximation quality of Markov state models},
  author       = {Sarich, Marco and Noé, Frank and Schütte, Christof},
  journaltitle = {Multiscale Model. Simul.},
  publisher    = {Society for Industrial \& Applied Mathematics (SIAM)},
  volume       = {8},
  issue        = {4},
  pages        = {1154--1177},
  date         = {2010-01}
}

@ARTICLE{Hummer2015-optimal,
  title        = {Optimal dimensionality reduction of multistate kinetic and
                  Markov-state models},
  author       = {Hummer, Gerhard and Szabo, Attila},
  journaltitle = {J. Phys. Chem. B},
  publisher    = {American Chemical Society (ACS)},
  volume       = {119},
  issue        = {29},
  pages        = {9029--9037},
  date         = {2015-07-23},
  language     = {en}
}

@INBOOK{Keilson1979-rarity,
  title     = {Rarity and Exponentiality},
  author    = {Keilson, Julian},
  booktitle = {Applied Mathematical Sciences},
  publisher = {Springer New York},
  location  = {New York, NY},
  pages     = {130--163},
  date      = {1979},
  series    = {Applied mathematical sciences}
}

@ARTICLE{Brown1983-approximating,
  title        = {Approximating {IMRL} distributions by exponential
                  distributions, with applications to first passage times},
  author       = {Brown, Mark},
  journaltitle = {Annals of Probability},
  publisher    = {Institute of Mathematical Statistics},
  volume       = {11},
  issue        = {2},
  pages        = {419--427},
  date         = {1983-05-01}
}

@ARTICLE{Mahoney2009-cur,
  title        = {{CUR} matrix decompositions for improved data analysis},
  author       = {Mahoney, Michael W and Drineas, Petros},
  journaltitle = {Proc. Natl. Acad. Sci. U. S. A.},
  publisher    = {National Acad Sciences},
  volume       = {106},
  issue        = {3},
  pages        = {697--702},
  date         = {2009-01-20},
  language     = {en}
}

@ARTICLE{Chaturantabut2010-nonlinear,
  title        = {Nonlinear model reduction via discrete empirical interpolation},
  author       = {Chaturantabut, Saifon and Sorensen, Danny C},
  journaltitle = {SIAM J. Sci. Comput.},
  publisher    = {Society for Industrial \& Applied Mathematics (SIAM)},
  volume       = {32},
  issue        = {5},
  pages        = {2737--2764},
  date         = {2010-01-01}
}

@ARTICLE{Cortinovis2024-adaptive,
  title        = {Adaptive randomized pivoting for column subset selection,
                  {DEIM}, and low-rank approximation},
  author       = {Cortinovis, Alice and Kressner, Daniel},
  journaltitle = {arXiv [math.NA]},
  date         = {2024-12-18},
  eprinttype   = {arXiv},
  eprintclass  = {math.NA}
}

@ARTICLE{Aldous1992-inequalities,
  title        = {Inequalities for rare events in time-reversible Markov chains
                  {I}},
  author       = {Aldous, B D and Brown, Mark},
  journaltitle = {Lecture Notes-Monograph Series},
  publisher    = {JSTOR},
  date         = {1992}
}

@ARTICLE{Zhang2025-accurate,
  title        = {Fast and accurate interpolative decompositions for general,
                  sparse, and structured tensors},
  author       = {Zhang, Yifan and Fornace, Mark and Lindsey, Michael},
  journaltitle = {arXiv [math.NA]},
  date         = {2025-03-24},
  eprinttype   = {arXiv},
  eprintclass  = {math.NA}
}

@ARTICLE{Schutte2011-markov,
  title        = {Markov state models based on milestoning},
  author       = {Schütte, Christof and Noé, Frank and Lu, Jianfeng and Sarich,
                  Marco and Vanden-Eijnden, Eric},
  journaltitle = {J. Chem. Phys.},
  publisher    = {AIP Publishing},
  volume       = {134},
  issue        = {20},
  pages        = {204105},
  date         = {2011-05-28},
  language     = {en}
}

@ARTICLE{Djurdjevac2012-estimating,
  title        = {Estimating the eigenvalue error of Markov state models},
  author       = {Djurdjevac, Natasa and Sarich, Marco and Schütte, Christof},
  journaltitle = {Multiscale Model. Simul.},
  publisher    = {Society for Industrial \& Applied Mathematics (SIAM)},
  volume       = {10},
  issue        = {1},
  pages        = {61--81},
  date         = {2012-01-28},
  urldate      = {2025-08-07},
  language     = {en}
}
